\documentclass[11pt]{amsart}
\usepackage{amscd,amssymb}
\usepackage{amsthm,amsmath,amssymb}
\usepackage[matrix,arrow]{xy}

\theoremstyle{definition}
\newtheorem{example}[equation]{Example}
\newtheorem{definition}[equation]{Definition}
\newtheorem{theorem}[equation]{Theorem}
\newtheorem{lemma}[equation]{Lemma}
\newtheorem{corollary}[equation]{Corollary}

\newtheorem{question}[equation]{Question}
\newtheorem*{question*}{Question}

\newtheorem*{problem*}{Problem}

\theoremstyle{remark}
\newtheorem{remark}[equation]{Remark}

\makeatletter\@addtoreset{equation}{section} \makeatother

\def\Q {\mathbb{Q}}
\def\C {\mathbb{C}}
\def\P {\mathbb{P}}
\def\Z {\mathbb{Z}}
\def\N {\mathbb{Z}_{{}>0}}
\def\F {\mathbb{F}}

\def\SS {\mathrm{S}}

\def\PSLF {\mathrm{PSL}_2(\F_7)}
\def\SLF {\mathrm{SL}_2(\F_7)}
\def\PGL {\mathrm{PGL}}

\def\CKlein {\mathfrak{C}}

\def\ge {\geqslant}

\author{Ivan Cheltsov  and Constantin Shramov}

\title{Three embeddings of the~Klein simple group\\ into the~Cremona group of rank three}%

\thanks{The~authors thank I. Dolgachev, Th. K\"oppe, A. Kuznetsov and Yu. Prokhorov for their help.}%

\thanks{The~first author thanks the~Max-Planck-Institute f\"ur Mathematik at Bonn for hospitality.}%

\thanks{The authors were partially supported by AG Laboratory GU-HSE, RF
government grant 11~11.G34.31.0023. The second author was
partially supported by the grants
RFFI~08-01-00395-a, RFFI~11-01-00185-a, RFFI~11-01-00336-a
and N.Sh.-4713.2010.1.}

\address{University of Edinburgh, Edinburgh EH9 3JZ, UK, \texttt{I.Cheltsov@ed.ac.uk}}

\address{Steklov Institute of Mathematics, Moscow 119991, Russia, \texttt{shramov@mccme.ru}}

\address{Laboratory of Algebraic Geometry, GU-HSE, 7 Vavilova street, Moscow 117312, Russia}%

\begin{document}

\begin{abstract}
We study the action of the~Klein simple group $\PSLF$ consisting
of $168$ elements on two rational threefolds: the
three-dimensional projective space and a smooth Fano threefold $X$
of anticanonical degree $22$ and index~$1$. We show that
the~Cremona group of rank three has at least three non-conjugate
subgroups isomorphic to $\PSLF$. As a by-product, we prove
that~$X$ admits a K\"ahler--Einstein~metric, and we construct a
smooth polarized $K3$ surface of degree~$22$ with an~action of the
group $\PSLF$.
\end{abstract}

\maketitle

Unless explicitly stated otherwise, varieties are assumed to be
projective,~normal~and~complex.

\section{Introduction}
\label{section:intro}

The~Cremona group of rank $n$, usually denoted by
$\mathrm{Cr}_{n}(\mathbb{C})$, is the~group of birational
automorphisms of the~complex projective space $\mathbb{P}^{n}$. It
is well-known that
$\mathrm{Cr}_{1}(\mathbb{C})=\mathrm{Aut}(\mathbb{P}^{1})\cong\PGL_2(\C)$.
For $n\geqslant 2$, the structure of the group
$\mathrm{Cr}_{n}(\mathbb{C})$ is much more complicated than of its
subgroup $\mathrm{Aut}(\mathbb{P}^{n})$. So one possible way to
study the Cremona groups of high rank is by analyzing their finite
subgroups.

Finite subgroups in $\mathrm{Cr}_{1}(\mathbb{C})\cong\PGL_2(\C)$
are cyclic, $\mathbb{Z}_{2}\times\mathbb{Z}_{2}$, dihedral,
$\mathrm{A}_{4}$, and~$\mathrm{A}_5$. Two finite subgroups in
$\mathrm{Cr}_1(\C)$ are conjugate if and only if they~are~isomorphic. Finite
subgroups in $\mathrm{Cr}_{2}(\mathbb{C})$ have been almost
completely classified in \cite{DoIs06}. This is an important and
still active research direction originating in the~works of
Kantor, Bertini, and Wiman.

\begin{example}[{\cite[Theorem~B.2]{Ch09}}]
\label{example:Cremona-2} Let $G$ be a~finite simple non-abelian
subgroup. Then $\mathrm{Cr}_2(\C)$ has a~subgroup isomorphic to
$G$ if and only if $G$ is one of the following groups:
$\mathrm{A}_5$, $\mathrm{PSL}_{2}(\F_7)$ or $\mathrm{A}_6$. The
group $\mathrm{Cr}_2(\C)$ has exactly $3$, $2$ and $1$
non-conjugate subgroups isomorphic to $\mathrm{A}_5$,
$\mathrm{PSL}_{2}(\F_7)$ and~$\mathrm{A}_6$, respectively.
\end{example}

Much less is known about finite subgroups in
$\mathrm{Cr}_3(\mathbb{C})$.  In
fact, they were so
poorly understood until recently that Serre asked the following

\begin{question}[{\cite[Question~6.0]{Serre}}]
\label{question:Serre} Does there exist a finite group which is
not embeddable in $\mathrm{Cr}_3(\mathbb{C})$?
\end{question}

Inspired by Question~\ref{question:Serre} and using methods of
modern three-dimensional birational geometry, Prokhorov proved the
following

\begin{theorem}[{\cite[Theorem~1.3]{Pr09}}]
\label{theorem:Yura-Cremona-1} Suppose that $G$ is a~finite simple
non-abelian subgroup.~Then $\mathrm{Cr}_3(\C)$ has a~subgroup
isomorphic to $G$ if and only if $G$ is one of the following
groups: $\mathrm{A}_5$, $\mathrm{PSL}_{2}(\F_7)$, $\mathrm{A}_6$,
$\mathrm{A}_7$, $\mathrm{PSL}_{2}(\F_8)$ or $\mathrm{PSU}_{4}(\F_2)$.
\end{theorem}

The technique introduced in \cite{Pr09} allows one to handle
finite subgroups in  $\mathrm{Cr}_3(\mathbb{C})$ in a similar way
to ones in $\mathrm{Cr}_2(\mathbb{C})$. Moreover, the proof of
Theorem~\ref{theorem:Yura-Cremona-1} gives us much more than just
a classification. For instance, one can easily use this proof to
obtain the following

\begin{theorem}[{\cite[Corollary~1.23]{ChSh09b}, \cite{Bea11}}]
\label{theorem:Yura-Cremona-2} Up to conjugation
the group $\mathrm{Cr}_3(\C)$
contains exactly one subgroup isomorphic to $\mathrm{PSL}_{2}(\F_8)$,
exactly one subgroup isomorphic to~$\mathrm{A}_7$ and
exactly two subgroups isomorphic to
$\mathrm{PSU}_{4}(\F_2)$.
\end{theorem}

Unfortunately, the proof of Theorem~\ref{theorem:Yura-Cremona-1}
can not be applied to study the conjugacy classes in
$\mathrm{Cr}_3(\C)$ of the subgroups $\mathrm{A}_5$,
$\mathrm{PSL}_{2}(\F_7)$, and $\mathrm{A}_6$, mostly because these
subgroups are rather small and $\mathrm{Cr}_3(\C)$ contains many
non-conjugate embeddings of these groups. For example, nothing is
known so far about the~number of non-conjugate subgroups
in~\mbox{$\mathrm{Cr}_3(\C)$} that are isomorphic to the~group
$\mathrm{A}_{5}$. As for the~group~$\mathrm{A}_{6}$, we have
the~following

\begin{theorem}[{\cite{ChSh09b}}]
\label{theorem:A6} The group $\mathrm{Cr}_3(\C)$ has at least $5$
non-conjugate subgroups isomorphic~to~$\mathrm{A}_{6}$.
\end{theorem}

The~main purpose of this paper is to prove the~following

\begin{theorem}
\label{theorem:PSL-2-7} The group $\mathrm{Cr}_3(\C)$ has at least
$3$ non-conjugate subgroups isomorphic to $\PSLF$.
\end{theorem}

From now on we denote the group $\PSLF$ by $G$. Any~embedding
$G\hookrightarrow\mathrm{Cr}_3(\mathbb{C})$ arises from
some~rational~threefold $X$ admitting a~faithful action of
the~group $G$ (for details see \cite{DoIs06}, \cite{Pr09}), and
the~first examples of such embeddings come from representation
theory. Moreover,  the normalizer of the groups $G$ in
$\mathrm{Cr}_3(\mathbb{C})$ is isomorphic to the normalizer of the
group $G$ in $\mathrm{Bir}(X)$. The latter group coincides with
the group of $G$-equivariant birational self-maps of $X$, which we
denote by $\mathrm{Bir}^{G}(X)$. Note that the group of
$G$-equivariant biregular self-maps of $X$ coincides with the
normalizer of $G$ in $\mathrm{Aut}(X)$, which we denote by
$\mathrm{Aut}^{G}(X)$.

\begin{example}
\label{example:lct-space} Up to conjugation, the group
$\mathrm{Aut}(\mathbb{P}^3)$ has two subgroups isomorphic to
$\PSLF$. The first subgroup (we will call it a subgroup of type $(\mathbf{I})$)
arises from a~faithful reducible four-dimensional representation
of the~group $\PSLF$, which splits as a sum of an irreducible
three-dimensional representation and a trivial one. The second
subgroup (we will call it a subgroup of type $(\mathbf{II})$) arises from
a~faithful irreducible four-dimensional representation of
the~group $\SLF$ (see~\cite{Atlas}). Note that the subgroup of
type $(\mathbf{I})$  fixes a~point in $\mathbb{P}^{3}$, while the
subgroup of type $(\mathbf{II})$ does not fix any point in
$\mathbb{P}^{3}$.
\end{example}

The next example of an~embedding
$G\hookrightarrow\mathrm{Cr}_3(\mathbb{C})$ comes from
the~celebrated Klein quartic curve --- the unique genus $3$ curve
with an action of the group $\PSLF$ (see~\cite{Beauty}).

\begin{example}[{\cite{MS01}, \cite{GrPop01}, \cite{MeRa03}, \cite{Ma04}}]
\label{example:V22} Let $\CKlein$ be the~quartic curve in $\mathbb{P}^{2}$ that
is given by
$$
xy^3+yz^3+zx^3=0\subset\mathbb{P}^{2}\cong\mathrm{Proj}
\Big(\mathbb{C}\big[x,y,z\big]\Big),
$$
put $X=\mathrm{VSP}(\CKlein,6)$ (see Definition~\ref{definition:VSP}). Then $X$
is a~rational smooth Fano threefold~such~that
$$
\mathrm{Aut}\big(X\big)\cong\mathrm{Aut}\big(C\big)\cong\PSLF,%
$$
the~group $\mathrm{Pic}(X)$ is generated by $-K_{X}$, and
 $-K_{X}^{3}=22$ (see
Section~\ref{section:Iskovskikh}, Appendix~\ref{section:mukai}
and~ \cite{IsPr99}). One can show that $X$ is a~compactification
of the~moduli space of $(1,7)$-polarized abelian surfaces
(see~\cite{MS01}, \cite{GrPop01} and~\cite{MeRa03} for details).
\end{example}

In Sections~\ref{section:auxiliary} and \ref{section:v22}, we
prove the~following

\begin{theorem}
\label{theorem:auxiliary} Let $G$ be a~subgroup in
$\mathrm{Aut}(\mathbb{P}^3)$ such that $G\cong\PSLF$ is of
type~$(\mathbf{II})$ in the notation of
Example~\ref{example:lct-space}. Then the~variety $\mathbb{P}^{3}$
is $G$-birationally rigid (see \cite[Definition~1.9]{ChSh09b}),
and there is a~$G$-equivariant birational non-biregular involution
$\tau\in\mathrm{Cr}_{3}(\mathbb{C})$ such that
$$
\mathrm{Bir}^{G}\big(\mathbb{P}^{3}\big)=\big\langle G,\tau\big\rangle\cong
\mathrm{PSL}_{2}\big(\F_7\big)\times\mathbb{Z}_{2}.%
$$
\end{theorem}

\begin{theorem}
\label{theorem:main} Let $X$ be the~threefold constructed in
Example~\ref{example:V22}. 
Moreover, $\mathrm{Bir}^{G}(X)=\mathrm{Aut}^{G}(X)=G\cong
\mathrm{PSL}_{2}(\F_7)$ and the~threefold $X$ is $G$-birationally
superrigid (see \cite[Definition~1.10]{ChSh09b}).
\end{theorem}

It should be pointed out that Theorems~\ref{theorem:main} and
\ref{theorem:auxiliary}  imply Theorem~\ref{theorem:PSL-2-7}.
Moreover, Theorem~\ref{theorem:main} also implies the following

\begin{corollary}
\label{corollary:Cremona-Aut} Let $G$ and $G^{\prime}$ be
subgroups in $\mathrm{Aut}(\mathbb{P}^{3})$ such that $G\cong
G^{\prime}\cong\PSLF$. Then $G$ and $G^{\prime}$ are conjugate in
$\mathrm{Cr}_{3}(\mathbb{C})$ if and only if $G$ and $G^{\prime}$ are
conjugate in $\mathrm{Aut}(\mathbb{P}^{3})$.
\end{corollary}

As a~by-product of the~proof of Theorem~\ref{theorem:main}, in
Section~\ref{section:v22} we prove the~following

\begin{theorem}
\label{theorem:V22-auxiliary} Let $X$ be a~threefold constructed in
Example~\ref{example:V22}, and let $R$ be an effective $G$-invariant
$\Q$-divisor such that $R\sim_{\Q} -K_X$. Then the pair
$(X, R)$ is log canonical.
\end{theorem}

Applying \cite[Theorem~2.1]{Ti87}, \cite[Theorem~A.3]{ChSh08c} and
Theorem~\ref{theorem:V22-auxiliary}, we immediately obtain
the~following

\begin{corollary}
\label{corollary:KE} Let $X$ be a~threefold constructed in
Example~\ref{example:V22}.
Then $X$ has a $G$-invariant K\"ahler--Einstein metric.
\end{corollary}

Note that the threefold constructed in
Example~\ref{example:V22} admits both
K\"ahler--Einstein and non-K\"ahler--Einstein
smooth deformations (see \cite[Corollary~1.3]{Ti97} and \cite[\S5.3]{Do08}).
However, there was only one previously known \emph{explicit}
example of a K\"ahler--Einstein threefold in this deformation family,
which is the famous Mukai--Umemura threefold (see
\cite[Theorem~5.4.3]{Do08}).

\begin{remark}
Let $\hat{G}$ be a~subgroup in $\mathrm{SL}_{4}(\mathbb{C})$ such
that $\hat{G}\cong \SLF$ is of type~$(\mathbf{II})$ in the
notation of Example~\ref{example:lct-space}. Then the~quotient
singularity $\mathbb{C}^{4}\slash\hat{G}$ is weakly-exceptional
(see~\cite[Definition~4.1]{Pr98plt})
by~\cite[Theorem~3.16]{ChSh09a} and~\cite[Theorem~4.3]{ChSh09a},
which implies that an assertion similar to
Theorem~\ref{theorem:V22-auxiliary} holds for $\P^3$. Namely,
let~$G$ be a~subgroup in $\mathrm{Aut}(\mathbb{P}^3)$ such that
$G\cong\PSLF$ is of type~$(\mathbf{II})$ in the notation of
Example~\ref{example:lct-space}, and let $R$ be an effective
$G$-invariant $\Q$-divisor on $\P^3$ such that $R\sim_{\Q}
-K_{\P^3}$. Then the pair $(\P^3, R)$ is log canonical
by~\cite[Theorem~3.16]{ChSh09a}. In particular, it follows from
\cite[Theorem~6.4]{Ch07b} that one can apply
Theorems~\ref{theorem:main} and~\ref{theorem:auxiliary} to
construct non-conjugate embeddings
$G^n\hookrightarrow\mathrm{Cr}_{3n}(\mathbb{C})$ for $n\geqslant
2$.
\end{remark}

As an another by-product of the~proof of Theorem~\ref{theorem:main}, we
give an~example of a~smooth $K3$ surface admitting a~faithful
action of the~group $\PSLF$ (see Lemma~\ref{lemma:v22-K3}). This
might be a~new~example (at least we were not able to find it
in the~literature).

\begin{remark}
\label{remark:K3-iso} There are other known examples of
smooth~$K3$~surfaces that admit faithful actions of the~group
$\PSLF$. The Edge quartic surface in $\mathbb{P}^3$ (see
\cite{Edge47}), the cyclic cover of degree $4$ of the plane
branched along the Klein quartic curve $\CKlein$ (see
\cite{Mu88}), the double cover of the plane branched over the
Hessian $H$ of the curve $\CKlein$ (see \cite{Mu88}), and the
variety of sum of powers $\mathrm{VSP}(10,H)$ (see
\cite{Ranestad}). We do not know whether or not the~surface
provided by Lemma~\ref{lemma:v22-K3} is~isomorphic as
a~non-polarized smooth $K3$ surface to one these
smooth~$K3$~surfaces.
\end{remark}

Let us sketch the proof of Theorem~\ref{theorem:auxiliary} (the
proof of Theorem~\ref{theorem:main} is similar and simpler). Let
$G$ be a~subgroup in $\mathrm{Aut}(\mathbb{P}^3)$ such that
$G\cong\PSLF$ is of type~$(\mathbf{II})$ in the notation of
Example~\ref{example:lct-space}. Then $\P^3$ is not
$G$-birationally superrigid,
since there are non-biregular $G$-equivariant birational selfmaps
of~$\P^3$. Indeed, it is well-known that there exists a unique
smooth $G$-invariant curve $C_{6}\subset\P^3$ of degree $6$ and
genus $3$ (see \cite{Edge47}). Blowing up the curve $C_{6}$, and
contracting the proper transform of the surface in $\P^3$ spanned
by three-secants of the curve $C_{6}$, we obtain a non-biregular
$G$-equivariant birational involution 
$\tau\colon\P^3\dasharrow\P^3$ which is not defined 
along the curve $C_{6}$ (see 
Lemma~\ref{lemma:involution}, \cite[Remark~6.8]{Do99}).

To prove that $\P^3$ is $G$-birationally rigid and
$\mathrm{Bir}^{G}(\mathbb{P}^{3})=\langle G,\tau\rangle$, it is
enough to prove the following statement: for every $G$-invariant
linear system $\mathcal{M}$ without fixed components on
$\mathbb{P}^3$, either the log pair $(\mathbb{P}^3,
\lambda\mathcal{M})$ has non-canonical singularities, or the log
pair $(\mathbb{P}^3, \lambda^\prime\tau(\mathcal{M}))$ has
non-canonical singularities, where $\lambda$ and $\lambda^\prime$
are positive rational numbers such that
$\lambda\mathcal{M}\sim_{\mathbb{Q}}\lambda^\prime\tau(\mathcal{M})\sim_{\mathbb{Q}}-K_{\mathbb{P}^3}$.
In fact, the latter property is equivalent to $G$-birational
rigidity of $\P^3$ and $\mathrm{Bir}^{G}(\mathbb{P}^{3})=\langle
G,\tau\rangle$ (see \cite{Ch09}).

Applying $\tau$, we may assume that either
$\mathrm{mult}_{C_{6}}(\mathcal{M})\leqslant 1/\lambda$ or
\mbox{$\mathrm{mult}_{C_{6}}(\tau(\mathcal{M}))\leqslant
1/\lambda^\prime$} (this is usually called ``untwisting of maximal
singularities''). Thus, without loss of generality, we may assume
that $\mathrm{mult}_{C_{6}}(\mathcal{M})\leqslant 1/\lambda$,
which simply means that $(\P^3,\lambda\mathcal{M})$ is canonical
in a general point of the curve $C_{6}$. Now we have to prove that
$(\mathbb{P}^3, \lambda\mathcal{M})$ has non-canonical
singularities. In \cite{ChSh09b}, we proposed a new approach to
prove assertions of the latter type, 
which we call the ``multiplication by
two trick''. Namely, we simply observed that the singularities of
the log pair $(\mathbb{P}^3, 2\lambda\mathcal{M})$ must be worse
than log canonical if the singularities of the log pair
$(\mathbb{P}^3, \lambda\mathcal{M})$ are worse than canonical
(see Lemma~\ref{lemma:mult-by-two}, cf.~\cite[Corollary~2.3]{ChSh09b}). 
Although
the former condition is much weaker than the latter one, we are in
a position to apply the machinery of multiplier ideal sheaves to
the log pair $(\mathbb{P}^3, 2\lambda\mathcal{M})$ if the
singularities $(\mathbb{P}^3, \lambda\mathcal{M})$ are not
canonical.

Using the ``multiplication by two trick'', we proved in \cite{ChSh09b}
that $\P^3$ is $\mathrm{A}_6$-birationally rigid (recall that there is a
unique subgroup in $\mathrm{Aut}(\P^3)$ that is isomorphic to
$\mathrm{A}_6$ up to conjugation). However, in the present case,
we meet two new problems. The first problem is that sometimes we are
just unable to prove that $(\P^3,2\lambda\mathcal{M})$ is log
canonical in a general point of some subvariety of $\P^3$ even
though we believe that this is true. For example, we do not know
how to prove that $(\P^3,2\lambda\mathcal{M})$ is log canonical in
a general point of the curve $C_{6}$ despite the fact that we know
that $(\P^3,\lambda\mathcal{M})$ is canonical in a general
point of the curve $C_{6}$. The second problem is worse: even if
$(\P^3,\lambda\mathcal{M})$ is canonical, the log pair
$(\P^3,2\lambda\mathcal{M})$ may still be not log canonical. One
can easily construct such examples (see
Example~\ref{example:novye-centry}). To solve both problems, we
introduce a new technique, which we call \emph{localization and
isolation of log canonical centers}. Let us describe this
technique.

Suppose that the singularities of the log pair
$(\P^3,\lambda\mathcal{M})$ are not canonical. Let us seek for a
contradiction. Take $\mu<2\lambda$ such that the log pair
$(\mathbb{P}^3,\mu\mathcal{M})$ is strictly log canonical 
(see Section~\ref{section:preliminaries}), and
pick up a minimal center $S$ of log canonical singularities of
$(\mathbb{P}^3,\mu\mathcal{M})$ (see \cite{Kaw97}, \cite{Kaw98},
\cite[Definition~2.8]{ChSh09a}). The minimality of the center~$S$
implies that the $G$-orbit of $S$ is either a finite set, or a
disjoint union of irreducible curves. We use
Lemma~\ref{lemma:Kawamata-Shokurov-trick} to observe that one may
assume that every center of log canonical singularities of the log
pair $(\P^3,\mu\mathcal{M})$ is $g(S)$ for some $g\in\bar{G}$.
Then applying the Nadel--Shokurov vanishing theorem (see
Theorem~\ref{theorem:Shokurov-vanishing},
\cite[Theorem~9.4.8]{La04}) we obtain an upper bound on the number
of irreducible components of the $G$-orbit of $S$ with some
additional information (for example, if $S$ is a point, then the
points in its $G$-orbit must impose independent linear conditions
on sections in $H^0(\mathcal{O}_{\P^3}(4))$). If $S$ is a curve,
then the Kawamata subadjunction theorem (see
Theorem~\ref{theorem:Kawamata}, \cite[Theorem~1]{Kaw98}) implies
that $S$ is smooth, and we can proceed with applying the
Nadel--Shokurov vanishing theorem, the Riemann--Roch theorem, the
Castelnuovo bound (see Theorem~\ref{theorem:Castelnuovo},
\cite[Theorem~6.4]{Har77}), and the Corti inequality (see
Theorem~\ref{theorem:Corti}, \cite[Theorem~3.1]{Co00}), to prove
that $S=C_6$. If $S$ is a point, then analyzing small $G$-orbits
in $\P^3$, we see that the $G$-orbit of the point $S$ consists of
either $8$ or $28$ points (see
Lemmas~\ref{lemma:space-orbits-8-24} and
\ref{lemma:space-24-points}). Note that there is a unique
$G$-orbit in $\P^3$ consisting of $8$ points, and there are
exactly two $G$-orbits in $\P^3$ consisting of $28$ points (see
Lemma~\ref{lemma:space-orbits-8-24}). Thus, all potentially
dangerous (for the proof of Theorem~\ref{theorem:auxiliary})
$G$-invariant subvarieties in $\P^3$ are explicitly described.
This is \emph{localization of log canonical centers}.

Let us denote by $\Sigma$ the union of the curve $C_{6}$, the
$G$-orbit consisting of eight points, and both $G$-orbits
consisting of $28$ points. Now we can use brute force to prove
that $(\P^3,\lambda\mathcal{M})$ is canonical along $\Sigma$
keeping in mind that $\mathrm{mult}{C_{6}}(\mathcal{M})\leqslant
1/\lambda$. Then we conclude that $(\P^3,2\lambda\mathcal{M})$ is
not log canonical outside of the subset $\Sigma$. Thus, there
exists $\mu^\prime<2\lambda$ such that
$(\mathbb{P}^3,\mu^\prime\mathcal{M})$ is strictly log canonical
outside of $\Sigma$. Let~$S^\prime$ be a minimal center of log
canonical singularities of the log pair
$(\mathbb{P}^3,\mu^\prime\mathcal{M})$ that is not contained in
$\Sigma$. Note that $S^\prime$ exists by construction. If
$S^\prime$ is a point, then we can proceed as before and easily
obtain a contradiction with the Nadel--Shokurov vanishing theorem.
Thus, we conclude that $S^\prime$ is a curve. If $S^\prime$ and
$\Sigma$ are disjoint, then we also can proceed as before and
obtain a contradiction with the Kawamata subadjunction theorem,
the Nadel--Shokurov vanishing theorem, the Riemann--Roch theorem,
the Castelnuovo bound, the Corti inequality, etc. This is
\emph{isolation of log canonical centers}.

If $S^\prime$ is a curve and $S^\prime\cap\Sigma\ne\varnothing$,
then we have a problem. Indeed, we can not apply the Kawamata
subadjunction theorem to $S^{\prime}$, because the log pair
$(\mathbb{P}^3,\mu^\prime\mathcal{M})$ may no longer be log
canonical in the points of the finite set $S^\prime\cap\Sigma$ if
$\mu^{\prime}>\mu$. Recall that $\mu$ is a rational number such
that $(\mathbb{P}^3,\mu\mathcal{M})$ is strictly log canonical,
i.e. log canonical and not Kawamata log terminal. To solve this
new problem, we have to isolate and localize new log canonical
centers again, i.e. to repeat the previous arguments to the union
of $\Sigma$ and new potentially dangerous curves in $\P^3$. And
then there is a chance that we have to repeat this process again
and again. So all together this looks messy. And the proof of
Theorem~\ref{theorem:auxiliary} is messy. To simplify it, we
describe all potentially dangerous log canonical centers before
the proof, and then we try to localize and isolate them together
at once. 

We organize this paper in the~following way. In
Section~\ref{section:preliminaries}, we recall several well-known
preliminary results. In Section~\ref{section:space}, we collect
results about the~action of the~group $\PSLF$ on $\mathbb{P}^{3}$.
In Section~\ref{section:Iskovskikh}, we collect results about
the~threefold constructed in Example~\ref{example:V22}. In
Section~\ref{section:auxiliary}, we prove
Theorem~\ref{theorem:auxiliary} using results obtained in
Section~\ref{section:space}. In Section~\ref{section:v22}, we
prove Theorems~\ref{theorem:main} and~\ref{theorem:V22-auxiliary}
using results obtained in Section~\ref{section:Iskovskikh}. In
Appendix~\ref{section:mukai}, we describe Mukai's construction of
Fano threefolds of degree $22$. In
Appendix~\ref{section:characters}, we collect elementary results
about the~groups~$\PSLF$~and~$\SLF$. Throughout the paper we use
standard notation for cyclic, dihedral, symmetric and alternating
groups. For a group $\Gamma$ we denote by~\mbox{$2.\Gamma$} a
(non-trivial) central extension of $\Gamma$ by the central
subgroup $\Z_2$.

\section{Preliminaries}%
\label{section:preliminaries}

Throughout the paper we use the standard language of the
singularities of pairs (see \cite{Ko97}). 
By \emph{strictly log canonical} singularities
we mean log canonical singularities that are not Kawamata log
terminal.

Let $X$ be a~variety that has at most log terminal singularities, 
and let $B_{X}$ be a~formal
$\mathbb{Q}$-linear combination of prime divisors and mobile
linear systems
$B_{X}=\sum_{i=1}^{r}a_{i}B_{i}+\sum_{j=1}^{s}c_{j}\mathcal{M}_{j}$,
where $B_{i}$ and $\mathcal{M}_{j}$ are~a~prime Weil divisor and
a~linear system on the~variety $X$ that has no fixed components,
respectively, and $a_{i}$ and $c_{j}$ are non-negative rational
numbers. Note that we can consider~$B_{X}$ as a Weil divisor.
Suppose that $B_{X}$ is a~$\mathbb{Q}$-Cartier divisor.

\begin{definition}
\label{definition:mobile} We say that $B_{X}$ and $(X,B_{X})$ are
mobile if $a_{1}=a_{2}=\ldots=a_{r}=0$.
\end{definition}

Let $\pi\colon\bar{X}\to X$ be a~log resolution for the~log pair
$(X, B_{X})$, let $\bar{B}_{i}$ and $\bar{\mathcal{M}}_{j}$ be
the~proper transforms of the~divisor $B_{i}$ and the~linear system
$\mathcal{M}_{j}$ on the~variety $\bar{X}$, respectively. Then
$$
K_{\bar{X}}+\sum_{i=1}^{r}a_{i}\bar{B}_{i}+\sum_{j=1}^{s}c_{j}\bar{\mathcal{M}}_{j}\sim_{\mathbb{Q}}\pi^{*}\Big(K_{X}+B_{X}\Big)+\sum_{i=1}^{m}d_{i}E_{i},
$$
where $E_{i}$ is an~exceptional divisor of the~morphism $\pi$, and
$d_{i}$ is a~rational number. Put
$$
\mathcal{I}\Big(X, B_{X}\Big)=\pi_{*}\Bigg(\mathcal{O}_{\bar{X}}\Big(\sum_{i=1}^{m}\lceil d_{i}\rceil E_{i}-\sum_{i=1}^{r}\lfloor a_{i}\rfloor B_{i}\Big)\Bigg),%
$$
and recall that $\mathcal{I}(X, B_{X})$ is known as the~multiplier
ideal sheaf (see \cite[Section~9.2]{La04}).

\begin{theorem}[{\cite[Theorem~9.4.8]{La04}}]
\label{theorem:Shokurov-vanishing} Let $H$ be a~nef  and big
$\mathbb{Q}$-divisor on $X$ such that
$K_{X}+B_{X}+H\sim_{\mathbb{Q}} D$ for some Cartier divisor $D$ on
the~variety $X$. Then $H^{i}(\mathcal{I}(X, B_{X})\otimes D)=0$
for every $i\geqslant 1$.
\end{theorem}

Let $\mathcal{L}(X, B_{X})$ be a~subscheme given by the ideal
sheaf $\mathcal{I}(X, B_{X})$. Put $\mathrm{LCS}(X,
B_{X})=\mathrm{Supp}(\mathcal{L}(X, B_{X}))$.

\begin{remark}
\label{remark:log-canonical-subscheme} If the~log pair $(X,B_{X})$
is log canonical, then the~subscheme $\mathcal{L}(X,B_{X})$ is
reduced.
\end{remark}

Let $Z$ be an~irreducible subvariety of the~variety $X$.

\begin{definition}[{\cite[Definition~1.3]{Kaw97}}]
\label{definition:lc-center} The~subvariety $Z$ is said to be a~
center of log cano\-ni\-cal~singularities (non-log canonical
singularities, respectively) of the~log pair $(X, B_{X})$ if
\begin{itemize}
\item either $a_{i}\geqslant 1$ ($a_{i}>1$, respectively) and $Z=B_{i}$  for some  $i\in\{1,\ldots,r\}$,%
\item or $d_{i}\leqslant -1$ ($d_{i}<-1$, respectively) and $Z=\pi(E_{i})$
for some  $i\in\{1,\ldots,m\}$  and some $\pi$.%
\end{itemize}
\end{definition}

Let $\mathbb{LCS}(X, B_{X})$ and $\mathbb{NLCS}(X, B_{X})$ be
the~sets of centers of log canonical and non-log canonical
singularities of the~log pair $(X, B_{X})$, respectively. Then
$\mathbb{NLCS}(X,B_{X})\subseteq\mathbb{LCS}(X,B_{X})$.

\begin{theorem}[{\cite[Theorem~3.1]{Co00}}]
\label{theorem:Corti} Suppose that $\dim(X)=2$, the~set
$\mathbb{NLCS}(X, B_{X})$ contains a~point $P\in
X\setminus\mathrm{Sing}(X)$, the~boundary $B_{X}$ is mobile and
$s=1$. Then
$$
\mathrm{mult}_{P}\Big(M_{1}\cdot M_{1}^{\prime}\Big)>4\slash c_{1}^{2},%
$$
where $M_{1}$ and $M_{1}^{\prime}$ are general curves in
the~linear system $\mathcal{M}_{1}$.
\end{theorem}

Let us denote by $\mathrm{NLCS}(X,B_{X})$ the~proper subset of
the~variety $X$ that is a union of all centers in
$\mathbb{NLCS}(X,B_{X})$.

\begin{definition}[{\cite[Definition 2.2]{ChSh08c}}]
\label{definition:canonical-center} The~subvariety $Z$ is said to
be a~center of canonical~singularities (non-canonical
singularities, respectively) of the~log pair $(X, B_{X})$ if
$Z=\pi(E_{i})$ and $d_{i}\leqslant 0$ ($d_{i}<0$, respectively)
for some $i\in\{1,\ldots,m\}$ and some choice of the~morphism
$\pi$.
\end{definition}

Let $\mathbb{CS}(X, B_{X})$ and $\mathbb{NCS}(X, B_{X})$ be
the~sets of centers of canonical and non-canonical singularities
of the~log pair $(X, B_{X})$, respectively. Then $\mathbb{NCS}(X,
B_{X})\subseteq\mathbb{CS}(X, B_{X})$.

\begin{theorem}[{\cite[Corollary 3.4]{Co00}}]
\label{theorem:Iskovskikh} Suppose that $\dim(X)=3$, the~set
$\mathbb{NCS}(X, B_{X})$ contains a~point $P\in
X\setminus\mathrm{Sing}(X)$, the~boundary $B_{X}$ is mobile and
$s=1$. Then
$$
\mathrm{mult}_{P}\Big(M_{1}\cdot M_{1}^{\prime}\Big)>4\slash c_{1}^{2},%
$$
where $M_{1}$ and $M_{1}^{\prime}$ are general surfaces in
the~linear system $\mathcal{M}_{1}$.
\end{theorem}

Let us denote by $\mathrm{NCS}(X,B_{X})$ the~proper subset of
the~variety $X$ that is a union of all centers in $\mathbb{NCS}(X,B_{X})$.

\begin{lemma}
\label{lemma:mult-by-two} Suppose that $X$ is smooth at a~general
point of the~subvariety $Z$. Then $\mathbb{CS}(X, B_{X})\subseteq
\mathbb{LCS}(X, 2B_{X})$ and $\mathbb{NCS}(X, B_{X})\subseteq
\mathbb{NLCS}(X, 2B_{X})$.
\end{lemma}

\begin{proof}
This is obvious, because $X$ is smooth at a~general point of
 $Z$.
\end{proof}

Suppose that $Z\in\mathbb{LCS}(X, B_{X})$ and $(X,B_{X})$ is log
canonical along  $Z$.

\begin{lemma}[{\cite[Proposition~1.5]{Kaw97}}]
\label{lemma:centers} Let $Z^{\prime}$ be a~center in
$\mathbb{LCS}(X, B_{X})$ such that $Z^{\prime}\ne Z$. Then any
irreducible component of the intersection $Z\cap Z^{\prime}$ is an
element in $\mathbb{LCS}(X, B_{X})$.
\end{lemma}

Suppose  that $Z$ is a~minimal center in $\mathbb{LCS}(X, B_{X})$
(see \cite{Kaw97}, \cite{Kaw98}, \cite[Definition~2.8]{ChSh09a}).

\begin{theorem}[{\cite[Theorem~1]{Kaw98}}]
\label{theorem:Kawamata}
The~variety $Z$ is normal  and has at most rational singularities.
If $\Delta$ is an~ample
$\mathbb{Q}$-Cartier $\mathbb{Q}$-divisor on $X$, then
there exists an~effective $\mathbb{Q}$-divisor $B_{Z}$ on
the~variety $Z$ such that
$$
\Big(K_{X}+B_{X}+\Delta\Big)\Big\vert_{Z}\sim_{\mathbb{Q}} K_{Z}+B_{Z},%
$$
and $(Z,B_{Z})$ has Kawamata log terminal singularities.
\end{theorem}

Let $G$ be a~finite subgroup of the~group $\mathrm{Aut}(X)$.

\begin{lemma}
\label{lemma:F21-A4} Let $P$, $C$ and $S$ be a~point, curve and
surface in $X$, respectively. Suppose that
$\mathrm{Sing}(X)\not\ni P\in C\subset S$ and $\mathrm{dim}(X)=3$.
Suppose that $P$ and $C$ are $G$-invariant, and either
$G\cong\mathrm{A}_{4}$ or $G\cong
\mathbb{Z}_{7}\rtimes\mathbb{Z}_{3}$. Then
$\mathrm{mult}_{P}(C)\geqslant 3$ and the~surface~$S$ is singular
at the~point $P\in X$.
\end{lemma}

\begin{proof}
Let $\gamma\colon U\to X$ be a~blow up of the~threefold $X$ at
the~point $P$, let $E$ be the~$\gamma$-exceptional divisor, and
let $\bar{C}$ be the~ proper transforms of the~curve $C$ on
the~threefold $U$. Then $\mathrm{mult}_{P}(C)\geqslant
|\bar{C}\cap E|$.

The~group $G$ naturally acts on $E\cong\mathbb{P}^{2}$. This
action comes from a~faithfull three-dimensional representation of
the~group $G$, which must be irreducible, because the~group $G$
does not have two-dimensional irreducible  representations and
the~ group~$G$ is not abelian. Thus $|\bar{C}\cap E|\geqslant 3$,
and the~points of the~set $\bar{C}\cap E$ are not contained in
a~single line in $E\cong\mathbb{P}^{2}$, which immediately implies
that the surface~$S$ must be singular at the~point $P\in X$.
\end{proof}

Suppose  that $B_{X}$ is $G$-invariant.  Recall that $(X,B_{X})$
is log canonical.

\begin{remark}
\label{remark:centers} Let $g$ be an~elements in $G$. Then
$g(Z)\in\mathbb{LCS}(X,B_{X})$. By Lemma~\ref{lemma:centers}, we
have $Z\cap g(Z)\ne \varnothing\iff Z=g(Z)$.
\end{remark}

Suppose  that $B_{X}$ is ample. Take an~arbitrary rational number
$\epsilon>1$.

\begin{lemma}
\label{lemma:Kawamata-Shokurov-trick} There is a~$G$-in\-va\-riant
linear system $\mathcal{B}$ on the~variety $X$ that has no
fixed~components, and there are rational numbers $\epsilon_{1}$
and $\epsilon_{2}$ such  that $1\geqslant \epsilon_{1}\gg 0$ and
$1\gg\epsilon_{2}\geqslant 0$~and
$$
\mathbb{LCS}\Big(X, \epsilon_{1} B_{X}+\epsilon_{2}\mathcal{B}\Big)=\Bigg(\bigsqcup_{g\in G}\Big\{g\big(Z\big)\Big\}\Bigg)\bigsqcup\mathbb{NLCS}\Big(X, \epsilon_{1} B_{X}+\epsilon_{2}\mathcal{B}\Big),%
$$
the~log pair $(X,\epsilon_{1} B_{X}+\epsilon_{2}\mathcal{B})$ is
log canonical along $g(Z)$ for every $g\in G$, the~equivalence
$\epsilon_{1} B_{X}+\epsilon_{2}\mathcal{B}\sim_{\mathbb{Q}}
\epsilon B_{X}$ holds, and $\mathbb{NLCS}(X, \epsilon_{1}
B_{X}+\epsilon_{2}\mathcal{B})=\mathbb{NLCS}(X, B_{X})$.
\end{lemma}

\begin{proof}
See the~proofs of \cite[Theorem~1.10]{Kaw97} and
\cite[Theorem~1]{Kaw98}.
\end{proof}

Let $\mathcal{C}$ be a~smooth irreducible curve in
$\mathbb{P}^{3}$ of genus $g$ and degree $d$.

\begin{theorem}[{\cite[Theorem~6.4]{Har77}}]
\label{theorem:Castelnuovo} If $\mathcal{C}$ is not contained in
a~hyperplane in $\mathbb{P}^{3}$, then
$$
g\leqslant
\left\{\aligned%
&\frac{(d-2)^2}{4}\ \text{if $d$ is even},\\
&\frac{(d-1)(d-3)}{4}\ \text{if $d$ is odd}.\\
\endaligned
\right.
$$
\end{theorem}

Suppose, in addition, that $G\cong\PSLF$ and $\mathcal{C}$ admits
a faithful action of the group $G$.

\begin{lemma}
\label{lemma:long-orbit} Let $\Sigma$ be a~$G$-orbit of a~point in
$\mathcal{C}$. Then $|\Sigma|\in\{24,42,56,84,168\}$.
\end{lemma}

\begin{proof}
This follows from Lemma~\ref{lemma:PSL-maximal-subgroups}, since
stabilizer subgroups of all points in $X$ are cyclic.
\end{proof}

\begin{lemma}
\label{lemma:sporadic-genera} Suppose that $g\leqslant 30$. Then
$g\in\{3 ,8, 10, 15, 17, 19, 22, 24, 29\}$, and the~number of
$G$-orbits in $\mathcal{C}$ consisting of  $24$, $42$, $56$, $84$
points can be described as follows:
\begin{center}\renewcommand\arraystretch{1.1}
\begin{tabular}{|c|c|c|c|c|}%\label{table:psl27}
\hline
\text{genus $g$} & \text{$24$ points} & \text{$42$ points} &  \text{$56$ points} & \text{$84$ points}\\
\hline
$3$ & $1$ & $0$ & $1$ & $1$ \\
\hline
$8$ & $0$ & $1$ & $2$ & $0$ \\
\hline
$10$ & $1$ & $1$ & $0$ & $1$\\
\hline
$15$ &$0$  & $2$ & $1$ & $0$ \\
\hline
$15$ & $0$ & $0$ & $1$ & $3$\\
\hline
$17$ & $1$ & $0$ & $2$ & $0$\\
\hline
$19$ & $2$ & $0$ & $0$ & $1$\\
\hline
$22$ & $0$ & $3$ & $0$ & $0$\\
\hline
$22$ & $0$ & $1$ & $0$ & $3$ \\
\hline
$24$ & $1$ & $1$ & $1$ & $0$\\
\hline
$29$ & $0$ & $0$ & $2$ & $2$\\
\hline
\end{tabular}
\end{center}
\end{lemma}

\begin{proof}
It follows from the~classification of finite subgroups of
the~group $\mathrm{PGL}_2(\mathbb{C})$ that $g\neq 0$, and it
follows from the~non-solvability of the~group $G$ that $g\neq 1$.

Let $\Gamma\subset G$ be a~stabilizer of a~point in $\mathcal{C}$.
Then $\Gamma\cong\mathbb{Z}_k$ for $k\in\{1,2,3,4,7\}$ by
Lemma~\ref{lemma:PSL-maximal-subgroups}.

Put $\bar{\mathcal{C}}=\mathcal{C}\slash G$. Then
$\bar{\mathcal{C}}$ is a~smooth curve of genus $\bar{g}$.
The~Riemann--Hurwitz formula gives
$$
2g-2=168\big(2\bar{g}-2\big)+84a_2+112a_3+126a_4+144a_7,
$$
where $a_k$ is the~number of $G$-orbits in $\mathcal{C}$ with
a~stabilizer of a~point isomorphic to~$\mathbb{Z}_k$.

Since $a_k\geqslant 0$, one has $\bar{g}=0$,  and
$2g-2=-336+84a_2+112a_3+126a_4+144a_7$, which easily implies
the~required assertions.
\end{proof}

\begin{remark}
We do not claim that every case listed in Lemma~\ref{lemma:sporadic-genera}
is realized.
\end{remark}

\medskip

Let $L$ be a~$G$-invariant line bundle on the~curve $\mathcal{C}$
(see \cite[\S1]{Do99}).

\begin{lemma}\label{lemma:G-linearized-dimension}
If $\mathrm{deg}(L)\leqslant 23$ and $L$ is a $G$-linearizable
line bundle (see \cite[\S1]{Do99}), then
$h^{0}(\mathcal{O}_{\mathcal{C}}(L))\not\in\{1,2,4,5\}$.
\end{lemma}

\begin{proof}
If $L$ is $G$-linearized, then there is a~natural linear action of
the~group $G$ on $H^{0}(\mathcal{O}_{\mathcal{C}}(L))$, and
the~required assertion follows from Lemma~\ref{lemma:long-orbit}
and Appendix~\ref{section:characters}.
\end{proof}

\begin{theorem}[{\cite[Theorem~2.4]{Do99},\cite[Lemma~2.10]{MeRa03}}]\label{theorem:Dolgachev}
If $g=3$, then $\mathcal{C}\cong \CKlein$ and there is
$\theta\in\mathrm{Pic}(\mathcal{C})$ such that $2\theta\sim
K_{\mathcal{C}}$ and
$\mathrm{Pic}^{G}(\mathcal{C})=\langle\theta\rangle$.
\end{theorem}

Thus, if $g=3$, then $\mathrm{deg}(L)$ is even by
Theorem~\ref{theorem:Dolgachev}, because we assume that $L$ is
$G$-invariant.

\begin{lemma}
\label{lemma:g-8-d-7} Suppose that $g=8$ and $L$
is~$G$-linearizable. Then $7\mid \mathrm{deg}(L)$.
\end{lemma}

\begin{proof}
Suppose that $7\nmid \mathrm{deg}(L)$. Then there are integers $a$
and $b$ such that \mbox{$14a+b\mathrm{deg}(L)=8$}. Put
$D=aK_{\mathcal{C}}+bL$. Then $\mathrm{deg}(D)=8$ and $D$ is
a~$G$-linearizable line bundle.

By the~Riemann--Roch theorem, the~Clifford theorem (see
\cite[Theorem~5.4]{Har77})~and~Lemma~\ref{lemma:G-linearized-dimension},
we have $h^{0}(\mathcal{O}_{\mathcal{C}}(D))=3$. Then
$h^{0}(\mathcal{O}_{\mathcal{C}}(K_{\mathcal{C}}-D))=2$, which
contradicts Lemma~\ref{lemma:G-linearized-dimension}.
\end{proof}

\begin{lemma}
\label{lemma:g-10-d-3} Suppose that $g=10$ and $L$
is~$G$-linearizable. Then $3\mid \mathrm{deg}(L)$.
\end{lemma}

\begin{proof}
It follows from \cite[page~6]{Do99} that there exists
$G$-linearizable line bundle $\gamma$ on $\mathcal{C}$ of degree
$6$ that generates the group of all $G$-linearizable line bundles
on $\mathcal{C}$. Then it follows from
\cite[Proposition~2.2]{Do99} that $3\mid \mathrm{deg}(L)$.
\end{proof}

\section{Projective space}
\label{section:space}

Let $\zeta$ be a~primitive seventh root of unity, let $\hat{G}$ be
a subgroup in $\mathrm{SL}_{4}(\mathbb{C})$ such that
$$
\hat{G}=\left\langle \left(
\begin{array}{cccc}
1 & 0& 0 & 0\\
0 & \zeta & 0 & 0\\
0 & 0& \zeta^4 & 0\\
0 & 0& 0 & \zeta^2\\
\end{array}
\right), \frac{1}{\sqrt{-1}}\left(
\begin{array}{cccc}
1 & 2 & 2& 2 \\
1 & \zeta+\zeta^6 & \zeta^2+\zeta^5& \zeta^3+\zeta^4\\
1 & \zeta^2+\zeta^5 & \zeta^3+\zeta^4& \zeta+\zeta^6\\
1 & \zeta^3+\zeta^4 & \zeta+\zeta^6& \zeta^2+\zeta^5 \\
\end{array}
\right)\right\rangle,
$$
and let us denote by the~symbol $U_{4}$ the~corresponding faithful
four-dimensional representation~of~the~group $\hat{G}$ (cf.
Appendix~\ref{section:characters}). Then
$\hat{G}\cong\mathrm{SL}_{2}(\F_7)$
and $U_{4}$ is irreducible (see \cite{MaSl73},~\cite{Atlas}).

Let
$\phi\colon\mathrm{SL}_{4}(\mathbb{C})\to\mathrm{Aut}(\mathbb{P}^{3})$
be a~natural projection.  Put $G=\phi(\hat{G})$. Then
\mbox{$G\cong\PSLF$}
is of type~$(\mathbf{II})$ in the notation of Example~\ref{example:lct-space}.

\begin{remark}\label{remark:Blichfeld}
It follows from \cite[Chapter~VII]{Bli17} that
$\mathrm{Aut}^G(\P^3)=G$.
\end{remark}

\begin{lemma}
\label{lemma:space-orbits-8-24} Let $P$ be a~point in
$\mathbb{P}^{3}$, and let $\Sigma$ be its $G$-orbit. Suppose that
$|\Sigma|\leqslant 41$. Then either $|\Sigma|=8$  and the~orbit
$\Sigma$ is unique, or $|\Sigma|=24$ and the~orbit~$\Sigma$ is
unique, or $|\Sigma|=28$ and there are exactly two possibilities
for the~orbit $\Sigma$.
\end{lemma}

\begin{proof}
It follows from Corollary~\ref{corollary:PSL-permutation} that
$|\Sigma|\in\{7, 8, 14, 21, 24, 28\}$, because the~representation
$U_{4}$ is irreducible (so that $|\Sigma|\ne 1$). Let $G_{P}$ be
a~stabilizer subgroup in $G$ of the~point~$P$, and let
$\hat{G}_{P}$ be the~preimage of the~subgroup $G_P$ under $\phi$.
If $|\Sigma|=21$, then $\hat{G}_{P}\cong 2.\mathrm{D}_4$, which is
impossible by Lemma~\ref{lemma:SL-2-7-subgroups}. If
$|\Sigma|\in\{7, 14\}$, then $\hat{G}_{P}$ has a~subgroup
isomorphic~to~$2.\mathrm{A}_4$, which is also impossible by
Lemma~\ref{lemma:SL-2-7-subgroups}. Thus, we see that
$|\Sigma|\in\{8,24,28\}$.

Suppose that $|\Sigma|=8$. Then it follows from
Lemmas~\ref{lemma:PSL-maximal-subgroups} and \ref{lemma:F21} that
\mbox{$G_{P}\cong \mathbb{Z}_7\rtimes\mathbb{Z}_{3}$}, the~orbit $\Sigma$
does exist, the~point $P$ is the~unique $G_{P}$-invariant point in
$\mathbb{P}^{3}$, and $\Sigma$ is unique, since all subgroups of
the~group $G$ that are isomorphic to
$\mathbb{Z}_7\rtimes\mathbb{Z}_{3}$ are conjugate by
Lemma~\ref{lemma:PSL-maximal-subgroups}.

Suppose that $|\Sigma|=24$. Then $G_{P}\cong\mathbb{Z}_7$ and
$\hat{G}_{P}\cong\mathbb{Z}_{14}$. Take any $g\in\hat{G}_{P}$ such
that~$\hat{G}_{P}=\langle g\rangle$, and let $R_n$ be
a~one-dimensional representation of the~group $\hat{G}_{P}$ such
that $g$ acts on $R_{n}$~by~multi\-plication by $-\zeta^{n}$. For
a suitable choice of $g$, we have isomorphism of
$\hat{G}_{P}$-representations $U_4\cong R_0\oplus R_1\oplus
R_2\oplus R_4$, which implies that $\mathbb{P}^3$ contains exactly
$3$ different points besides $P\in\mathbb{P}^3$, say $P_1$, $P_2$
and $P_3$, that are fixed by the~group $G_{P}$. There is a~unique
subgroup $H\subset G$ such that
$\mathbb{Z}_7\rtimes\mathbb{Z}_{3}\cong H\supset G_{P}$, and we
may assume that the~point $P_{1}$ is $H$-invariant (that is,
corresponds to the~subrepresentation $R_0$). Then its~$G$-orbit
consists of eight~points. Thus, we see that $\{P, P_2, P_3\}$ is
a~$H$-orbit, which implies that the~orbit $\Sigma$ exists and it
is unique.

Suppose that $|\Sigma|=28$. Then $G_{P}\cong\mathrm{S}_3$. The
action of $G_{P}$ on $\P^3$ is induced by a four-dimensional
representation of $2.\SS_3$. By
Lemma~\ref{lemma:SL-2-7-subgroups}, this representation splits as
a sum of an irreducible two-dimensional representation and two
non-isomorphic one-dimensional representations. Thus, there is
a~unique point $P^{\prime}\in\mathbb{P}^3$~such~that $P\ne
P^{\prime}$ and $P^{\prime}$ is fixed by $G_{P}$. On the other
hand, the group $G$ contains exactly $28$ subgroups isomorphic to
$\mathrm{S}_3$. Moreover, it easily follows from
Lemma~\ref{lemma:SL-2-7-subgroups} that two different subgroups in
$G$ isomorphic to $\SS_3$ generates either a subgroup isomorphic
to $\SS_4$ or the whole group $G$. Thus, it follows from
Lemma~\ref{lemma:SL-2-7-subgroups} that no two of these $28$
subgroups isomorphic to $\SS_3$ can fix one point in $\P^3$. Thus,
there are exactly two $G$-orbits in~$\mathbb{P}^{3}$ consisting of
$28$ points.
\end{proof}

Let $\Sigma_{8}$, $\Sigma_{24}$ and
$\Sigma_{28}\ne\Sigma_{28}^{\prime}$ be $G$-orbits in
$\mathbb{P}^{3}$ consisting of $8$, $24$ and $28$ points,
respectively.

The~group $\hat{G}$ naturally acts on
$\mathbb{C}[x_1,x_2,x_3,x_4]$. Put
$$
a=x_{2}x_{3}x_{4}, b=x_2^3x_3+x_3^3x_4+x_4^3x_2,
c=x_2^2x_3^3+x_3^2x_4^3+x_4^2x_2^3,
d=a^2+x_2x_3^5+x_3x_4^5+x_4x_2^5,
$$
and $e=7ab+x_2^7+x_3^7+x_4^7$. Furthermore, put $\Phi_{4}=2x_1^4+6ax_1+b$ and
\begin{gather*}
\Phi_{6}=8x_1^6-20ax_1^3-10bx_1^2-10cx_1-14a^2-d,\\
\Phi_{8}=x_1^8-2ax_1^5+bx_1^4+2cx_1^3+(6a^2+d)x_1^2+2abx_1+ac,\\
\Phi_{8}^{\prime}=x_1^8+14ax_1^5-7bx_1^4+14cx_1^3-7dx_1^2+ex_1,\\
\Phi_{14}=48x_1^{14}+168ax_1^{11}+308bx_1^{10}-1596cx_1^9+
126\big(42a^2+11d\big)x_1^8-\\ -8\big(37e+490ab\big)x_1^7+
196\big(12ac+5b^2\big)x_1^6+196\big(15ad-13bc\big)x_1^5+\\
+14\big(182c^2-86ae-7bd\big)x_1^4
+28\big(11be-42cd\big)x_1^3+14\big(21d^2-16ce\big)x_1^2+14dex_1-e^2.
\end{gather*}

\begin{theorem}[{\cite{Edge47}, \cite[Theorem~1]{MaSl73}}]
\label{theorem:SL-invariants-in-U4} The forms $\Phi_4$, $\Phi_6$,
$\Phi_8$, $\Phi_8^{\prime}$ and $\Phi_{14}$ are
$\hat{G}$-invariant.
\end{theorem}

\begin{remark}
\label{remark:pencils} There are no $\hat{G}$-invariant two-dimensional vector
subspaces in $\mathbb{C}[x_1,x_2,x_3,x_4]$
that consist of linear, quadratic, cubic, quartic, quintic or sextic forms (see \cite[Appendix 1]{Do99}).%
\end{remark}

Let $F_{i}$ be a~surface in $\mathbb{P}^{3}$ that is given by the~equation
$$
\Phi_{i}\big(x_{1},x_{2},x_{3},x_{4}\big)=0\subset\mathbb{P}^{3}\cong\mathrm{Proj}\Big(\mathbb{C}\big[x_1,x_2,x_3,x_4\big]\Big),
$$
and let $F_{8}^{\prime}$ be a~surface in $\mathbb{P}^{3}$ that is given by
$\Phi_{8}^{\prime}(x_{1},x_{2},x_{3},x_{4})=0$.

\begin{theorem}[{\cite[Theorem~1]{MaSl73}}]
\label{theorem:invariants} There are no $G$-invariant odd degree
surfaces in~$\mathbb{P}^{3}$, there are no $G$-invariant quadric
surfaces in $\mathbb{P}^{3}$, and the~only $G$-invariant quartic
surface in $\mathbb{P}^{3}$ is the~surface $F_4$.
\end{theorem}

One can check that the~surface $F_{4}$ is smooth.

\begin{lemma}[{cf. \cite[Theorem 1]{MaSl73}}]
\label{lemma:invariants-properties} The~sets $F_{4}\cap F_{6}\cap
F_{8}^{\prime}$, $F_{4}\cap F_{6}\cap F_{14}$, and~\mbox{$F_{4}\cap
F_{8}^{\prime}\cap F_{14}$} are finite.
\end{lemma}

\begin{proof}
This follows from explicit computations. We used the Magma
software~\cite{Magma} to carry them out.
\end{proof}

There is a~$G$-invariant irreducible smooth curve
$C_{6}\subset\mathbb{P}^{3}$ of genus $3$ and degree $6$ such that
$\Sigma_{24}=C_{6}\cap F_{4}$, and $C_{6}$ is an~intersection of
cubic surfaces in $\mathbb{P}^{3}$ (see \cite[page~154]{Edge47},
\cite[Example 2.8]{Do99}).

\begin{lemma}
\label{lemma:jacobian-curve} Let $C$ be a~$G$-invariant curve in
$\mathbb{P}^{3}$ such that $\mathrm{deg}(C)\leqslant 6$. Then $C=C_{6}$.
\end{lemma}

\begin{proof}
By Corollary~\ref{corollary:PSL-permutation}, we may assume that
the~curve $C$ is irreducible. If the~curve~$C$ is singular, then
$|\mathrm{Sing}(C)|\geqslant 8$ by
Lemma~\ref{lemma:space-orbits-8-24}, which easily leads to
a~contradiction by applying Lemma~\ref{lemma:long-orbit} to
the~normalization of the~curve $C$. Then $C$~is~smooth. Since
$U_4$ is an irreducible representation of the~group $\hat{G}$,
the~curve $C$ is not contained in a~plane in $\mathbb{P}^{3}$.
Then $C$ is a~curve of genus $3$ and degree $6$ by
Theorem~\ref{theorem:Castelnuovo} and
Lemma~\ref{lemma:sporadic-genera}. By
Theorem~\ref{theorem:Dolgachev}, there is a~unique $G$-invariant
line bundle of degree $6$ on the~curve $C$, which implies that
the~embedding $C\hookrightarrow\P^3$ is unique up to the~action of
the~group $\mathrm{Aut}^G(\P^3)$. But $\mathrm{Aut}^G(\P^3)=G$ by
Remark~\ref{remark:Blichfeld}, which implies that $C=C_6$.
\end{proof}

Note that $C_{6}\cap\Sigma_{8}=\varnothing$ by
Lemma~\ref{lemma:long-orbit}.

\begin{lemma}
\label{lemma:involution} There is a~non-biregular involution
$\tau\in\mathrm{Bir}^{G}(\mathbb{P}^{3})$ such~that~the~diagram
\begin{equation}
\label{equation:space-involution} \xymatrix{
&V\ar@{->}[dl]_{\alpha}\ar@{->}[dr]^{\beta}&\\
\mathbb{P}^{3}\ar@{-->}[rr]_{\tau}&&\mathbb{P}^{3}}
\end{equation}
commutes and $\langle G,\tau\rangle\cong G\times\mathbb{Z}_{2}$,
where $\alpha$ and $\beta$ are blow ups of the~curve $C_{6}$.
\end{lemma}

\begin{proof}
The~existence of the commutative diagram
$(\ref{equation:space-involution})$ is well-known (see
\cite[Remark~6.8]{Do99}), the~isomorphism $\langle
G,\tau\rangle\cong\mathrm{PSL}_{2}(\F_7)\times\mathbb{Z}_{2}$
follows from the~last three lines of the~proof of
\cite[Lemma~6.4]{Do99}.
\end{proof}

Let us introduce a~$G$-invariant curve in $\mathbb{P}^{3}$, which has never
been mentioned in the~literature.

\begin{lemma}
\label{lemma:space-C-14} There is a~$G$-invariant irreducible
curve $C_{14}\subset\mathbb{P}^{3}$ of degree $14$ such that
$\Sigma_{8}\subset C_{14}$.
\end{lemma}

\begin{proof}
Let $\CKlein$ be the~genus $3$ curve introduced in
Example~\ref{example:V22}. Then \mbox{$\mathrm{Aut}(\CKlein)\cong G$},
which implies that $\CKlein$ admits a~natural action of the~group
$G$. It follows from Theorem~\ref{theorem:Dolgachev} that
$\mathrm{Pic}^{G}(\CKlein)=\langle\theta\rangle$, where $\theta$
is a~$G$-invariant line bundle of degree $2$. By
\cite[Lemma~6.4]{Do99}, there exists an isomorphism
$H^{0}(\mathcal{O}_{\CKlein}(7\theta))\cong U_{4}\oplus U_{8}$,
where $U_{8}$ is an irreducible eight-dimensional representation
of the~group $\hat{G}$ (see Appendix~\ref{section:characters}).

The~linear system $|7\theta|$ gives a~$G$-equivariant embedding
$\rho\colon \CKlein\hookrightarrow \mathbb{P}^{11}$ such that
there exist unique $G$-invariant linear subspaces $\Pi_{3}$ and
$\Pi_{7}$ in $\mathbb{P}^{11}$ of dimensions $3$ and $7$,
respectively. Then $\rho(\CKlein)\cap \Pi_{7}=\varnothing$ by
Lemma~\ref{lemma:long-orbit}. Let
$\iota\colon\mathbb{P}^{11}\dasharrow\Pi_{3}$ be a~$G$-equivariant
projection from $\Pi_{7}$, put
$C_{14}=\iota\circ\rho\big(\CKlein\big)$, and identify $\Pi_{3}$
with our~$\mathbb{P}^{3}$. Then $C_{14}\subset\P^3$ is an
irreducible $G$-invariant curve of degree $14$.

Let $H$ be subgroup in $G$ such that
$H\cong\mathbb{Z}_7\rtimes\mathbb{Z}_{3}$. Then there is
a~$H$-invariant subset $\Sigma_{3}\subset \CKlein$ such that
$|\Sigma_{3}|=3$. Note that $\Sigma_{3}$ is a~subset of
the~$G$-orbit of length~$24$, which implies that
$\Sigma_{8}\subset C_{14}$ if and only if $|\iota\circ\rho(\Sigma_{3})|=1$
by Lemma~\ref{lemma:space-orbits-8-24}. Let us show that
$\iota\circ\rho(\Sigma_{3})$ consists of a~single point.

Let $T$ be a~vector subspace in
$H^{0}(\mathcal{O}_{\CKlein}(7\theta))$ that consists of sections
vanishing at the~subset~$\Sigma_{3}$, and let $L_{i}$ be
a~$\hat{G}$-subrepresentation in
$H^{0}(\mathcal{O}_{\CKlein}(7\theta))$ such that $L_{i}\cong
U_{i}$ for $i\in\{4,8\}$. Then $|\iota\circ\rho(\Sigma_{3})|=1$ if
and only if $\mathrm{dim}(L_{4}\cap T)=3$ by the~construction of
the~map $\iota\circ\rho$. Let us show that $\mathrm{dim}(L_{4}\cap
T)=3$.

Take a~subgroup $\hat{H}\subset\hat{G}$ such that
$\phi(\hat{H})=H$. Then it follows from Lemma~\ref{lemma:F21}~that
$L_8\cong V_3\oplus V_3^{\prime}\oplus V_1^{\prime}\oplus
V_1^{\prime\prime}$ and $L_4\cong V_3\oplus V_1$ as
representations of the~group $\hat{H}$, where $V_{3}$,
$V_3^{\prime}$, $V_1$, $V_1^{\prime}$, $V_1^{\prime\prime}$  are
different~irredu\-cible representations of dimensions $3$, $3$,
$1$, $1$, $1$, respectively. Thus, there is an isomorphism $T\cong
V_3\oplus V_3\oplus V_3^{\prime}$, because $T$ does not contain
one-dimensional $\hat{H}$-subrepresentations, since the~curve~$\CKlein$
does not contain $H$-invariant subsets consisting of
$\mathrm{deg}(7\theta)-3=11$ points. Hence $L_{4}\cap T\cong V_3$.
\end{proof}

Let us denote the points in $\Sigma_{8}$ by
$O_{1},O_{2},\ldots,O_{8}$, and let us denote by $\mathcal{Q}$
the~linear system of quadric surfaces in $\mathbb{P}^{3}$ that
pass through $\Sigma_{8}$.

\begin{example}
\label{example:novye-centry} The log~pair $(\P^3, 2\mathcal{Q})$ is canonical.
But $\mathrm{NLCS}(\P^3, 4\mathcal{Q})=\Sigma_8$.
\end{example}

Let $\pi\colon U\to\mathbb{P}^{3}$ be the~blow up of the~subset $\Sigma_{8}$,
let $E_{i}$ be the~exceptional divisor of the~birational morphism $\pi$ such
that $\pi(E_{i})=O_{i}$ for every $i$. Then there is a~commutative~diagram
\begin{equation}
\label{equation:projection-to-plane} \xymatrix{
&U\ar@{->}[dl]_{\pi}\ar@{->}[dr]^{\eta}&\\
\mathbb{P}^{3}\ar@{-->}[rr]_{\psi}&&\mathbb{P}^{2},}
\end{equation}
where $\psi$ is a~rational~map~that~is~given~by~$\mathcal{Q}$, and
$\eta$ is an~elliptic fibration.

\begin{lemma}
\label{lemma:space-C-14-properties} The~curve $C_{14}$ has an
ordinary triple point at every point of the~set~$\Sigma_{8}$,
the~proper transform of the~curve~$C_{14}$ on the~threefold $U$ is
smooth, the~curve $C_{14}$ is smooth outside of the~points of
the~set $\Sigma_{8}$, the~map
$\psi\colon\mathbb{P}^{3}\dasharrow\mathbb{P}^{2}$ induces a~
birational map $C_{14}\dasharrow\psi(C_{14})$, the~curve
$\psi(C_{14})$ is a~smooth curve of genus~$3$ and degree $4$,
the~intersection $C_{6}\cap C_{14}$ is empty.
\end{lemma}

\begin{proof}
Let $\bar{C}_{14}$ and $\bar{Q}$ be the~proper transforms of
the~curve $C_{14}$ and a~general surface in $\mathcal{Q}$ on
the~threefold $U$, respectively. Then
$\mathrm{mult}_{O_{i}}(C_{14})\geqslant |\bar{C}_{14}\cap
E_{i}|\geqslant 3$ by Lemma~\ref{lemma:F21-A4}. Thus
$$
4\geqslant 28-8\big|\bar{C}_{14}\cap E_{i}\big|\geqslant 2\mathrm{deg}\big(C_{14}\big)-\sum_{i=1}^{8}E_{i}\cdot\bar{C}_{14}=\Big(\pi^{*}\big(H\big)-\sum_{i=1}^{8}E_{i}\Big)\bar{C}_{14}=\bar{Q}\cdot\bar{C}_{14}\geqslant 0,%
$$
which implies that $\bar{C}_{14}$ has an ordinary triple point at
every point of the~set $\Sigma_{8}$. Since
$4=2\mathrm{deg}(C_{14})-24=\bar{Q}\cdot\bar{C}_{14}$, we see that
 $\eta(\bar{C}_{14})$ is a~smooth curve of genus~$3$
and degree $4$ by Lemma~\ref{lemma:Klein-small-invariants}.
Therefore $\bar{C}_{14}\cong\eta(\bar{C}_{14})$, which implies
that~$C_{14}$ is smooth outside of the~points of the~set
$\Sigma_{8}$.

Let us show that $C_{6}\cap C_{14}=\varnothing$. Suppose that
$C_{6}\cap C_{14}\ne\varnothing$. Then $|C_{6}\cap
C_{14}|\geqslant 56$ by Lemma~\ref{lemma:sporadic-genera}, since
$\Sigma_{8}\not\subset C_{6}$. Let $S$ be a~general cubic surface
in $\mathbb{P}^{3}$ such that $C_{6}\subset S$. Then
$$
42=S\cdot C_{14}\geqslant\sum_{O\in C_{6}\cap C_{14}}\mathrm{mult}_{O}\big(S\big)\mathrm{mult}_{O}\big(C_{14}\big)\geqslant 56,%
$$
because $C_{14}\not\subset S$. Thus, the~intersection $C_{6}\cap
C_{14}$ is empty.
\end{proof}

\begin{lemma}
\label{lemma:space-curves-through-8-points-degree-14} Let $C$ be
a~ $G$-invariant curve such that $\mathrm{deg}(C)\leqslant 15$ and
$\Sigma_{8}\subset C$. Then $C=C_{14}$.
\end{lemma}

\begin{proof}
Let $\bar{C}$ and $\bar{Q}$ be the~proper transforms of the~curve
$C$ and a~general surface in~$\mathcal{Q}$ on the~threefold $U$,
respectively. Then $\mathrm{mult}_{O_{i}}(C)\geqslant |\bar{C}\cap
E_{i}|\geqslant 3$  by Lemma~\ref{lemma:F21-A4}. Then
$$
6\geqslant 30-8\big|\bar{C}\cap E_{i}\big|\geqslant 2\mathrm{deg}\big(C\big)-\sum_{i=1}^{8}E_{i}\cdot\bar{C}=\Big(\pi^{*}\big(H\big)-\sum_{i=1}^{8}E_{i}\Big)\bar{C}=\bar{Q}\cdot\bar{C}\geqslant 0,%
$$
which implies that $E_{i}\cdot\bar{C}=|\bar{C}\cap E_{i}|=3$.

We may assume that $C$ is a~$G$-orbit of an~irreducible
curve~$\Gamma\subset\mathbb{P}^{3}$. Then
$$
6\geqslant 2\mathrm{deg}\big(C\big)-24=\bar{Q}\cdot\bar{C}=\delta\mathrm{deg}\Big(\eta\big(\bar{C}\big)\Big)%
$$
for some positive integer $\delta$. We may assume that $\delta=1$
if $\bar{C}$ is contracted by $\eta$. Thus, we have
$\mathrm{deg}(C)\in\{12,13,14,15\}$, which implies that either
$\mathrm{deg}(\Gamma)\leqslant 2$ or~\mbox{$\Gamma=C$} (see
Corollary~\ref{corollary:PSL-permutation}).

By Lemma~\ref{lemma:Klein-small-invariants}, we have
$\mathrm{deg}(\psi(C))\not\in\{1,2,3\}$. Then $\mathrm{deg}(C)\in\{12,14,15\}$
and~\mbox{$\delta=1$}.~But
\begin{equation}
\label{equation:8-points-degree-of-curve-2}
6\geqslant 2\mathrm{deg}\big(C\big)-24=\mathrm{deg}\Big(\eta\big(\bar{C}\big)\Big),%
\end{equation}
which implies that $\deg\big(\eta(\bar{C})\big)\in\{0,4,6\}$ by
Lemma~\ref{lemma:Klein-small-invariants}.

Recall that $G$-invariant quartic curve and sextic curve in $\mathbb{P}^{2}$
are irreducible (see Lemma~\ref{lemma:Klein-small-invariants}), which easily
implies using $(\ref{equation:8-points-degree-of-curve-2})$ that either
$\mathrm{deg}(\Gamma)=1$ and \mbox{$\mathrm{deg}(C)=12$}, or $\Gamma=C$.

If $\mathrm{deg}(\Gamma)=1$ and $\mathrm{deg}(C)=12$, then it
follows from $(\ref{equation:8-points-degree-of-curve-2})$ that
$|\eta(\bar{C})|\leqslant 12$, which is impossible, since
$G$-orbit of every point in $\mathbb{P}^{2}$ consists of at least
$21$ points by Lemma~\ref{lemma:Klein-small-orbits}.

We see that $\Gamma=C$. Then $\bar{C}$ is not contracted by
$\eta$, since there is no $G$-invariant point~in~$\mathbb{P}^{2}$.

Suppose that $\mathrm{deg}(C)=15$.~Then $\bar{C}\cong\psi(C)$ and
$\deg(\psi(C))=6$, which immediately implies that $\psi(C)$ is a~smooth  curve
of genus $10$ by Lemma~\ref{lemma:Klein-small-invariants}. Then there is
a~natural monomorphism
$$
U_{4}\cong
H^{0}\Bigg(\mathcal{O}_{U}\Big(\pi^{*}\big(H\big)\Big)\Bigg)\hookrightarrow H^{0}\Bigg(\mathcal{O}_{\bar{C}}\otimes\mathcal{O}_{U}\Big(\pi^{*}\big(H\big)\Big)\Bigg)\cong\mathbb{C}^{6},%
$$
which contradicts  Lemma~\ref{lemma:long-orbit}, since $\SLF$ has no
irreducible two-dimensional representations.

We see that $\mathrm{deg}(C)=14$. Thus $\bar{C}\cong\psi(C)$,
which implies that $\bar{C}$ is a~smooth curve of genus~$3$.
Arguing as in the~proof of Lemma~\ref{lemma:space-C-14}, we see
that $C=\sigma(C_{14})$ for some
$\sigma\in\mathrm{Aut}^{G}(\mathbb{P}^{3})$. But
$\mathrm{Aut}^{G}(\mathbb{P}^{3})=G$ by
Remark~\ref{remark:Blichfeld}. Hence, we must have $C=C_{14}$,
since the~curve $C_{14}$ is $G$-invariant.
\end{proof}

\begin{lemma}
\label{lemma:space-C-14-properties-more} Let $\mathcal{D}$ be
the~linear system consisting of all quintic surfaces in
$\mathbb{P}^{3}$ that contain $C_{14}$. Then $\mathcal{D}$ is not
empty and does not have fixed components, a~general surface in
$\mathcal{D}$ has a~double point in every point of the~set
$\Sigma_{8}$, all curves contained in the~base locus of the~linear
system $\mathcal{D}$ are disjoint from $C_{14}$, a~general surface
in $\mathcal{D}$ is smooth in a~general point of the~curve
$C_{14}$, any two general surfaces in $\mathcal{D}$ are not
tangent to each other along the~curve $C_{14}$.
\end{lemma}

\begin{proof}
Let $\bar{C}_{14}$ be the~proper transform of the~curve $C_{14}$
on the~threefold $U$, and let~$\mathcal{I}$ be~the~ideal sheaf of
the~curve $\bar{C}_{14}$. Put $R=\pi^{*}(5H)-\sum_{i=1}^{8}E_{i}$.
Then
$$
h^{0}\Big(\mathcal{O}_{U}\big(R\big)\otimes\mathcal{I}\Big)\geqslant
h^{0}\Big(\mathcal{O}_{U}\big(R\Big)\Big)-h^{0}\Big(\mathcal{O}_{\bar{C}_{14}}\otimes\mathcal{O}_{U}\big(R\big)\Big)=
48-h^{0}\Big(\mathcal{O}_{\bar{C}_{14}}\otimes\mathcal{O}_{U}\big(R\big)\Big)=4,
$$
which implies that $\mathrm{dim}(\mathcal{D})\geqslant 3$.

Let $D$ be a~general surface in $\mathcal{D}$. Then $D$ is
irreducible, because the~linear system~$\mathcal{D}$ does not have
fixed components by Theorem~\ref{theorem:invariants} and
$\mathcal{D}$ is not composed of a~pencil by
Remark~\ref{remark:pencils}.

Let $\bar{D}$ be its proper transform of the~surface $D$ on
the~threefold $U$, and let $\Gamma$ be a~general fiber of
the~elliptic fibration $\eta$. Then
$$
20-8\mathrm{mult}_{O_{1}}\big(\mathcal{D}\big)=\Big(\pi^{*}\big(5H\big)-\sum_{i=1}^{8}\mathrm{mult}_{O_{i}}\big(\mathcal{D}\big)E_{i}\Big)\cdot\Gamma=\bar{D}\cdot\Gamma\geqslant 0%
$$
which implies that $\mathrm{mult}_{O_{1}}(\mathcal{D})\leqslant
2$. Thus $\mathrm{mult}_{O_{1}}(\mathcal{D})=2$ by
Lemma~\ref{lemma:F21-A4}.

Let $D^{\prime}$ be a~general surface in $\mathcal{D}$ such that
$D\ne D^{\prime}$. Hence there is~\mbox{$\mu\in\N$} such~that $D\cdot
D^{\prime}=\mu C_{14}+B+Z$, where $Z$ is a~curve not contained in
the~base locus of the~linear system $\mathcal{D}$, and $B$ is
a~curve  contained in the~base locus of the~linear system
$\mathcal{D}$ such that $C_{14}\not\subseteq\mathrm{Supp}(B)$.
Then
$$
25=\mathrm{deg}\Big(D\cdot D^{\prime}\Big)=14\mu+\mathrm{deg}\big(B\big)+\mathrm{deg}\big(Z\big)\geqslant 20+\mathrm{deg}\big(Z\big)%
$$
by Lemma~\ref{lemma:jacobian-curve}. Thus, we see that $\mu=1$ and
$\mathrm{deg}(B)+\mathrm{deg}(Z)\leqslant 11$, which implies, in
particular, that the~surfaces $D$ and $D^{\prime}$ are not tangent
to each other along $C_{14}$, since $\mu=1$.

Note that $B$ is $G$-invariant, because $\mathcal{D}$ is
$G$-invariant. Therefore $\mathrm{deg}(B)\geqslant 6$ by
Lemma~\ref{lemma:jacobian-curve}.

If the~base locus of the~linear system $\mathcal{D}$ consists of the~curve
$C_{14}$, then $C_{14}$ is a~scheme-theoretic intersection of surfaces in
$\mathcal{D}$ outside some finite subset of the curve~$C_{14}$, because
$\mu=1$.

Let $\bar{Z}$ and $\bar{B}$ be the~proper transforms of the~curves $Z$ and $B$
on the~threefold $U$, respectively, and let $\bar{Q}$ be a~proper transform of
a~general quadric surface in $\mathcal{Q}$ on the~threefold $U$. Then
\begin{multline*}
0\leqslant
\bar{Q}\cdot\bar{B}=
2\mathrm{deg}\big(B\big)-\sum_{i=1}^{8}E_{i}\cdot\bar{B}\leqslant{}\\
{}\leqslant
22-8\Big|\mathrm{Supp}\big(\bar{B}\big)\cap E_{i}\Big|\leqslant
22-8\times\left\{\aligned%
&3\ \text{if $O_{i}\in \mathrm{Supp}\big(B\big)$},\\
&0\ \text{if $O_{i}\not\in \mathrm{Supp}\big(B\big)$},\\
\endaligned
\right.
\end{multline*}
which implies that $\Sigma_{8}\cap\mathrm{Supp}(B)=\varnothing$.
But
$$
3+\mathrm{mult}_{O_{i}}\big(Z\big)=\mathrm{mult}_{O_{i}}\big(C_{14}\big)+\mathrm{mult}_{O_{i}}\big(B\big)+\mathrm{mult}_{O_{i}}\big(Z\big)=\mathrm{mult}_{O_{i}}\Big(D_{1}\cdot D_{2}\Big)\geqslant 4,%
$$
which implies that $\Sigma_{8}\subset\mathrm{Supp}(Z)$. Hence
$$
2\mathrm{deg}\big(Z\big)-8\geqslant 2\mathrm{deg}\big(Z\big)-\sum_{i=1}^{8}E_{i}\cdot\bar{Z}=\bar{Q}\cdot\bar{Z}\geqslant 0,%
$$
which implies that $\mathrm{deg}(Z)\geqslant 4$, and
$\mathrm{deg}(Z)=4$ if and only if $\bar{Z}$ is contracted by
$\eta$. But $10\leqslant 6+\mathrm{deg}(Z)\leqslant
\mathrm{deg}(B)+\mathrm{deg}(Z)\leqslant 11$, which implies that
either $\mathrm{deg}(B)=7$ and $\mathrm{deg}(Z)=4$, or
$\mathrm{deg}(B)=6$ and $\mathrm{deg}(Z)=5$.

If $\mathrm{deg}(Z)=4$, then $\bar{D}$ is contracted by the~
morphism $\eta$ to a~curve of degree $d$, then $\bar{D}\sim
d(\pi^{*}(2H)-\sum_{i=1}^{8}E_{i})$, which is a~contradiction.
Thus, we see that $\mathrm{deg}(B)=6$ and $\mathrm{deg}(Z)=5$.

By Lemma~\ref{lemma:jacobian-curve}, we have $B=C_{6}$. Therefore
$C_{6}\cap C_{14}=\varnothing$ by
Lemma~\ref{lemma:space-C-14-properties}.
\end{proof}

Let us study some properties of the~subset $\Sigma_{28}\subset\mathbb{P}^{3}$,
which also hold for $\Sigma_{28}^{\prime}$.

\begin{lemma}
\label{lemma:space-curves-throught-28-points} Let $Z$ be
a~$G$-invariant curve in $\mathbb{P}^{3}$ that contains
$\Sigma_{28}$. Then~\mbox{$\mathrm{deg}(Z)\geqslant 16$} if the~set
$\Sigma_{28}$ imposes independent linear conditions on quartic
surfaces in $\mathbb{P}^{3}$.
\end{lemma}

\begin{proof}
Suppose that the~set $\Sigma_{28}$ imposes independent linear conditions on
quartic surfaces~in~$\mathbb{P}^{3}$, and suppose that
$\mathrm{deg}(Z)\leqslant 15$. Let us derive a~contradiction.

Without loss of generality, we may assume that $Z$ is a~$G$-orbit of an
irreducible curve $Z_{1}$.

Put $Z=\sum_{i=1}^{r}Z_{i}$, where $Z_{i}$ is an irreducible curve
in $\mathbb{P}^{3}$ and $r\in\N$. Then
$\mathrm{deg}(Z)=r\mathrm{deg}(Z_{i})\leqslant 15$, which implies
that $r\in\{1,7,8,14\}$ by
Corollary~\ref{corollary:PSL-permutation}. Thus, if $r\ne 1$, then
$\mathrm{deg}(Z_{i})\in\{1,2\}$.

If $\mathrm{deg}(Z_{1})=2$ and $r\ne 1$, then $r=7$, which contradicts
Lemma~\ref{lemma:space-orbits-8-24}. Hence $Z_{1}$ is a~line~if~$r\ne 1$.

Let $\Gamma$ be the~stabilizer subgroup in $G$ of the~curve
$Z_{1}$. If $r=7$, then $\Gamma\cong\mathrm{S}_{4}$, which
implies~that $|Z_{1}\cap\Sigma_{28}|\geqslant 6$, which is
impossible, since $\Sigma_{28}$ impose independent linear
conditions on quartic surfaces.

Let $\hat{\Gamma}$ be the~smallest subgroup of the~group $\hat{G}$ such that
$\phi(\hat{\Gamma})=\Gamma$. Then
\begin{itemize}
\item  if $r=8$, then $\hat{\Gamma}\cong 2.(\mathbb{Z}_7\rtimes\mathbb{Z}_{3})$, which contradicts Lemma~\ref{lemma:F21}.%

\item if $r=14$, then $\hat{\Gamma}\cong 2.\mathrm{A}_{4}$, which contradicts
Lemma~\ref{lemma:SL-2-7-subgroups}.
\end{itemize}

Thus, we see that the~curve $Z$ is irreducible.

By Lemma~\ref{lemma:long-orbit}, the~curve $Z$ must be singular at
every point of the~set $\Sigma_{28}$, which implies that $Z\ne
C_{14}$ by Lemma~\ref{lemma:space-C-14-properties}. Hence
$\Sigma_{8}\not\subset Z$ by
Lemma~\ref{lemma:space-curves-through-8-points-degree-14}.

Let $A$ be a~point in $\Sigma_{28}$. Then there exists a~quartic surface
$S\subset\mathbb{P}^{3}$ such that $S$ contain all points of the~set
$\Sigma_{28}\setminus\{A\}$ and the~surface $S$ does not contain the~point~$A$.
Therefore
$$
4\mathrm{deg}\big(Z\big)=S\cdot Z\geqslant
\sum_{O\in \Sigma_{28}\setminus\{A\}}\mathrm{mult}_{O}\big(Z\big)\geqslant
\sum_{O\in \Sigma_{28}\setminus\{A\}}2\geqslant 54,%
$$
which implies that $\mathrm{deg}(Z)\geqslant 14$. Thus, either
$\mathrm{deg}(Z)=14$ or $\mathrm{deg}(Z)=15$.

If $\mathrm{deg}(Z)=15$, then $Z\subset F_{4}\cap F_{6}\cap
F_{8}^{\prime}$ by Lemma~\ref{lemma:long-orbit}, since $60$, $90$
and $120$ are not equal to
$24n_{1}+42n_{2}+56n_{3}+84n_{4}+168n_{5}$ for any non-negative
integers $n_{1}$, $n_{2}$, $n_{3}$, $n_{4}$ and $n_{5}$. Then
$\mathrm{deg}(Z)=14$ by Lemma~\ref{lemma:invariants-properties}.

Let $\bar{Z}$ be the~normalization of the~curve $Z$, and let $g$
be the~genus of the~curve $\bar{Z}$. Then
\begin{equation}
\label{equation:degree-14-genus}
g\leqslant 66-\big|\mathrm{Sing}\big(Z\big)\big|\leqslant 38,%
\end{equation}
because the~projection from a~general points of the~curve $Z$
gives us a~birational isomorphism between $Z$ and a~plane curve of
degree $13$ with at least $|\mathrm{Sing}(Z)|$ singular points.

Note that $G$ naturally acts on both curves $Z$ and $\bar{Z}$. Let us show that
$Z\setminus\Sigma_{28}$ is smooth.

Suppose that $\mathrm{Sing}(Z)\ne\Sigma_{28}$. Let $\Lambda$ be
a~$G$-orbit of a~point in $\mathrm{Sing}(Z)\setminus\Sigma_{28}$.
Then $|\Lambda|\leqslant 66-|\Sigma_{28}|=38$ by
$(\ref{equation:degree-14-genus})$, and $|\Lambda|\in\{24,28\}$ by
Lemma~\ref{lemma:space-orbits-8-24}. If $\Lambda=C_{6}\cap F_{4}$,
then
$$
42=S\cdot Z\geqslant\sum_{O\in \Lambda}\mathrm{mult}_{O}\big(S\big)\mathrm{mult}_{O}\big(Z\big)\geqslant 48,%
$$
where $S$ is a~general cubic surface such that $C_{6}\subset S$.
Thus, we see that $|\Lambda|\ne 24$ by
Lemma~\ref{lemma:space-orbits-8-24}, which implies that
$|\Lambda|=28$. Therefore, it follows from
$(\ref{equation:degree-14-genus})$ and
Lemma~\ref{lemma:sporadic-genera} that $g\in\{3,8,10\}$ and
the~points of the~set $\Lambda\cup\Sigma_{28}$ must be singular
points of the~curve $Z$ of multiplicity two, which implies that
$\bar{Z}$ has at least two $G$-orbits consisting of $56$ points,
which contradicts Lemma~\ref{lemma:sporadic-genera}.

Thus, we see that $Z$ is smooth outside of the~set $\Sigma_{28}$.

Let $B$ be a~sufficiently general point of the~curve $Z$, let
$\mathcal{M}$ be a linear system consisting of all quartic
surfaces in $\mathbb{P}^{3}$ that contain the~set $\Sigma_{28}\cup
B$. Then one has \mbox{$\mathrm{dim}(\mathcal{M})\geqslant 35-29=6$}.

Let $M$ be a~general surface in $\mathcal{M}$. If $Z\not\subset
M$, then
$$
56=M\cdot Z\geqslant\mathrm{mult}_{B}\big(M\big)\mathrm{mult}_{B}\big(Z\big)+\sum_{O\in\Sigma_{28}}\mathrm{mult}_{O}\big(M\big)\mathrm{mult}_{O}\big(Z\big)\geqslant 57,%
$$
which is a~contradiction. Thus, the~curve $Z$ is contained in the~base locus of
the~linear system~$\mathcal{M}$, which implies that the~linear system
$\mathcal{M}$ is $G$-invariant, because $Z$ is \mbox{$G$-invariant}.

The linear system $\mathcal{M}$ does not have fixed components by
Theorem~\ref{theorem:invariants} and $\mathcal{M}$ is not composed
of a~pencil by Remark~\ref{remark:pencils}. This implies, in
particular, that the~surface $M$ is irreducible.

Let $M^{\prime}$ be another general surface in the~linear system
$\mathcal{M}$. Put $M\cdot M^{\prime}=\mu Z+\Upsilon$, where
$\mu\in\N$, and $\Upsilon$ is an~effective one-cycle such that
$Z\not\subset\mathrm{Supp}(\Upsilon)$. Then
$$
16=\mathrm{deg}\Big(M\cdot
M^{\prime}\Big)=14\mu+\mathrm{deg}\big(\Upsilon\big),
$$
which implies that $\mu=1$, and the~base locus of the~linear
system $\mathcal{M}$ contain no curves except~$Z$, because there
are no $G$-invariant lines or conics in $\mathbb{P}^{3}$. Thus
$\mathrm{mult}_{Z}(\mathcal{M})=1$~and~$\mathrm{deg}(\Upsilon)=2$,
which implies, in particular, that the~surface $M$ is uniruled.

Let $M^{\prime\prime}$ be another general surface in the~linear
system $\mathcal{M}$. Then the intersection $M\cap M^{\prime}\cap
M^{\prime\prime}$ consists of the curve $C$ and the intersection
$\Upsilon\cap M^{\prime\prime}$, which immediately implies that
the~base locus of the~linear system $\mathcal{M}$ consists of at
most $8$~points outside of the~curve $Z$. Thus, it follows from
Lemma~\ref{lemma:space-orbits-8-24} that either the~base locus of
the~linear system $\mathcal{M}$ consists of the~curve $Z$, or
the~base locus of the~linear system~$\mathcal{M}$ consists of
the~curve $Z$ and the~set $\Sigma_{8}$.

Let $Q$ be a~general surface in the~linear system $\mathcal{Q}$. If
$\Sigma_{8}\subset M\cap M^{\prime}$, then
\begin{multline*}
32=Q\cdot M\cdot M^{\prime}=Q\cdot Z+\sum_{O\in\Sigma_{8}}\mathrm{mult}_{O}\big(Q\big)\mathrm{mult}_{O}\big(\Upsilon\big)=\\
=28+\sum_{O\in\Sigma_{8}}\mathrm{mult}_{O}\big(Q\big)\mathrm{mult}_{O}\big(\Upsilon\big)\geqslant 36,%
\end{multline*}
which implies that $\Sigma_{8}\not\subset M\cap M^{\prime}$. Hence
$Z$ is the~base locus of the~linear system $\mathcal{M}$.

By adjunction formula, we see that the~singularities of
the~surface $M$ are not canonical, which implies that
$(\mathbb{P}^{3},\mathcal{M})$ is not canonical by
\cite[Theorems~4.8.1 and~7.9]{Ko97} (cf.
\cite[Theorem~7.5]{Ko97}).~But
$Z\not\in\mathbb{NCS}(\mathbb{P}^{3},\mathcal{M})$, because
$\mathrm{mult}_{Z}(\mathcal{M})=1$. Thus, the~set
$\mathbb{NCS}(\mathbb{P}^{3},\mathcal{M})$ contains a~point $P\in
Z$. Thus
$$
\mathrm{mult}_{P}\big(\Upsilon\big)\ge 4-\mathrm{mult}_{P}\big(Z\big)=
\left\{\aligned%
&2\ \text{if $P\in\Sigma_{28}$},\\
&3\ \text{if $P\not\in\Sigma_{28}$},\\
\endaligned
\right.
$$
since $\mathrm{mult}_P(\mathcal{M})\ge 2$. Let $\Theta$ be the~$G$-orbit
of the~point $P$. Then $|\Theta|\geqslant 28$ by
Lemma~\ref{lemma:space-orbits-8-24}, and
$$
8=M\cdot M^{\prime}\cdot M^{\prime\prime}-M^{\prime\prime}\cdot Z=
M^{\prime\prime}\cdot\Upsilon\geqslant\sum_{O\in\Theta}
\mathrm{mult}_{O}\big(M^{\prime\prime}\big)
\mathrm{mult}_{O}\big(\Upsilon\big)\ge 56,%
$$
which is a~contradiction.
\end{proof}

Let $\mathcal{R}$ be a~linear system consisting of all quartic surfaces in
$\mathbb{P}^{3}$ that pass through~$\Sigma_{28}$.

\begin{lemma}
\label{lemma:space-28-points} If the~set $\Sigma_{28}$ imposes
independent linear conditions on quartic surfaces in
$\mathbb{P}^{3}$, then every curve in the~base locus of the~linear
system $\mathcal{R}$ contains no points of the~set $\Sigma_{28}$.
\end{lemma}

\begin{proof}
Let $C$ be an irreducible curve in $\mathbb{P}^{3}$ such that $C$
is contained in the~base locus of the~linear system $\mathcal{R}$,
let $Z$ be the~$G$-orbit of the~curve $C$, let $R_{1}$ and $R_{2}$
be general surfaces in $\mathcal{R}$.~Put $R_{1}\cdot R_{2}=\mu
Z+\Upsilon$, where $\mu\in\N$, and $\Upsilon$ is an~effective
one-cycle such that $C\not\subset\mathrm{Supp}(\Upsilon)$.~Then
$16=\mathrm{deg}(R_{1}\cdot
R_{2})=\mu\mathrm{deg}(Z)+\mathrm{deg}(\Upsilon)$, which implies
that $\mathrm{deg}(Z)\leqslant 15$ by Remark~\ref{remark:pencils}.
Now it follows from
Lemma~\ref{lemma:space-curves-throught-28-points} that
$\Sigma_{28}\not\subset Z$ and $C\cap\Sigma_{28}=\varnothing$ if
the~points of the~set $\Sigma_{28}$ impose independent linear
conditions on quartic surfaces in $\mathbb{P}^{3}$.
\end{proof}

We do not know whether or not $\Sigma_{28}$ imposes independent linear
conditions on quartic surfaces.

\section{Compactified moduli space of $(1,7)$-polarized abelian surfaces}
\label{section:Iskovskikh}

Let $\CKlein$ be a~curve in $\mathbb{P}^{2}$ that is given by
the~equation
\begin{equation}\label{eq:another-Klein}
x^4+y^4+z^4+3\epsilon\Big(x^2y^2+x^2z^2+y^2z^2\Big)=0\subset\mathbb{P}^{2}\cong\mathrm{Proj}\Big(\mathbb{C}\big[x,y,z\big]\Big),
\end{equation}
where $\epsilon=-1/2+\sqrt{-7}/2$. Then $\CKlein$ is isomorphic to
the~curve described in Example~\ref{example:V22} (see
\cite[page~55]{Beauty}).

\begin{remark}
\label{remark:Klein-curve-Hessian} The Hessian curve of the~curve $\CKlein$ is
a~smooth sextic curve (see \cite[Example~6.1.1]{DoKa93}).
\end{remark}

\begin{remark}\label{remark:Klein-not-Luroth}
The quartic $\CKlein$ is not degenerate (see
\cite[Definition~2.8]{DoKa93}).
To see this apply~\cite[Theorem~6.12.2]{DoKa93} (keeping 
in mind~\cite[Definition~6.12.1]{DoKa93}) and 
use~\cite[Corollary~6.6.3]{DoKa93}, \cite[Example~6.7]{DoKa93}
and Remark~\ref{remark:Klein-curve-Hessian} above. 
\end{remark}

Put $\bar{\epsilon}=-1/2-\sqrt{-7}/2$. Let
$\psi\colon\mathrm{SL}_{3}(\mathbb{C})\to
\mathrm{Aut}(\mathbb{P}^{2})\cong\mathrm{PGL}_{3}(\mathbb{C})$
be a~natural projection. Put
\begin{multline*}
A=\left(%
\begin{array}{ccc}
-1&0&0\\
0&0&-1\\
0&-1&0\\
\end{array}%
\right),\
B=\left(%
\begin{array}{ccc}
0&1&0\\
0&0&1\\
1&0&0\\
\end{array}%
\right),\\
C=\left(%
\begin{array}{ccc}
0&1&0\\
1&0&0\\
0&0&-1\\
\end{array}%
\right),\
D=\frac{1}{2}\left(%
\begin{array}{ccc}
-1&1&\bar{\epsilon}\\
\epsilon &\epsilon&0\\
-1&1&-\bar{\epsilon}\\
\end{array}%
\right),
\end{multline*}
put $G=\psi(\langle A,B,C,D\rangle)$, and let
$X=\mathrm{VSP}(\CKlein,6)$ (see Definition~\ref{definition:VSP}).
By \cite[p.~55]{Beauty}, we have $\mathrm{Aut}(\CKlein)\cong
G\cong\langle A,B,C,D\rangle\cong\mathrm{PSL}_{2}(\F_7)$.

\begin{theorem}[{\cite[Theorem~4.4]{MS01}}]
\label{theorem:V22-Klein-smooth} The~threefold $X$ is a~smooth
Fano threefold such that $\mathrm{Pic}(X)=\mathbb{Z}[-K_{X}]$ and
$(-K_X)^3=22$.
\end{theorem}

The~action of the~group $G$ on the~plane $\mathbb{P}^{2}$ induces
its~natural action on the~threefold $X$. Therefore, the~vector
space $H^{0}(\mathcal{O}_{X}(-K_{X}))$ has a natural structure of
a~$14$-dimensional representation of the~group~$G$.

\begin{theorem}
\label{theorem:v22-U-14} In the~notation of
Appendix~\ref{section:characters}, one has
$$H^{0}(\mathcal{O}_{X}(-K_{X}))\cong I\oplus W_6\oplus W_7.$$
\end{theorem}

\begin{proof}
We may assume that $W\cong W_3$. By Theorem~\ref{theorem:U14}, one
has
$$
U_7\cong \mathrm{Sym}^3
\Big(W_3^{\vee}\Big)\Big\slash W_3\cong W_7\cong U_7^{\vee},%
$$
where all isomorphisms are isomorphisms of $G$-representations.
Hence, we have
$$
U_{14}\cong\Lambda^3\big(U_7\big)\Big\slash\Big(W_3^{\vee}\otimes U_7^{\vee}\Big)\cong\Lambda^3\big(W_7\big)\Big\slash\Big(W_3^{\vee}\otimes W_7\Big)\cong U_{14}^{\vee},%
$$
which implies the~required assertion by
Corollary~\ref{corollary:characters}.
\end{proof}

\begin{corollary}
\label{corollary:v22-K3} There is a~unique $G$-invariant surface
in $|-K_{X}|$.
\end{corollary}

Let us identify $X$ with its anticanonical image in $\mathbb{P}^{13}$ (see
Appendix~\ref{section:mukai}).

\begin{lemma}\label{lemma:V22-Hilbert-scheme-of-lines}
Let $Q\subset X$ be a~surface swept out by the~lines contained in
$X$. Then $Q\sim -2K_X$, and $Q$ is irreducible.
\end{lemma}

\begin{proof}
By \cite{Pr90}, the~surface $Q$ is reduced and $Q\sim -2K_X$.
Suppose that $Q$ is reducible. Then $Q=Q_1\cup Q_2$, where $Q_{1}$
and $Q_{2}$ are irreducible surfaces such that $Q_{1}\ne Q_2$.
Then $Q_1\sim Q_{2}\sim -K_{X}$, and both $Q_{1}$ and $Q_{2}$ must
be $G$-invariant. On the other hand, there is a unique
$G$-invariant surface in $|-K_{X}|$ is by
Lemma~\ref{theorem:v22-U-14}, which is a contradiction.
\end{proof}

\medskip

In the~remaining part of this section we are going to prove the~
following result.

\begin{theorem}
\label{theorem:v22-orbits} Let $\Sigma$ be a~$G$-orbit of a~point
in $X$ such that $|\Sigma|\leqslant 20$. Then
$|\Sigma|\in\{8,14\}$, and if $|\Sigma|=8$, then $\Sigma$ is
unique.
\end{theorem}

\begin{remark}
\label{remark:v22-14-points} There exist finitely many points in
$X$ whose $G$-orbits consist of $14$ points.
\end{remark}

Before proving Theorem~\ref{theorem:v22-orbits}, let us use
Theorem~\ref{theorem:v22-orbits} to prove the~following result.

\begin{lemma}
\label{lemma:v22-K3} Let $F$ be the~unique $G$-invariant surface
in $|-K_{X}|$. Then $F$ is smooth.
\end{lemma}

\begin{proof}
The minimal resolution of the~surface $F$ admits a~faithful action
of the~group $\PSLF$, which implies that $F$ is smooth by
\cite[Claim~2.1]{Og02} if it has at most canonical singularities.
Let us show that $F$ has at most canonical singularities (Du Val
singularities). It is well-known that $F$ has canonical
singularities if and only if the~surface $F$ has
rational~sin\-gu\-la\-rities (see \cite[Theorem~1.11]{Ko97}). By
Theorem~\ref{theorem:V22-auxiliary}, the~log pair $(X,F)$ is log
canonical.

Suppose that $F$ has worse than canonical singularities. Let $S$
be a~minimal center in $\mathbb{LCS}(X,F)$. Then $S\ne F$ by
Theorem~\ref{theorem:Kawamata}, because $F$ has worse than
canonical singularities. Hence, either $S$ is a curve or a point
by Theorem~\ref{theorem:Kawamata}. Let $Z$ be the~$G$-orbit of
the~subvariety $S$. Then $Z\subset F$.

Choose any $\epsilon\in\mathbb{Q}$ such that $2\gg\epsilon>1$.
Then arguing as in the~proof of \cite[Theorem~1.10]{Kaw97}, we can
find a~$G$-in\-va\-riant $\mathbb{Q}$-divisor $D$ such that
$D\sim_{\mathbb{Q}} \epsilon F$, the~log pair $(X,D)$ is log
canonical, and every minimal center in $\mathbb{LCS}(X,D)$ is an
irreducible component of the~subvariety $Z$ (see
Lemma~\ref{lemma:Kawamata-Shokurov-trick}).

Let $H$ be a~sufficiently general hyperplane section of the~threefold $X$, and
let~$\mathcal{I}_{Z}$ be the~ideal sheaf of the~subvariety $Z$. Then there is
an exact~sequence of \mbox{$G$-representations}
$$
0\to H^{0}\Big(\mathcal{O}_{X}\big(H\big)\otimes\mathcal{I}_{Z}\Big)\to H^{0}\Big(\mathcal{O}_{X}\big(H\big)\Big)\to H^{0}\Big(\mathcal{O}_{Z}\otimes\mathcal{O}_{X}\big(H\big)\Big)\to 0%
$$
by Theorem~\ref{theorem:Shokurov-vanishing}. Put
$q=h^{0}(\mathcal{O}_{X}(H)\otimes\mathcal{I}_{Z})$. One has
\begin{equation}
\label{equation:v22-K3-equalitu}
h^{0}\Big(\mathcal{O}_{Z}\otimes\mathcal{O}_{X}\big(H\big)\Big)=14-q\in\big\{6,7,13\big\},%
\end{equation}
because $q\in\{1,7,8\}$ by Theorem~\ref{theorem:v22-U-14}, since $Z\subset F$.
Hence $S$ is not a~point by Theorem~\ref{theorem:v22-orbits}.

We see that $S$ is a~curve. Moreover, $S$ is a~smooth curve of
genus $g$ such that $-K_{X}\cdot S=\mathrm{deg}(S)\geqslant 2g-1$
by Theorem~\ref{theorem:Kawamata}. By Remark~\ref{remark:centers},
the~curve $Z$ is a~disjoint union of smooth irreducible curves.
Let $r$ be the~number of connected components of the~curve $Z$.
Put $d=\mathrm{deg}(S)$. Then
\begin{equation}
\label{equation:v22-K3-equality-special} g\leqslant d-g+1\leqslant
r\big(d-g+1\big)=14-q\in\big\{6,7,13\big\}%
\end{equation}
by $(\ref{equation:v22-K3-equalitu})$, which implies that
$d\leqslant 13+g\leqslant 27$.

Let us show that $r=1$. Suppose that $r\geqslant 2$. Then
$r\geqslant 7$ and $r\ne 13$ by
Corollary~\ref{corollary:PSL-permutation}. Thus $r=7$ and $d-g+1=1$
by $(\ref{equation:v22-K3-equality-special})$, which implies that
$d=g$. But $d\geqslant 2g-1$, so that $d=g=1$, which is absurd.

We see that $r=1$. There is a~natural faithful action of the~group
$G$ on the~curve~$S$. But $d=13+g-q\geqslant 2g-1$, by
$(\ref{equation:v22-K3-equality-special})$, which implies that
$g\leqslant 14-q$. Hence $g\in\{3,8,10\}$ by
Lemma~\ref{lemma:sporadic-genera}.

Let us show that $g=3$. Suppose that $g\ne 3$. Then
$g\in\{8,10\}$. It follows from~$(\ref{equation:v22-K3-equality-special})$
that $q\leqslant
14-g\leqslant 6$, which implies that $q=1$, because
$q\in\{1,7,8\}$.~Hence, it follows from
$(\ref{equation:v22-K3-equality-special})$
that~$d\in\nolinebreak\{20,22\}$, which is impossible by
Lemmas~\ref{lemma:g-8-d-7} and \ref{lemma:g-10-d-3}.

Thus, we have $g=3$, so that $d=16-q$ by
$(\ref{equation:v22-K3-equality-special})$, where $q\in\{1,7,8\}$.
By Theorem~\ref{theorem:Dolgachev}, there is a~$G$-invariant line
bundle $\theta\in\mathrm{Pic}(S)$ of degree $2$ such that
$\mathrm{Pic}^{G}(S)=\langle\theta\rangle$, which implies that
$q=8$ and $H\vert_{S}\sim 2K_{S}$, because $d=16-q$. Thus $d=8$.

Let $Q$ be a~surface in $X$ that is spanned by lines. Then $Q$ is
$G$-invariant. If $S\not\subset Q$, then $|Q\cap S|\leqslant
2\mathrm{deg}(S)=16$, which is impossible by
Lemma~\ref{lemma:long-orbit}. Thus, we see that $S\subset Q\cap
F$. By Lemma~\ref{lemma:V22-Hilbert-scheme-of-lines}, the~surface
$Q$ is irreducible, and $Q\sim-2K_{X}$. Let $P$ be a~general point
in $S$. There exists a~line $L\subset X$ such that $P\in L$. If
$L\not\subset F$, then
$$
1=\mathrm{deg}\big(L\big)=F\cdot L\geqslant\mathrm{mult}_{P}\big(F\big)\mathrm{mult}_{P}\big(L\big)\geqslant \mathrm{mult}_{S}\big(F\big)\geqslant 2,%
$$
which simply means $L\subset F$. Since $g\ne 0$, the~generality of
the~point $P\in S$ implies that the~surface $F$ is swept out by
lines,  which is impossible since $Q\not\subset F$.
\end{proof}

\begin{remark}
\label{remark:8-points-not-degenerate} Let $\Sigma$ be
a~$G$-invariant subset of the~threefold $X$ such that we have
$|\Sigma|\in\{8,14\}$, and let~$F$~be~the~unique $G$-invariant
surface in $|-K_{X}|$. Then $F$ is smooth by
Lemma~\ref{lemma:v22-K3}. Since $G$ acts symplectically on $F$, it
follows from \cite{Xiao} that $\Sigma\cap F=\varnothing$.
\end{remark}

Now we are going to prove Theorem~\ref{theorem:v22-orbits}. Put
$$
\square=X\Big\backslash\Bigg\{ \Gamma \in \mathrm{Hilb}_{6} \Big(\check{\mathbb{P}}^2\Big)\ \Big|\ \ \Gamma\ \text{is polar to the~curve}\ C \Bigg\} \subset\mathrm{Hilb}_{6}\Big(\check{\mathbb{P}}^2\Big).%
$$

\begin{lemma}
\label{lemma:v22-orbit-8-points} Let $\Sigma$ be a~$G$-orbit in $\square$ such
that $|\Sigma|\leqslant 20$. Then \mbox{$|\Sigma|=8$} and~$\Sigma$~is~unique.
\end{lemma}

\begin{proof}
It follows from~\cite[\S~2.3]{MeRa03} that there exists
an~effective one-cycle
$$m_{1}L_{1}+\ldots+m_{r}L_{r}\sim\mathcal{O}_{\mathbb{P}^{2}}(6)$$
on $\mathbb{P}^{2}$ that corresponds to a~point in $\Sigma$, where
$L_{i}$ are lines in $\mathbb{P}^{2}$ and $m_{i}$ are positive
integers.

Without loss of generality, we may assume that $m_{1}\geqslant\ldots\geqslant
m_{r}$. Then
$$
\big(m_{1},\ldots,m_{r}\big)\in\Big\{\big(2,1,1,1,1\big),\big(3,1,1,1\big),\big(2,2,1,1\big),\big(2,2,2\big),\big(4,2\big)\Big\}%
$$
by \cite[Theorem~1.1]{MeRa03} (cf. \cite[Theorem~A]{Ma04}). Thus,
it follows from $|\Sigma|\leqslant 20$ that
$(m_{1},\ldots,m_{r})=(2,2,2)$, because the~smallest $G$-invariant
subset in $\mathbb{P}^{2}$ has at least~$21$ points by
Lemma~\ref{lemma:Klein-small-orbits}.

Let us consider the~lines $L_{1}$, $L_{2}$, $L_{3}$ as points in
the~dual projective plane $\check{\mathbb{P}}^2$ with a~natural
action of the~group $G$, and let $\hat{\CKlein}$ be the~unique
$G$-invariant quartic curve in $\check{\mathbb{P}}^2$. Then
$\hat{\CKlein}\cap\{L_{1},L_{2},L_{3}\}\ne\varnothing$ by
\cite[Proposition 3.13]{MeRa03}, which easily implies that
$|\Sigma|=8$ and $\Sigma$ is unique by
Lemma~\ref{lemma:sporadic-genera}.
\end{proof}

Therefore, to prove Theorem~\ref{theorem:v22-orbits} it is enough to consider
only $G$-orbits of points in~\mbox{$X\setminus\square$},
which are  points in $X$ that
can be represented by polar hexagons to the~curve~$\CKlein$.

Suppose that the~assertion of Theorem~\ref{theorem:v22-orbits} is false. Let us
derive a contradiction.

By Lemmas~\ref{lemma:v22-orbit-8-points} and~\ref{lemma:PSL-maximal-subgroups},
there are six lines $L_{1},\ldots,L_{6}$ on $\mathbb{P}^{2}$~such~that
\begin{equation}
\label{equation:polar}
F\big(x,y,z\big)=L_{1}^4\big(x,y,z\big)+\cdots+L_{6}^4\big(x,y,z\big),
\end{equation}
where $L_{i}(x,y,z)$ is a~linear form such that $L_{i}$ is given
by $L_{i}(x,y,z)=0$, and the~stabilizer subgroup of the~hexagon
$\sum_{i=1}^{6}L_{i}$ in the~group $G$ is isomorphic either to
$\mathbb{Z}_7\rtimes\mathbb{Z}_{3}$ or to $\mathrm{S}_4$.

\begin{remark}
\label{remark:6-lines-different} The lines $L_{1},\ldots,L_{6}$
are distinct 
since the quartic~$\CKlein$ is not degenerate 
(see Remark~\ref{remark:Klein-not-Luroth}).
\end{remark}

\begin{lemma}\label{lemma:v22-non-trivial}
Let $g\in G$ be an~element of order $7$. Then
$\sum_{i=1}^{6}L_{i}$ is not $g$-invariant.
\end{lemma}

\begin{proof}
Suppose that $\sum_{i=1}^{6}L_{i}$ is $g$-invariant. Then
$g(L_{i})=L_{i}$ for every $i\in\{1,\ldots,6\}$. Take
$M\in\mathrm{SL}_{3}(\mathbb{C})$ such that $\psi(M)=g^{-1}$ and
$M^7=1$. Then $L_{i}((x,y,z)M)=\lambda_{i}L_{i}(x,y,z)$ for some
$\lambda_{i}\in\mathbb{C}$ for every $i\in\{1,\ldots,6\}$, i. e.
$\lambda_{1},\ldots,\lambda_{6}$ are eigenvalues of the~matrix
$M$.

Since $g$ has order $7$, we have
$\lambda_{1}^{7}=\ldots=\lambda_{6}^{7}=1$. We may assume that
$$\lambda_{1}=\lambda_{2}=\lambda_{3}=\lambda_{4}=\lambda_{5},$$
since $L_{1},\ldots,L_{6}$ are different lines. Then
$\lambda_{6}\ne \lambda_{1}$ by $(\ref{equation:polar})$. Thus one obtains
\mbox{$|L_{1}\cap L_{2}\cap L_{3}\cap L_{4}\cap L_{5}|=1$} and $L_{1}\cap
L_{2}\cap L_{3}\cap L_{4}\cap L_{5}\not\in L_{6}$. Therefore
$$
\sum_{i=1}^{6}L_{i}^4\big(x,y,z\big)=F\Big(\big(x,y,z\big)M\Big)=\lambda_{6}^{4}L_{6}^4\big(x,y,z\big)+\lambda_{1}^{4}\sum_{i=1}^{5}L_{i}^4\big(x,y,z\big),
$$
by $(\ref{equation:polar})$, which easily implies that
$\lambda_{1}^{4}=\lambda_{6}^{4}=1$. Hence
$\lambda_{1}=\lambda_{6}=1$, which is a~contradiction.
\end{proof}

Let $\Gamma$ be a~stabilizer subgroup in $G$ of the~hexagon
$\sum_{i=1}^{6}L_{i}$. Then $\Gamma\cong\mathrm{S}_{4}$ by
Lemma~\ref{lemma:v22-non-trivial}. We may assume that
$\Gamma=\langle\psi(A), \psi(B), \psi(C)\rangle$, because $G$ has
two subgroups isomorphic to $\SS_4$ up to a conjugation (see
Lemma~\ref{lemma:PSL-maximal-subgroups}), which are switched by
an~outer automorphism in $\mathrm{Aut}(G)$ that can be realized as
a~complex conjugation in appropriate coordinates.

\begin{lemma}
\label{lemma:v22-transitively} The~group $\Gamma$ acts
transitively on the~lines $L_{1},L_{2},L_{3},L_{4},L_{5},L_{6}$.
\end{lemma}

\begin{proof}
Since the~action of the~group $\Gamma$ on $\mathbb{P}^{2}$ comes
from an~irreducible three-dimensional representation of the~ group
$\Gamma$, the~$\Gamma$-orbit of every line in $\mathbb{P}^{2}$
contains at least $3$ lines. But the~only $\Gamma$-invariant
$3$-tuple of lines in $\mathbb{P}^{2}$ is given by $xyz=0$.
\end{proof}

Let $\Omega$ be a~stabilizer subgroup in $\Gamma$ of the~line
$L_{1}$. Then either
$\Omega\cong\mathbb{Z}_{2}\times\mathbb{Z}_{2}$
or~\mbox{$\Omega\cong\mathbb{Z}_{4}$}.

\begin{lemma}
\label{lemma:Z4} The~group $\Omega$ is not isomorphic to
$\mathbb{Z}_{4}$.
\end{lemma}

\begin{proof}
Suppose that $\Omega\cong\mathbb{Z}_{4}$. Then we may assume that
$\Omega$ is generated $\psi(CB)$. One has
$$
CB=\left(%
\begin{array}{ccc}
0&0&1\\
0&1&0\\
-1&0&0\\
\end{array}%
\right),
$$
and the~eigenvalues of the~matrix $CB$ are $1$ and $\pm
\sqrt{-1}$.

Note that $(0,1,0)$ is the~eigenvector of the~matrix $CB$ that
corresponds to the~eigenvalue $1$, and $(\mp\sqrt{-1},0,1)$ is its
eigenvector that corresponds to the~eigenvalue $\pm\sqrt{-1}$.
Thus $L_{1}(x,y,z)=\mu(x\pm\sqrt{-1}z)$, where $\mu$ is a~non-zero
complex number. The~hexagon $\sum_{i=1}^{6}L_{i}$ is given by
$$
\Big(x-\sqrt{-1}z\Big)\Big(x+\sqrt{-1}z\Big)\Big(y+\sqrt{-1}x\Big)\Big(y-\sqrt{-1}x\Big)\Big(z+\sqrt{-1}y\Big)\Big(z-\sqrt{-1}y\Big)=0,
$$
which implies that the~quartic form
$x^4+y^4+z^4+3\epsilon(x^2y^2+x^2z^2+y^2z^2)$ is equal to
\begin{multline*}
\mu_{1}\Big(x-\sqrt{-1}z\Big)^{4}+\mu_{2}\Big(x+\sqrt{-1}z\Big)^{4}+\mu_{3}\Big(y+\sqrt{-1}x\Big)^{4}+\\
+\mu_{4}\Big(y-\sqrt{-1}x\Big)^{4}+\mu_{5}\Big(z+\sqrt{-1}y\Big)^{4}+\mu_{6}\Big(z-\sqrt{-1}y\Big)^{4}
\end{multline*}
for some
$(\mu_{1},\mu_{2},\mu_{3},\mu_{4},\mu_{5},\mu_{6})\in\mathbb{C}^{6}$.
In particular, we obtain a~system of linear equations
$$
\left\{\aligned%
&4\mu_{3}+4\mu_{4}=1-\sqrt{-7},\\
&4\mu_{1}+4\mu_{2}=1-\sqrt{-7},\\
&4\mu_{5}+4\mu_{6}=1-\sqrt{-7},\\
&\mu_1+\mu_2+\mu_{3}+\mu_{4}=1,
\endaligned
\right.
$$
which is inconsistent.
\end{proof}

Thus, we see that $\Omega\cong\mathbb{Z}_{2}\times\mathbb{Z}_{2}$.

\begin{lemma}
\label{lemma:Z2-Z2} The~subgroup $\Omega$ is not contained in
a~subgroup of the~group $G$ isomorphic~to~$\mathrm{A}_{4}$.
\end{lemma}

\begin{proof}
Suppose that $\Omega$ is contained in a~subgroup of the~group $G$
that is isomorphic to $\mathrm{A}_{4}$. Then
$$
\Omega=\left\{\left(%
\begin{array}{ccc}
1&0&0\\
0&1&0\\
0&0&1\\
\end{array}%
\right), \left(\begin{array}{ccc}
1&0&0\\
0&-1&0\\
0&0&-1\\
\end{array}%
\right), \left(\begin{array}{ccc}
-1&0&0\\
0&-1&0\\
0&0&1\\
\end{array}%
\right), \left(\begin{array}{ccc}
-1&0&0\\
0&1&0\\
0&0&-1\\
\end{array}%
\right)\right\},
$$
which implies that the~$G$-orbit of the~line $L_{1}$ is the~curve
that is given by $xyz=0$, which is impossible by
Lemma~\ref{lemma:v22-transitively}.
\end{proof}

By Lemma~\ref{lemma:Z2-Z2}, we may assume that $\Omega$ is
generated by $\psi(AB)$ and $\psi(C)$. Therefore
$$
\Omega=\left\{\left(%
\begin{array}{ccc}
1&0&0\\
0&1&0\\
0&0&1\\
\end{array}%
\right), \left(\begin{array}{ccc}
0&-1&0\\
-1&0&0\\
0&0&-1\\
\end{array}%
\right), \left(\begin{array}{ccc}
0&1&0\\
1&0&0\\
0&0&-1\\
\end{array}%
\right), \left(\begin{array}{ccc}
-1&0&0\\
0&-1&0\\
0&0&1\\
\end{array}%
\right)\right\},
$$
and the~equation of the~hexagon $\sum_{i=1}^{6}L_{i}$ is
$(x^2-z^2)(y^2-x^2)(z^2-y^2)=0$. Arguing as in the~proof of
Lemma~\ref{lemma:Z4}, we obtain a~contradiction.

The assertion of Theorem~\ref{theorem:v22-orbits} is proved.

\section{Proof of Theorem~\ref{theorem:auxiliary}}
\label{section:auxiliary}

Throughout this section we use the~assumptions and notation of
Section~\ref{section:space}. Suppose that
Theorem~\ref{theorem:auxiliary} is false. Let us derive a
contradiction.

\begin{lemma}
\label{lemma:v22-NF} There is a~$G$-invariant linear system
$\mathcal{M}$ without fixed components on~$\mathbb{P}^{3}$
such~that $\mathbb{NCS}(\mathbb{P}^{3}, \lambda
\mathcal{M})\ne\varnothing$ and $\mathbb{NCS}(\mathbb{P}^{3},
\lambda^{\prime}\tau(\mathcal{M}))\ne\varnothing$, where $\lambda$
and $\lambda^{\prime}$ are positive rational numbers such that
$\lambda\mathcal{M}\sim_{\mathbb{Q}}\lambda^{\prime}\tau(\mathcal{M})\sim_{\mathbb{Q}}\mathcal{O}_{\mathbb{P}^{3}}(4)$.
\end{lemma}

\begin{proof}
The~required assertion is well-known (see \cite{Co00},
\cite[Theorem~A.16]{Ch09}, \cite[Corollary~A.22]{Ch09}).
\end{proof}

The assertion of Lemma~\ref{lemma:v22-NF} is known as
the~Noether--Fano inequality.

\begin{lemma}
\label{lemma:space-genus-3-multiplicity} Either
$\mathrm{mult}_{C_{6}}(\mathcal{M})\leqslant 1/\lambda$ or
$\mathrm{mult}_{C_{6}}(\tau(\mathcal{M}))\leqslant
1/\lambda^{\prime}$.
\end{lemma}

\begin{proof}
This is easy (cf. the~proof of \cite[Lemma B.15]{Ch09}).
\end{proof}

Without loss of generality, we may assume that
$\mathrm{mult}_{C_{6}}(\mathcal{M})\leqslant 1/\lambda$. Then
$C_{6}\not\in\mathbb{NCS}(\mathbb{P}^{3}, \lambda \mathcal{M})$.
Note that the~set $\mathbb{NLCS}(\mathbb{P}^{3}, 2\lambda
\mathcal{M})$ contains every center in
$\mathbb{NCS}(\mathbb{P}^{3}, \lambda \mathcal{M})$ by
Lemma~\ref{lemma:mult-by-two}. However the set
$\mathbb{NLCS}(\mathbb{P}^{3}, 2\lambda \mathcal{M})$ may be
non-empty even if $\mathbb{NCS}(\mathbb{P}^{3}, \lambda
\mathcal{M})=\varnothing$ (see
Example~\ref{example:novye-centry}).

\begin{lemma}
\label{lemma:space-curves} Let $\Lambda$ be a~union of all curves
contained in $\mathbb{NLCS}(\mathbb{P}^{3}, 2\lambda
\mathcal{M})$. Then $\mathrm{deg}(\Lambda)\leqslant 15$.
\end{lemma}

\begin{proof}
Let $M_{1}$ and $M_{2}$ be general surfaces in $\mathcal{M}$, and
let $H$ be a~sufficiently general hyperplane in $\mathbb{P}^3$.
Then
$$
16\big\slash\lambda^{2}=H\cdot M_{1}\cdot M_{2}\geqslant\sum_{P\in \Lambda\cap H}\mathrm{mult}_{P}\Big(M_{1}\cdot M_{2}\Big)>\mathrm{deg}\big(\Lambda\big)\big\slash\lambda^{2}%
$$
by Theorem~\ref{theorem:Corti}, which implies that
$\mathrm{deg}(\Lambda)\leqslant 15$.
\end{proof}

\begin{lemma}
\label{lemma:space-8-points} The~set $\mathbb{NCS}(\mathbb{P}^{3},
\lambda \mathcal{M})$ does not contain any point in
$\Sigma_{8}\sqcup C_{6}$.
\end{lemma}

\begin{proof}
Let $M_{1}$ and $M_{2}$ be general surfaces in $\mathcal{M}$, and
let $Q$ be a~general surface in~$\mathcal{Q}$. Then
$$
32\big\slash\lambda^{2}=Q\cdot M_{1}\cdot M_{2}\geqslant\sum_{O_{i}\in\Sigma_{8}}\mathrm{mult}_{O_{i}}\Big(M_{1}\cdot M_{2}\Big)=8\mathrm{mult}_{O_{i}}\Big(M_{1}\cdot M_{2}\Big),%
$$
which implies that $\mathbb{NCS}(\mathbb{P}^{3}, \lambda
\mathcal{M})$ does not contain any point in $\Sigma_{8}$ by
Theorem~\ref{theorem:Iskovskikh}.

Suppose that  $\mathbb{NCS}(\mathbb{P}^{3}, \lambda \mathcal{M})$
contains a~point $P\in C_{6}$. Put \mbox{$M_{1}\cdot M_{2}=\nu
C_{6}+\Upsilon$}, where $\mu$ is a~non-negative integer, and
$\Upsilon$ is an~effective one-cycle such that
\mbox{$C_{6}\not\subset\mathrm{Supp}(\Upsilon)$}. Then
$\mathrm{mult}_{P}(\nu C_{14}+\Upsilon)\geqslant 4/\lambda^{2}$ by
Theorem~\ref{theorem:Iskovskikh}. Let $\Theta$ be the~$G$-orbit of
the~point $P$. Then $|\Theta|\geqslant 24$ by
Lemma~\ref{lemma:long-orbit}.

Let $S$ be a~general cubic surface in $\mathbb{P}^{3}$ such that
$C_{6}\subset S$. Then
$$
48\big\slash\lambda^{2}-18\nu=S\cdot\Upsilon\geqslant\sum_{O\in\Theta}\mathrm{mult}_{O}\big(\Upsilon\big)=\big|\Theta\big|\mathrm{mult}_{O}\big(\Upsilon\big)>24\Big(4\big\slash\lambda^{2}-\nu\Big),%
$$
because  $C_{6}$ is a~scheme-theoretic intersection of cubic
surfaces in $\mathbb{P}^{3}$. Thus $\nu>8/\lambda^{2}$. Let $H$ be
a~sufficiently general hyperplane in $\mathbb{P}^3$. Then
$$
16\big\slash\lambda^{2}=H\cdot M_{1}\cdot M_{2}=\nu H\cdot C_{6}+H\cdot\Upsilon\geqslant\nu H\cdot C_{6}=6\mu,%
$$
which implies that $\nu\leqslant 8/(3\lambda^{2})$, which is a~contradiction,
since $\nu>8/\lambda^{2}$.
\end{proof}

\begin{lemma}
\label{lemma:space-curves-through-8-points-degree-14-canonical}
The~set $\mathbb{NCS}(\mathbb{P}^{3}, \lambda \mathcal{M})$ does
not contain the~curve $C_{14}$.
\end{lemma}

\begin{proof}
Suppose that $C_{14}\in\mathbb{NCS}(\mathbb{P}^{3}, \lambda
\mathcal{M})$. Let us put $\mu=\mathrm{mult}_{C_{14}}(\mathcal{M})$ and
\mbox{$m=\mathrm{mult}_{O_{1}}(\mathcal{M})$}. Then $\mu>1/\lambda$
because $C_{14}\in\mathbb{NCS}(\mathbb{P}^{3}, \lambda
\mathcal{M})$.

Let us find an~upper bound for $\mu$. Let $M_{1}$ and $M_{2}$ be
general surfaces in $\mathcal{M}$, and let $H$ be a~sufficiently
general hyperplane in $\mathbb{P}^3$. Then
$$
16\big\slash\lambda^{2}=H\cdot M_{1}\cdot M_{2}\geqslant\sum_{P\in C_{14}\cap H}\mathrm{mult}_{C_{14}}\Big(M_{1}\cdot M_{2}\Big)\geqslant 14\mu^{2},%
$$
which implies that $\mu\leqslant \sqrt{8/7}/\lambda$. Thus, we
have $1/\lambda<\mu\leqslant \sqrt{8/7}/\lambda$.

Let us find a~lower and an upper bound for $m$.

Let $\bar{C}_{14}$ and $\bar{\mathcal{M}}$ be~the~proper
transforms of the~curve $C_{14}$ and the~linear
system~$\mathcal{M}$ on the~threefold $U$,~respectively, let $C$ be
a~general conic in $E_{1}\cong\mathbb{P}^{2}$ such that
$\bar{C}_{14}\cap E_{1}\subset C$. Then $|\bar{C}_{14}\cap
E_{1}|=3$ and the~points of the~set $\bar{C}_{14}\cap E_{1}$ are
non-coplanar (see the~proof of Lemma~\ref{lemma:F21-A4}), which
implies  that the~conic $C$ is not contained in the~base locus of
the~linear system $\bar{\mathcal{M}}$. Hence
$$
2m=C\cdot\bar{\mathcal{M}}\geqslant\sum_{P\in\bar{C}_{14}\cap
E_{1}}\mathrm{mult}_{P}\big(\bar{\mathcal{M}}\big)\geqslant\sum_{P\in\bar{C}_{14}\cap
E_{1}}\mathrm{mult}_{\bar{C}_{14}}\big(\bar{\mathcal{M}}\big)=3\mu,
$$
which implies that $m\geqslant 3\mu/2$. Thus, we see that
$m>3/(2\lambda)$, since $\mu>1/\lambda$.

Let us find an upper bound for $m$ (a~trivial upper bound $m\leqslant
2/\lambda$ follows by Lemma~\ref{lemma:space-8-points}).

Let $\bar{M}_{1}$ and $\bar{M}_{2}$ be the~proper transforms on
$U$ of the~surfaces $M_{1}$ and $M_{2}$, respectively, and let
$\bar{Q}$ be the~proper transform of a~general surface in
the~linear system $\mathcal{Q}$. Then
\begin{multline*}
32\big\slash\lambda^{2}-8m^{2}=\bar{Q}\cdot\bar{M}_{1}\cdot\bar{M}_{2}\geqslant\sum_{P\in\bar{C}_{14}\cap\bar{Q}}\mathrm{mult}_{\bar{C}_{14}}\Big(\bar{M}_{1}\cdot\bar{M}_{2}\Big)=\\
=4\mathrm{mult}_{\bar{C}_{14}}\Big(\bar{M}_{1}\cdot\bar{M}_{2}\Big)\geqslant 4\mu^{2}\geqslant 4\big\slash\lambda^{2},%
\end{multline*}
which implies that $m\leqslant\sqrt{7/2}/\lambda$. Thus, we have
$3/(2\lambda)<m\leqslant\sqrt{7/2}/\lambda$.

Let $\upsilon\colon W\to U$ be the~blow up of the~smooth curve
$\bar{C}_{14}$, let $F$ be the~$\upsilon$-exceptional divisor, let
$\hat{\mathcal{M}}$ and $\hat{\mathcal{D}}$ be the~proper
transforms of $\mathcal{M}$ and $\mathcal{D}$ (see
Lemma~\ref{lemma:space-C-14-properties-more}) on~$W$,
respectively. Then
\begin{equation}
\label{equation:qunitics-etc}
\left\{\aligned%
&\hat{\mathcal{M}}\sim_{\mathbb{Q}}\big(\pi\circ\upsilon\big)^{*}\Bigg(\frac{4}{\lambda}H\Bigg)-m\sum_{i=1}^{8}\upsilon^{*}\big(E_{i}\big)-\mu F,\\
&\hat{\mathcal{D}}\sim_{\mathbb{Q}}\big(\pi\circ\upsilon\big)^{*}\Big(5H\Big)-2\sum_{i=1}^{8}\upsilon^{*}\big(E_{i}\big)-F.\\
\endaligned
\right.
\end{equation}

Let $\hat{D}$ be a~general surface in $\hat{\mathcal{D}}$. Then $\hat{D}$ is
nef by Lemma~\ref{lemma:space-C-14-properties-more} and
$(\ref{equation:qunitics-etc})$.

Let $\hat{M}_{1}$ and $\hat{M}_{2}$ be the~proper transforms on
$W$ of the~surfaces $M_{1}$ and $M_{2}$, respectively.~Then
$$
80\big\slash\lambda^{2}-2\Big(17\mu^{2}+4m^{2}+56\mu\big\slash\lambda-24m\mu\Big)=\hat{D}\cdot\hat{M}_{1}\cdot \hat{M}_{2}\geqslant 0,%
$$
because $\hat{D}$ is nef and $F^{3}=-12$. Put
$\bar{\mu}=\mu\lambda$ and $\bar{m}=m\lambda$. One has
$$
17\bar{\mu}^{2}+4\bar{m}^{2}+56\bar{\mu}-24\bar{m}\bar{\mu}\geqslant 40,
\quad 1<\bar{\mu}\leqslant\sqrt{\frac{8}{7}},
\quad \frac{3}{2}<\bar{m}\leqslant \sqrt{\frac{7}{2}},%
$$
which easily leads to a~contradiction.
\end{proof}

\begin{lemma}
\label{lemma:space-curves-through-8-points-degree-14-canonical-points}
The~set $\mathbb{NCS}(\mathbb{P}^{3}, \lambda \mathcal{M})$ does
not contain any point in $C_{14}$.
\end{lemma}

\begin{proof}
Suppose that $\mathbb{NCS}(\mathbb{P}^{3}, \lambda \mathcal{M})$
contains a~point $P\in C_{14}$.
By~Lemma~\ref{lemma:space-8-points}, we have $P\not\in\Sigma_{8}$,
which implies that $P$ is a~smooth point of the~curve $C_{14}$, by
Lemma~\ref{lemma:space-C-14-properties}.

Choose $M_{1}$ and $M_{2}$ to be general surfaces in the~linear system
$\mathcal{M}$. Put \mbox{$M_{1}\cdot M_{2}=\mu C_{14}+\Upsilon+\Xi$},
where $\mu$ is a~non-negative integer, and $\Upsilon$ and $\Xi$
are~effective one-cycles such that
$C_{14}\not\subset\mathrm{Supp}(\Upsilon)$, every irreducible
component of the~curve $\mathrm{Supp}(\Upsilon)$ intersects the~
curve $C_{14}$ by a~non-empty set, and none of the~irreducible
components of the~curve $\mathrm{Supp}(\Xi)$ intersects $C_{14}$.
Then $\mathrm{mult}_{P}(\mu C_{14}+\Upsilon)\geqslant
4/\lambda^{2}$ by Theorem~\ref{theorem:Iskovskikh}.

Let $H$ be a~sufficiently general hyperplane in $\mathbb{P}^3$.
Then
$$
16\big\slash\lambda^{2}=H\cdot M_{1}\cdot M_{2}=\mu H\cdot C_{14}+H\cdot \Big(\Upsilon+\Xi\Big)\geqslant\mu H\cdot C_{14}=14\mu,%
$$
which implies that $\mu\leqslant 8/(7\lambda^{2})$.

Let $\Theta$ be the~$G$-orbit of the~point $P$. Then
$|\Theta|\geqslant 42$ by Lemma~\ref{lemma:long-orbit}.

Let $D$ be a~general quintic surface in the~linear system
$\mathcal{D}$ (see Lemma~\ref{lemma:space-C-14-properties-more}).
Then
\begin{multline*}
80\big\slash\lambda^{2}-70\mu\geqslant 80\big\slash\lambda^{2}-70\mu-D\cdot\Xi=D\cdot\Upsilon\geqslant\sum_{O\in\Theta}\mathrm{mult}_{O}\big(\Upsilon\big)=\\
=\big|\Theta\big|\mathrm{mult}_{O}\big(\Upsilon\big)>42\Big(4\big\slash\lambda^{2}-\mu\Big),%
\end{multline*}
which immediately leads to a~contradiction.
\end{proof}

Thus, the~set $\mathbb{NLCS}(\mathbb{P}^{3}, 2\lambda \mathcal{M})$ contains
a~center that is not contained in $C_{14}\sqcup C_{6}$.

\begin{lemma}
\label{lemma:space-curves-through-8-points-degree-14-othercurves}
Suppose that $\mathbb{NLCS}(\mathbb{P}^{3}, 2\lambda \mathcal{M})$
contains a~curve $C$ different from both~$C_{6}$ and~$C_{14}$. Then
$C_{14}\not\in\mathbb{NLCS}(\mathbb{P}^{3}, 2\lambda\mathcal{M})$,
and if $C_{6}\in\mathbb{NLCS}(\mathbb{P}^{3}, \lambda
\mathcal{M})$, then one has \mbox{$C_{6}\cap C=\varnothing$} and
$\mathrm{deg}(C)\leqslant 9$.
\end{lemma}

\begin{proof}
Let $Z$ be a~$G$-orbit of the~curve $C$. If
$C_{14}\in\mathbb{NLCS}(\mathbb{P}^{3}, 2\lambda\mathcal{M})$,
then
\mbox{$\mathrm{deg}(Z)+14=\mathrm{deg}(Z)+\mathrm{deg}(C_{14})\leqslant
15$} by Lemma~\ref{lemma:space-curves}, which is a~contradiction,
because there are no $G$-invariant lines in $\mathbb{P}^{3}$.

Suppose that $C_{6}\in\mathbb{NLCS}(\mathbb{P}^{3}, \lambda
\mathcal{M})$. Then
$\mathrm{deg}(C)\leqslant\mathrm{deg}(Z)\leqslant 9$ by
Lemma~\ref{lemma:space-curves}. Thus, to complete the~proof, we
must show that $C_{6}\cap Z=\varnothing$.

Suppose that $C_{6}\cap Z\ne\varnothing$. Then $|C_{6}\cap
Z|\geqslant 24$ by Lemma~\ref{lemma:long-orbit}. Let $\bar{Z}$ be
a~proper transform of the~curve $Z$ on the~ threefold $V$ (see
$(\ref{equation:space-involution})$), and let
$\bar{S}$~be~a~proper transform on $V$ of a~general cubic surface
in $\mathbb{P}^{3}$ that passes through the~curve~$C_{6}$. Then
$$
3\geqslant 3\mathrm{deg}\big(Z\big)-\big|C_{6}\cap Z\big|\geqslant\bar{S}\cdot\bar{Z}=\mathrm{deg}\Big(\beta\big(\bar{Z}\big)\Big),%
$$
which implies that $\bar{Z}$ is contracted by $\beta$ by
Lemma~\ref{lemma:jacobian-curve}, since $\beta(\bar{Z})$ is
$G$-invariant.

The~only curves contracted by the~morphism $\beta$ are proper
transforms of the~lines in $\mathbb{P}^{3}$ that are trisecants of
the~curve $C_{6}$. Then  $\beta(\bar{Z})\subset C_{6}$, the~subset
$\beta(\bar{Z})$ is $G$-invariant and $|\beta(\bar{Z})|\leqslant
9$, which is impossible by Lemma~\ref{lemma:long-orbit}.
\end{proof}

It follows from Lemmas~\ref{lemma:space-8-points},
 \ref{lemma:space-curves-through-8-points-degree-14-canonical} and
\ref{lemma:space-curves-through-8-points-degree-14-canonical-points} that
the~set $\mathbb{NLCS}(\mathbb{P}^{3}, 2\lambda \mathcal{M})$ contains
an~irreducible subvariety that is not contained in $C_{6}\cup C_{14}$. In fact,
we can say a~little bit more than this.

\begin{lemma}
\label{lemma:space-main} There are $\mu\in\mathbb{Q}$ and
$S\in\mathbb{LCS}(\mathbb{P}^{3}, \mu\mathcal{M})$ such that
\begin{itemize}
\item the~inequalities $0<\mu<2\lambda$ holds,%

\item the~log pair $(\mathbb{P}^{3}, \mu\mathcal{M})$ is log canonical along $S$,%

\item the~subvariety $S$ is a~minimal center in $\mathbb{LCS}(\mathbb{P}^{3}, \mu\mathcal{M})$,%

\item the~subvariety $S$ is not a~point of the~set $\Sigma_{8}$,%

\item the~subvariety $S$ is neither the~curve $C_{6}$ nor the~curve $C_{14}$,%

\item exactly one of the~following six cases is possible:

\begin{itemize}
\item[$(\mathbf{A})$] the~log pair $(\mathbb{P}^{3}, \mu\mathcal{M})$ has log canonical singularities,%
\item[$(\mathbf{B})$] $\mathrm{NCS}(\mathbb{P}^{3}, \mu\mathcal{M})=C_{14}$ and $S$ is a~point such that $S\not\in C_{14}$,%
\item[$(\mathbf{C})$] $\mathrm{NCS}(\mathbb{P}^{3}, \mu\mathcal{M})=\Sigma_{8}$ and $S\cap\Sigma_{8}=\varnothing$,%
\item[$(\mathbf{D})$] $\mathrm{NCS}(\mathbb{P}^{3}, \mu\mathcal{M})=C_{6}$ and $S$ is a~curve such that $S\cap C_{6}=\varnothing$ and
\mbox{$\mathrm{deg}(S)\leqslant 9$},%
\item[$(\mathbf{E})$] $\mathrm{NCS}(\mathbb{P}^{3}, \mu\mathcal{M})=C_{6}$ and $S$ is a~point such that $S\not\in\Sigma_{8}$,%
\item[$(\mathbf{F})$] $\mathrm{NCS}(\mathbb{P}^{3},
\mu\mathcal{M})=C_{6}\cup\Sigma_{8}$ and $S$ is a~curve such that one has
\mbox{$S\cap(C_{6}\cup \Sigma_{8})=\varnothing$}
and~$\mathrm{deg}(S)\leqslant 9$,%
\item[$(\mathbf{G})$] $\mathrm{NCS}(\mathbb{P}^{3}, \mu\mathcal{M})=C_{6}\cup\Sigma_{8}$ and $S$ is a~point such that $S\not\in C_{6}\cup\Sigma_{8}$.%
\end{itemize}
\end{itemize}
\end{lemma}

\begin{proof}
Let us show how to find $\mu$ and $S$ in several steps. Put
$$
\mu_{1}=\mathrm{sup}\left\{\epsilon\in\mathbb{Q}\ \left|%
\aligned
&\text{the~log pair}\ \Big(\mathbb{P}^{3}, \epsilon\mathcal{M}\Big)\ \text{is log canonical}\\
\endaligned\right.\right\},
$$
and let $S_{1}$ be a~minimal center in
$\mathbb{LCS}(\mathbb{P}^{3}, \mu_{1}\mathcal{M})$. Then
$\mu_{1}<2\lambda$.

If~$S_{1}$~is a~curve, then $S_{1}\cap\Sigma_{8}=\varnothing$ by
Lemmas~\ref{lemma:space-curves-through-8-points-degree-14} and
\ref{lemma:space-curves}, since $S_{1}$ is smooth by
Theorem~\ref{theorem:Kawamata}. If
$S_{1}\cap\Sigma_{8}=\varnothing$ and $S_{1}\ne C_{6}$, then we
have the~case $(\mathbf{A})$ by putting $\mu=\mu_{1}$ and
$S=S_{1}$. Thus, to complete the~proof, we may assume that either
$S_{1}\cap\Sigma_{8}\ne\varnothing$ or $S_{1}=C_{6}$.

Let us consider the~mutually excluding cases
$S_{1}\cap\Sigma_{8}\ne\varnothing$ and $S_{1}=C_{6}$ separately.

Suppose that $S_{1}=C_{6}$. Put
$$
\mu_{2}=\mathrm{sup}\left\{\epsilon\in\mathbb{Q}\ \left|%
\aligned
&\text{the~log pair}\ \Big(\mathbb{P}^{3}, \epsilon\mathcal{M}\Big)\
\text{is log canonical outside }\ C_{6}\\
\endaligned\right.\right\}<2\lambda,
$$
let $T_{2}$ be a~center in the~set $\mathbb{LCS}(\mathbb{P}^{3},
\mu_{2}\mathcal{M})$ such that $T_{2}\ne C_{6}$, and let $S_{2}$
be a~minimal center~in the~set $\mathbb{LCS}(\mathbb{P}^{3},
\mu_{2}\mathcal{M})$ such that $S_{2}\subseteq T_{2}$. If $T_{2}$
is a~curve, then $\mathrm{deg}(T_{2})\leqslant 9$ and
$T_{2}\cap\Sigma_{8}=T_{2}\cap C_{6}=\varnothing$ by
Lemmas~\ref{lemma:space-curves-through-8-points-degree-14},
\ref{lemma:space-curves} and
\ref{lemma:space-curves-through-8-points-degree-14-othercurves},
since $C_{14}\cap C_{6}=\varnothing$ by
Lemma~\ref{lemma:space-C-14-properties} and
$C_{6}\in\mathbb{NLCS}(\mathbb{P}^{3}, 2\lambda \mathcal{M})$.

If $T_{2}$ is a~curve and $S_{2}=T_{2}$, then we have the~case
$(\mathbf{D})$ by putting $\mu=\mu_{2}$ and $S=S_{2}$. If $T_{2}$
is a~curve and $S_{2}$ is a~point, then we have the~case
$(\mathbf{E})$ by putting $\mu=\mu_{2}$ and $S=S_{2}$. If
$T_{2}=S_{2}$ is a~point not in $\Sigma_{8}$, then we have
the~case $(\mathbf{E})$ by putting $\mu=\mu_{2}$ and $S=S_{2}$.
Thus, to finish the~case $S_{1}=C_{6}$, we may assume that
$S_{2}=T_{2}$ is a~point in $\Sigma_{8}$. Put
$$
\mu_{3}=\mathrm{sup}\left\{\epsilon\in\mathbb{Q}\ \left|%
\aligned
&\text{the~log pair}\ \Big(\mathbb{P}^{3}, \epsilon\mathcal{M}\Big)\
\text{is log canonical outside }\ C_{6}\cup\Sigma_{8}\\
\endaligned\right.\right\}<2\lambda,
$$
let $T_{3}$ be a~center in the~set $\mathbb{LCS}(\mathbb{P}^{3},
\mu_{3}\mathcal{M})$ such that $T_{3}\not\subset
C_{6}\cup\Sigma_{8}$, and let $S_{3}$ be a~minimal center~in
the~set $\mathbb{LCS}(\mathbb{P}^{3}, \mu_{3}\mathcal{M})$ such
that $S_{3}\subseteq T_{3}$. If $T_{3}$ is a~curve, then
$\mathrm{deg}(T_{3})\leqslant 9$ and
$T_{3}\cap\Sigma_{8}=T_{3}\cap C_{6}=\varnothing$ by
Lemmas~\ref{lemma:space-curves-through-8-points-degree-14},
\ref{lemma:space-curves} and
\ref{lemma:space-curves-through-8-points-degree-14-othercurves},
since $C_{14}\cap C_{6}=\varnothing$ by
Lemma~\ref{lemma:space-C-14-properties} and
$C_{6}\in\mathbb{NLCS}(\mathbb{P}^{3}, 2\lambda \mathcal{M})$.

If $T_{3}$ is a~curve and $S_{3}=T_{3}$, then we have the~case
$(\mathbf{F})$ by putting $\mu=\mu_{2}$ and $S=S_{2}$. If $T_{2}$
is a~curve and $S_{2}$ is a~point, then we have the~case
$(\mathbf{G})$ by putting $\mu=\mu_{2}$ and $S=S_{2}$. If
$T_{2}=S_{2}$ is a~point, then we have the~case $(\mathbf{G})$ by
putting $\mu=\mu_{2}$ and $S=S_{2}$. Therefore, the~case
$S_{1}=C_{6}$ is done, and we may assume that
$S_{1}\cap\Sigma_{8}\ne\varnothing$, which implies that
the~subvariety $S_{1}$ is a~point in $\Sigma_{8}$ by
Theorem~\ref{theorem:Kawamata} and
Lemmas~\ref{lemma:space-curves-through-8-points-degree-14}~and~\ref{lemma:space-curves}.
Put
$$
\mu_{2}^{\prime}=\mathrm{sup}\left\{\epsilon\in\mathbb{Q}\ \left|%
\aligned
&\text{the~log pair}\ \Big(\mathbb{P}^{3}, \epsilon\mathcal{M}\Big)\
\text{is log canonical outside }\ \Sigma_{8}\\
\endaligned\right.\right\}<2\lambda,
$$
let $T_{2}^{\prime}$ be a~center in the~set
$\mathbb{LCS}(\mathbb{P}^{3}, \mu_{2}^{\prime}\mathcal{M})$ such
that $T_{2}^{\prime}\not\subset\Sigma_{8}$, and let
$S_{2}^{\prime}$ be a~minimal center~in the~set
$\mathbb{LCS}(\mathbb{P}^{3},
\mu_{2}^{\prime}\mathcal{M})$~such~that~$S_{2}^{\prime}\subseteq
T_{2}^{\prime}$. Note that if $S_{2}^{\prime}$ is a~point in
$\Sigma_{8}$, then $T_{2}^{\prime}\ne S_{2}^{\prime}$ which
implies that $T_{2}^{\prime}$ must be the~curve $C_{14}$ by
Lemmas~\ref{lemma:space-curves-through-8-points-degree-14} and
\ref{lemma:space-curves}.

If $T_{2}^{\prime}$ is a~point, then
$S_{2}^{\prime}=T_{2}^{\prime}$, and we have the~case
$(\mathbf{C})$ by putting
$\mu=\mu_{2}^{\prime}$ and~\mbox{$S=S_{2}^{\prime}$}. So, to complete
the~proof, we may assume that $T_{2}^{\prime}$ is an irreducible
curve.

If $S_{2}^{\prime}=T_{2}^{\prime}$, then $S_{2}^{\prime}$ is
a~smooth curve such that $\mathrm{deg}(S_{2}^{\prime})\leqslant
15$ by Theorem~\ref{theorem:Kawamata} and
Lemma~\ref{lemma:space-curves}, which immediately implies that
$T_{2}^{\prime}\cap\Sigma_{8}=\varnothing$ by
Lemma~\ref{lemma:space-curves-through-8-points-degree-14}. If
$S_{2}^{\prime}=T_{2}^{\prime}\ne C_{6}$, then we have the~case
$(\mathbf{C})$ by putting
$\mu=\mu_{2}^{\prime}$~and~$S=S_{2}^{\prime}$. If $S_{2}^{\prime}$
is a~point not in $\Sigma_{8}$, then we have the~case
$(\mathbf{C})$ by putting $\mu=\mu_{2}^{\prime}$ and
$S=S_{2}^{\prime}$. Hence, to complete the~proof, we may assume
that either $S_{2}^{\prime}$ is a~point of the~set $\Sigma_{8}$
and $T_{2}^{\prime}=C_{14}$, or we have
$S_{2}^{\prime}=T_{2}^{\prime}=C_{6}$. Let us consider these cases
separately.

Suppose that $S_{2}^{\prime}$ is a~point of the~set $\Sigma_{8}$
and $T_{2}^{\prime}=C_{14}$. Put
$$
\mu_{3}^{\prime}=\mathrm{sup}\left\{\epsilon\in\mathbb{Q}\ \left|%
\aligned
&\text{the~log pair}\ \Big(\mathbb{P}^{3}, \epsilon\mathcal{M}\Big)\
\text{is log canonical outside }\ C_{14}\\
\endaligned\right.\right\}<2\lambda,
$$
let $T_{3}^{\prime}$ be a~center in the~set
$\mathbb{LCS}(\mathbb{P}^{3}, \mu_{3}^{\prime}\mathcal{M})$ such
that $T_{3}^{\prime}\not\subset C_{14}$. Then $T_{3}$ is a~point
by
Lemma~\ref{lemma:space-curves-through-8-points-degree-14-othercurves},
which implies that we have the~case $(\mathbf{B})$ by putting
$\mu=\mu_{3}^{\prime}$~and~$S=T_{3}^{\prime}$.

To complete the~proof, we may assume that
$S_{2}^{\prime}=T_{2}^{\prime}=C_{6}$. Put
$$
\mu_{3}^{\prime\prime}=\mathrm{sup}\left\{\epsilon\in\mathbb{Q}\ \left|%
\aligned
&\text{the~log pair}\ \Big(\mathbb{P}^{3}, \epsilon\mathcal{M}\Big)\
\text{is log canonical outside }\ C_{6}\cup\Sigma_{8}\\
\endaligned\right.\right\}<2\lambda,
$$
let $T_{3}^{\prime}$ be a~center in the~set
$\mathbb{LCS}(\mathbb{P}^{3}, \mu_{3}^{\prime\prime}\mathcal{M})$
such that $T_{3}^{\prime\prime}\not\subset C_{6}\cap\Sigma_{8}$,
and let $S_{3}^{\prime\prime}$ be a~minimal center~in the~set
$\mathbb{LCS}(\mathbb{P}^{3},
\mu_{3}^{\prime\prime}\mathcal{M})$~such~that~$S_{3}^{\prime\prime}\subseteq
T_{3}^{\prime\prime}$. Note that
\mbox{$C_{14}\not\in\mathbb{LCS}(\mathbb{P}^{3}, 2\lambda\mathcal{M})$}
by~Lemma~\ref{lemma:space-curves-through-8-points-degree-14-othercurves}.

If $T_{3}^{\prime\prime}$ is a~point, then
$S_{3}^{\prime\prime}=T_{3}^{\prime\prime}$, and we have the~case
$(\mathbf{G})$ by
putting~\mbox{$\mu=\mu_{3}^{\prime\prime}$} and~$S=S_{3}^{\prime\prime}$. Thus,
we may assume that $T_{3}^{\prime\prime}$ is a~curve. Then one has
\mbox{$\mathrm{deg}(T_{3}^{\prime\prime})\leqslant 9$} and
\mbox{$T_{3}^{\prime\prime}\cap
C_{6}=T_{3}^{\prime\prime}\cap\Sigma_{8}=\varnothing$} by
Lemmas~\ref{lemma:space-curves-through-8-points-degree-14},
\ref{lemma:space-curves} and
\ref{lemma:space-curves-through-8-points-degree-14-othercurves},
since~\mbox{$C_{6}\in\mathbb{LCS}(\mathbb{P}^{3},
2\lambda\mathcal{M})\not\ni C_{14}$}. Finally, we put
$\mu=\mu_{3}^{\prime\prime}$ and $S=S_{3}^{\prime\prime}$. Then we
have the~case $(\mathbf{G})$  if the~center~$S_{3}^{\prime\prime}$
is a~point, and we have the~case $(\mathbf{F})$  if the~center
$S_{3}^{\prime\prime}$ is a~curve, which completes the~proof.
\end{proof}

Let $\epsilon$ be any rational number such that
$\mu<\epsilon\mu<2\lambda$. Then it follows from
Lemma~\ref{lemma:Kawamata-Shokurov-trick} that there is
a~$G$-in\-va\-riant linear system $\mathcal{B}$ on
$\mathbb{P}^{3}$ such that $\mathcal{B}$ does not have
fixed~components, and there are positive  rational numbers
$\epsilon_{1}$ and $\epsilon_{2}$ such that $1\geqslant
\epsilon_{1}\gg\epsilon_{2}\geqslant 0$ and
$$
\mathbb{LCS}\Big(\mathbb{P}^{3}, \epsilon_{1}\mu\mathcal{M}+\epsilon_{2}\mathcal{B}\Big)=\Bigg(\bigsqcup_{g\in G}\Big\{g\big(S\big)\Big\}\Bigg)\bigsqcup\mathbb{NLCS}\Big(\mathbb{P}^{3}, \mu\mathcal{M}\Big),%
$$
the~log pair
$(\mathbb{P}^{3},\epsilon_{1}\mu\mathcal{M}+\epsilon_{2}\mathcal{B})$
is log canonical at every point of $g(Z)$ for every $g\in G$, and
$\epsilon_{1}
\mu\mathcal{M}+\epsilon_{2}\mathcal{B}\sim_{\mathbb{Q}} \epsilon
\mu\mathcal{M}$. Let $Z$ be the~$G$-orbit of the~subvariety $S$.
Then one of the~following cases is possible:
\begin{itemize}
\item[$(\mathbf{A})$] $\mathrm{LCS}(\mathbb{P}^{3}, \epsilon_{1} \mu\mathcal{M}+\epsilon_{2}\mathcal{B})=Z$,%
\item[$(\mathbf{B})$] $\mathrm{LCS}(\mathbb{P}^{3}, \epsilon_{1} \mu\mathcal{M}+\epsilon_{2}\mathcal{B})=Z\sqcup C_{14}$ and $Z$ is finite set,%
\item[$(\mathbf{C})$] $\mathrm{LCS}(\mathbb{P}^{3}, \epsilon_{1} \mu\mathcal{M}+\epsilon_{2}\mathcal{B})=Z\sqcup\Sigma_{8}$,%
\item[$(\mathbf{D})$] $\mathrm{LCS}(\mathbb{P}^{3}, \epsilon_{1} \mu\mathcal{M}+\epsilon_{2}\mathcal{B})=Z\sqcup C_{6}$ and $Z$ is a~curve such that  $\mathrm{deg}(Z)\leqslant 9$,%
\item[$(\mathbf{E})$]
$\mathrm{LCS}(\mathbb{P}^{3},
\epsilon_{1} \mu\mathcal{M}+\epsilon_{2}\mathcal{B})=Z\sqcup C_{6}$
and $Z$ is a~finite set such that $Z\neq\Sigma_{8}$,%
\item[$(\mathbf{F})$] $\mathrm{LCS}(\mathbb{P}^{3}, \epsilon_{1} \mu\mathcal{M}+\epsilon_{2}\mathcal{B})=Z\sqcup C_{6}\sqcup\Sigma_{8}$ and $Z$ is a~curve
such that $\mathrm{deg}(Z)\leqslant 9$,%
\item[$(\mathbf{G})$] $\mathrm{LCS}(\mathbb{P}^{3}, \epsilon_{1} \mu\mathcal{M}+\epsilon_{2}\mathcal{B})=Z\sqcup C_{6}\sqcup\Sigma_{8}$ and $Z$ is a~finite set.%
\end{itemize}

Put $D=\epsilon_{1} \mu\mathcal{M}+\epsilon_{2}\mathcal{B}$. Let
$\mathcal{L}$ be the~union of all connected components of the~log
canonical singularity subscheme $\mathcal{L}(\mathbb{P}^{3}, D)$
whose supports contains no components of the~subvariety~$Z$.~Then
$$
\mathrm{Supp}\big(\mathcal{L}\big)=\left\{\aligned%
&\varnothing\ \text{in the~case $(\mathbf{A})$},\\
&C_{14}\ \text{in the~case $(\mathbf{B})$},\\
&\Sigma_{8}\ \text{in the~case $(\mathbf{C})$},\\
&C_{6}\ \text{in the~cases $(\mathbf{D})$ and $(\mathbf{E})$},\\
&C_{6}\sqcup\Sigma_{8}\ \text{in the~cases $(\mathbf{F})$ and $(\mathbf{G})$},\\
\endaligned
\right.
$$

Let $\mathcal{I}(\mathbb{P}^{3},D)$ be the~multiplier ideal sheaf
of the~log pair $(\mathbb{P}^{3},D)$. Then
\begin{equation}
\label{equation:appendix-equality}
h^{0}\Big(\mathcal{O}_{\mathcal{L}}\otimes\mathcal{O}_{\mathbb{P}^{3}}\big(4\big)\Big)+h^{0}\Big(\mathcal{O}_{Z}\otimes\mathcal{O}_{\mathbb{P}^{3}}\big(4\big)\Big)=35-h^{0}\Big(\mathcal{O}_{\mathbb{P}^{3}}\big(4\big)\otimes\mathcal{I}\big(\mathbb{P}^{3},D\big)\Big)\leqslant 35%
\end{equation}
by Theorem~\ref{theorem:Shokurov-vanishing}.

\begin{corollary}
\label{corollary:space-ILC} Suppose that $Z$ is a~finite set. Then
$Z$ contains at most $35$ points, and $Z$ imposes independent
linear conditions on quartic surfaces in $\mathbb{P}^{3}$.
\end{corollary}

\begin{lemma}
\label{lemma:space-24-points} Suppose that $S$ is a~point. Then
$|Z|\ne 24$.
\end{lemma}

\begin{proof}
Suppose that $|Z|=24$. Then $Z=\Sigma_{24}$ by
Lemma~\ref{lemma:space-orbits-8-24}, and there is a hypersurface
$F$  of degree $4$ in $\mathbb{P}^{3}$ such~that $Z\setminus
S\subset F$ and $S\not\subset F$ by
Corollary~\ref{corollary:space-ILC}. Thus, there is a~unique point
$P\in C_{6}$ such that $P\ne S$ and $F\vert_{C_{6}}=P+Z\setminus
S\sim F_{4}\vert_{C_{6}}=Z$, which implies that $P\sim S$ on
$C_6$, which is impossible, since $C_{6}$ is a~smooth curve of
genus $3$.
\end{proof}

\begin{lemma}
\label{lemma:space-no-points} The~subvariety $S$ is a~curve.
\end{lemma}

\begin{proof}
Suppose that $S$ is a~point. By Lemmas~\ref{lemma:space-orbits-8-24} and
\ref{lemma:space-24-points} and Corollary~\ref{corollary:space-ILC}, we have
$|Z|=28$, because $Z\ne \Sigma_{8}$. Without loss of generality, we may assume
that $Z=\Sigma_{28}$.

Let $M_{1}$ and $M_{2}$ be general surfaces in $\mathcal{M}$. Put
$M_{1}\cdot M_{1}=\Xi+\Lambda$, where $\Xi$ and $\Lambda$ are
effective cycles such that
$\mathrm{Supp}(\Xi)\cap\mathrm{Supp}(\Lambda)$ consists of
finitely many points, and
$\mathrm{Supp}(\Xi)\supset\Sigma_{28}\not\subset\mathrm{Supp}(\Lambda)$.
If $S\in\mathbb{NCS}(\mathbb{P}^{3}, \lambda \mathcal{M})$, then
$\mathrm{mult}_{S}(\Xi)>4/\lambda^{2}$ by
Theorem~\ref{theorem:Iskovskikh}.

Let $R$ be a~general surface in $\mathcal{R}$ (see
Lemma~\ref{lemma:space-28-points}). Then $\mathrm{Supp}(\Xi)\cap
R$ consists of at most~finitely many points by
Lemma~\ref{lemma:space-28-points} and
Corollary~\ref{corollary:space-ILC}. Thus, we must have
$$
64\big\slash\lambda^{2}\geqslant 64\big\slash\lambda^{2}-R\cdot\Lambda=R\cdot \Xi\geqslant\sum_{P\in\Sigma_{28}}\mathrm{mult}_{P}\big(\Xi\big)=28\mathrm{mult}_{S}\big(\Xi\big),%
$$
which implies that $\mathrm{mult}_{S}(\Xi)\leqslant
16/(7\lambda^{2})$. So $g(S)\not\in\mathbb{NCS}(\mathbb{P}^{3},
\lambda \mathcal{M})$ for every $g\in G$.

Note that $\Sigma_{28}\not\subset C_{6}\sqcup C_{14}$ by
Lemma~\ref{lemma:long-orbit}, since $C_{6}$ is smooth and
$\mathrm{Sing}(C_{14})=\Sigma_8$.

It follows from Lemmas~\ref{lemma:space-8-points},
 \ref{lemma:space-curves-through-8-points-degree-14-canonical} and
\ref{lemma:space-curves-through-8-points-degree-14-canonical-points}
that the~set $\mathbb{NLCS}(\mathbb{P}^{3}, 2\lambda \mathcal{M})$
contains an~irreducible subvariety that is not contained in
$C_{6}\cup C_{14}\cup \Sigma_{28}$. Put
$$
\bar{\mu}=\mathrm{sup}\left\{\epsilon\in\mathbb{Q}\ \left|%
\aligned
&\text{the~log pair}\ \Big(\mathbb{P}^{3}, \epsilon\mathcal{M}\Big)\ \text{is log canonical outside}\ C_{6}\cup C_{14}\cup \Sigma_{28}\\
\endaligned\right.\right\}.
$$

The set $\mathbb{LCS}(\mathbb{P}^{3}, \bar{\mu} \mathcal{M})$
contains every point in $\Sigma_{28}$. If $\bar{\mu}>\mu$, then
$\mathbb{NLCS}(\mathbb{P}^{3}, \bar{\mu} \mathcal{M})$~also
contains every points in $\Sigma_{28}$. It follows from
Lemma~\ref{lemma:mult-by-two} that $\bar{\mu}<2\lambda$. Note that
$\bar{\mu}\geqslant\mu$.

Let $\Omega$ be a~center $\mathbb{LCS}(\mathbb{P}^{3},
\bar{\mu}\mathcal{M})$ such that $\Omega\not\subset C_{6}\cup
C_{14}\cup \Sigma_{28}$. Note that $\Omega$ does exist. Let us
choose a~center $\Gamma\in\mathbb{LCS}(\mathbb{P}^{3}, \bar{\mu}
\mathcal{M})$ in the~following way:
\begin{itemize}
\item if $\Omega$ is a~point, then we put $\Gamma=\Omega$,%

\item if $\Omega$ is a~curve that is a~minimal center in $\mathbb{LCS}(\mathbb{P}^{3}, \bar{\mu} \mathcal{M})$, then we put $\Gamma=\Omega$,%

\item if $\Omega$ is a~curve that is not a~minimal center in
$\mathbb{LCS}(\mathbb{P}^{3}, \bar{\mu} \mathcal{M})$,\\ then
let $\Gamma$ be a~point in $\Omega$ that is also a~center in $\mathbb{LCS}(\mathbb{P}^{3}, \bar{\mu} \mathcal{M})$.%
\end{itemize}

Let $\Delta$ be a~$G$-orbit of the~center $\Gamma$. Then
$\Delta\cap\Sigma_{8}=\Delta\cap \Sigma_{28}=\varnothing$ by
Lemmas~\ref{lemma:space-curves-through-8-points-degree-14},~\ref{lemma:space-curves-throught-28-points},~\ref{lemma:space-curves-through-8-points-degree-14-othercurves},~\ref{lemma:space-curves}.

Let $\bar{\epsilon}$ be a~rational number such that
$\bar{\mu}<\bar{\epsilon}\bar{\mu}<2\lambda$. Arguing as in the~proof of
Lemma~\ref{lemma:Kawamata-Shokurov-trick}, we~obtain a~$G$-in\-va\-riant linear
system $\mathcal{B}^{\prime}$ on $\mathbb{P}^{3}$ such that
$\mathcal{B}^{\prime}$ does not have fixed~components. Moreover, we can choose
positive
rational numbers~$\bar{\epsilon}_{1}$ and~$\bar{\epsilon}_{2}$ such that
$1\geqslant \bar{\epsilon}_{1}\gg\bar{\epsilon}_{2}\geqslant 0$ and
$$
\mathbb{LCS}\Big(\mathbb{P}^{3},
\bar{\epsilon}_{1}\bar{\mu}\mathcal{M}+\bar{\epsilon}_{2}\mathcal{B}^{\prime}\Big)=
\Bigg(\bigsqcup_{g\in G}\Big\{g\big(\Gamma\big)\Big\}\Bigg)\bigsqcup\Bigg(\bigsqcup_{P\in\Sigma_{28}}\Big\{P\Big\}\Bigg)\bigsqcup\mathbb{NLCS}\Big(\mathbb{P}^{3}, \bar{\mu}\mathcal{M}\Big)$$
if $\bar{\mu}=\mu$, or
$$
\mathbb{LCS}\Big(\mathbb{P}^{3},
\bar{\epsilon}_{1}\bar{\mu}\mathcal{M}+\bar{\epsilon}_{2}\mathcal{B}^{\prime}\Big)=
\Bigg(\bigsqcup_{g\in G}\Big\{g\big(\Gamma\big)\Big\}\Bigg)\bigsqcup\mathbb{NLCS}\Big(\mathbb{P}^{3}, \bar{\mu}\mathcal{M}\Big)$$
if $\bar{\mu}>\mu$.%

Put $\bar{D}=\bar{\epsilon}_{1}
\bar{\mu}\mathcal{M}+\bar{\epsilon}_{2}\mathcal{B}^{\prime}$. Then
$\Gamma$ is a~connected component of the~subscheme
$\mathcal{L}(\mathbb{P}^{3}, \bar{D})$. Let $\bar{\mathcal{L}}$ be
the~union of all connected components of the~subscheme
$\mathcal{L}(\mathbb{P}^{3}, \bar{D})$ whose supports contain no
components of the~subvariety $\Delta$. Then
\mbox{$h^{0}(\mathcal{O}_{\bar{\mathcal{L}}}
\otimes\mathcal{O}_{\mathbb{P}^{3}}(4))
\geqslant 28$}, since $\Sigma_{28}\subseteq\mathrm{Supp}(\bar{\mathcal{L}})$.

Let $\mathcal{I}(\mathbb{P}^{3},\bar{D})$ be the~multiplier ideal
sheaf of the~log pair $(\mathbb{P}^{3},\bar{D})$.~Then
$$
h^{0}\Big(\mathcal{O}_{\Delta}\otimes\mathcal{O}_{\mathbb{P}^{3}}\big(4\big)\Big)=35-h^{0}\Big(\mathcal{O}_{\mathbb{P}^{3}}\big(4\big)\otimes\mathcal{I}\big(\mathbb{P}^{3},\bar{D}\big)\Big)-h^{0}\Big(\mathcal{O}_{\bar{\mathcal{L}}}\otimes\mathcal{O}_{\mathbb{P}^{3}}\big(4\big)\Big)\leqslant 7%
$$
by Theorem~\ref{theorem:Shokurov-vanishing}, which implies that
$\Gamma$ is not a~point by Lemma~\ref{lemma:space-orbits-8-24}. We
see that $\Delta$ is a~curve. By Theorem~\ref{theorem:Kawamata},
the~curve $\Delta$ is a~smooth curve in $\mathbb{P}^{3}$ of degree
$d$ and genus $g\leqslant 2d$. Then
$$
\frac{\mathrm{deg}\big(\Delta\big)}{\mathrm{deg}\big(\Gamma\big)}\Big(2d+1\Big)\leqslant\frac{\mathrm{deg}\big(\Delta\big)}{\mathrm{deg}\big(\Gamma\big)}\Big(4d-g+1\Big)\leqslant 7,%
$$
which gives $d\leqslant 3$. Thus, we have $\Delta\ne\Gamma$, so that
$\mathrm{deg}(\Delta)\geqslant 7\mathrm{deg}(\Gamma)$ by
Corollary~\ref{corollary:PSL-permutation}. We have
$$
21\leqslant 7\Big(2d+1\Big)\leqslant \Big(2d+1\Big)\frac{\mathrm{deg}\big(\Delta\big)}{\mathrm{deg}\big(\Gamma\big)}\Big(2d+1\Big)\leqslant\frac{\mathrm{deg}\big(\Delta\big)}{\mathrm{deg}\big(\Gamma\big)}\Big(4d-g+1\Big)\leqslant 7,%
$$
which is a~contradiction.
\end{proof}

By Theorem~\ref{theorem:Kawamata}, the~curve $S$ is a~smooth curve
of degree $d$ and genus $g$ such that $g\leqslant 2d$, which
implies, in particular, that the~case $(\mathbf{B})$ is not
possible by
Lemma~\ref{lemma:space-curves-through-8-points-degree-14-othercurves}.

Let $G_{S}$ be the~stabilizer subgroup in $G$ of the~subvariety
$S$. Put
\mbox{$p=h^{0}(\mathcal{L}\otimes\mathcal{O}_{\mathbb{P}^{3}}(4))$}, put
$q=h^{0}(\mathcal{O}_{\mathbb{P}^{3}}(4)\otimes\mathcal{I}(\mathbb{P}^{3},D))$,
and let $r$ be the~number of irreducible components of the~curve
$Z$. Then
\begin{equation}
\label{equation:lemma:35-q-equality}
r\big(4d-g+1\big)=35-q-p
\end{equation}
by $(\ref{equation:appendix-equality})$, the~Riemann--Roch
theorem and Remark~\ref{remark:centers}.

\begin{lemma}
\label{lemma:A6-reducible-curves} The~equality $r=1$ holds.
\end{lemma}

\begin{proof}
Suppose that $r\ne 1$. Then $r\geqslant 7$ by
Corollary~\ref{corollary:PSL-permutation}, which implies that
$$
4d-g+1=\frac{35-q-p}{r}\leqslant 5
$$
by $(\ref{equation:lemma:35-q-equality})$. But $g\leqslant 2d$.
Then $4d-g+1\leqslant 5$, which implies that $g=p=q=0$, $d=1$ and
$r=7$.

The~induced action of the~group $G_{S}\cong\mathrm{S}_{4}$ on
the~line $S$ is faithful, which implies that $S\subset F_{4}$,
since $G_{S}$-invariant subsets in $S$ have at least $6$ points.
Then $Z\subset F_{4}$, which contradicts $q=0$.
\end{proof}

Therefore, we see that $Z=S$.

\begin{remark}
\label{remark:space-p+q} Let $I$ be the~trivial representation of
the~group $G$. By Lemma~\ref{lemma:some-SL-representations} we
have
$$
H^{0}\Big(\mathcal{O}_{\mathbb{P}^{3}}\big(4\big)\Big)\cong
I\oplus
W_{6}\oplus W_{6}\oplus W_{7}\oplus W_{7}\oplus W_{8},%
$$
where $W_{i}$ is an~irreducible representation of the~group
$G\cong\PSLF$ of dimension~$i$. Then
$p+q\not\in\{2,3,4,5,10,11,17\}$, because $p$ is divisible by $8$.
If $S\subset F_{4}$, then $p+q\not\in\{0,2,3,4,5,6,10,11,12,17\}$.
\end{remark}

Note that there is a~natural faithful action of the~group $G$ on
the~curve $S$.

\begin{lemma}
\label{lemma:space-g-19} We have $g\in\{3 ,8, 10, 15, 17, 22, 24,
29\}$.
\end{lemma}

\begin{proof}
This follows from Lemmas~\ref{lemma:sporadic-genera}
and~\ref{lemma:space-orbits-8-24}, since $g\leqslant 2d\leqslant
30$ by Lemma~\ref{lemma:space-curves}.
\end{proof}

\begin{lemma}
\label{lemma:space-curves-in-PHI-4} Suppose that $d\ne 6$ and
$d\ne 12$. Then $S\subset F_{4}$.
\end{lemma}

\begin{proof}
Suppose that $S\not\subset F_{4}$. Then $F_{4}\cap S$ is union of
some  $G$-orbits $\Lambda_{1},\ldots,\Lambda_{s}$. Thus
$\sum_{i=1}^{s}n_{i}|\Lambda_{i}|=4d$ for some positive integers
$n_{1},n_{2},\ldots,n_{s}$. Using
Lemma~\ref{lemma:sporadic-genera}, we obtain that $d=14$, $s=1$,
$n_{1}=1$ and $|\Lambda_{1}|=56$. Then $g=22+q+p\leqslant 2d=28$,
which implies that $g=22$ by Lemma~\ref{lemma:space-g-19} and
Remark~\ref{remark:space-p+q}. By
Lemma~\ref{lemma:sporadic-genera}, we have $|\Lambda_{1}|\ne 56$,
which is a~contradiction.
\end{proof}

Note that $d\in\{6,7,8,9\}$ in the~cases $(\mathbf{D})$,
$(\mathbf{E})$, $(\mathbf{F})$, $(\mathbf{G})$.

\begin{lemma}
\label{lemma:space-d-6} The equality $d=6$ is impossible.
\end{lemma}

\begin{proof}
Suppose that $d=6$. Then it follows from
Lemma~\ref{lemma:space-g-19} that
$$3\leqslant g=q+p-10\leqslant
2d=12,$$ which implies that $g\in\{3,8,10\}$. But $g\leqslant 4$ by
Theorem~\ref{theorem:Castelnuovo}. Hence $S=C_{6}$ by
Lemma~\ref{lemma:jacobian-curve}, which is a~contradiction,
because $S\ne C_{6}$ by the~choice of the~center
$S\in\mathbb{LCS}(\mathbb{P}^{3},\mu\mathcal{M})$.
\end{proof}

\begin{lemma}
\label{lemma:space-d-7} The equality $d=7$ is impossible.
\end{lemma}

\begin{proof}
Suppose that $d=7$. Then $g\leqslant 7$ by
Theorem~\ref{theorem:Castelnuovo} and $g=3$ by
Lemma~\ref{lemma:space-g-19}.

By Lemma~\ref{lemma:space-curves-in-PHI-4}, we see that $S\subset
F_{4}$. Arguing as in the~proof of Lemma~\ref{lemma:space-g-19},
we see that the~curve $S$ is contained in the~intersection
$F_{4}\cap F_{6}\cap F_{14}$, which is impossible by
Lemma~\ref{lemma:invariants-properties}.
\end{proof}

\begin{lemma}
\label{lemma:space-d-8} The equality $d=8$ is impossible.
\end{lemma}

\begin{proof}
Suppose that $d=8$. Then $g\leqslant 9$ by
Theorem~\ref{theorem:Castelnuovo} and $g\in\{3,8\}$ by
Lemma~\ref{lemma:space-g-19}. By
Lemma~\ref{lemma:space-curves-in-PHI-4}, we see that $S\subset
F_{4}$. Arguing as in the~proof of Lemma~\ref{lemma:space-g-19},
we see that the~curve $S$ is contained in $F_{4}\cap F_{6}\cap
F_{8}^{\prime}$ if $g=8$. Thus $g=3$ by
Lemma~\ref{lemma:invariants-properties}. Then
$\mathcal{O}_{\mathbb{P}^{3}}(1)\vert_{S}\sim 2K_{S}$ by
Theorem~\ref{theorem:Dolgachev}, which is impossible, since
the~vector space $H^{0}(\mathcal{O}_{S}(2K_{S}))$ is an
irreducible six-dimensional representation of the~group
$G\cong\PSLF$ by Lemma~\ref{corollary:characters}.
\end{proof}

\begin{lemma}
\label{lemma:space-d-9} The equality $d=9$ is impossible.
\end{lemma}

\begin{proof}
Suppose that $d=9$. Then $g\leqslant 12$ by
Theorem~\ref{theorem:Castelnuovo} and $g\in\{3,8,10\}$ by
Lemma~\ref{lemma:space-g-19}. By
Lemma~\ref{lemma:space-curves-in-PHI-4}, we have $S\subset F_{4}$.
Arguing as in the~proof of Lemma~\ref{lemma:space-g-19}, we see
that $S\subset F_{4}\cap F_{6}\cap F_{8}^{\prime}$ if $g=8$.
Similarly, we see that \mbox{$S\subset F_{4}\cap F_{6}\cap F_{14}$}
if
$g=3$. Thus $g=10$ by Lemma~\ref{lemma:invariants-properties}.
Then $S$ is a~complete intersection of two cubic surfaces in
$\mathbb{P}^{3}$ (see~\cite[Example~6.4.3]{Har77}), which is
impossible, because there are no $G$-invariant pencils of cubic
surfaces by Remark~\ref{remark:pencils}.
\end{proof}

Thus, the~cases $(\mathbf{D})$, $(\mathbf{E})$, $(\mathbf{F})$,
$(\mathbf{G})$ are not possible.

\begin{lemma}
\label{lemma:space-d-10} The equality $d=10$ is impossible.
\end{lemma}

\begin{proof}
Suppose that $d=10$. Then $p+q\not\in\{2,4,11\}$ by
Lemma~\ref{lemma:space-curves-in-PHI-4} and
Remark~\ref{remark:space-p+q}.~But $g=6+q+p\leqslant 2d=20$, which
implies that $g\in\{8,10,15,17\}$ by Lemma~\ref{lemma:space-g-19}.
Thus, we see that $g=15$ and $p+q=9$. By
Lemma~\ref{lemma:space-curves-in-PHI-4}, we see that $S\subset
F_{4}$. Arguing as in the~proof of Lemma~\ref{lemma:space-g-19},
we see that the~curve $S$ is contained in the~intersection
$F_{4}\cap F_{6}\cap F_{8}^{\prime}$, which is impossible by
Lemma~\ref{lemma:invariants-properties}.
\end{proof}

\begin{lemma}
\label{lemma:space-d-11} The equality $d=11$ is impossible.
\end{lemma}

\begin{proof}
Suppose that $d=11$. Then $p+q\not\in\{0,5,12\}$ by
Lemma~\ref{lemma:space-curves-in-PHI-4} and
Remark~\ref{remark:space-p+q}.~But $g=10+q+p\leqslant 2d=22$,
which implies that $g\in\{10,15,17,22\}$ by
Lemma~\ref{lemma:space-g-19}. Thus, we see that $g=17$ and
$p+q=7$. By Lemma~\ref{lemma:space-curves-in-PHI-4}, we see that
$S\subset F_{4}$. Arguing as in the~proof of
Lemma~\ref{lemma:space-g-19}, we see that the~curve $S$ is
contained in the~intersection $F_{4}\cap F_{6}\cap
F_{8}^{\prime}$, which is impossible by
Lemma~\ref{lemma:invariants-properties}.
\end{proof}

\begin{lemma}
\label{lemma:space-d-12} The equality $d=12$ is impossible.
\end{lemma}

\begin{proof}
Suppose that $d=12$. Then $p+q\not\in\{3,5,10\}$ by
Remark~\ref{remark:space-p+q}. But \mbox{$g=14+q+p\leqslant 2d=24$},
which implies that $g\in\{15,17,22,24\}$ by
Lemma~\ref{lemma:space-g-19}. Thus either $g=15$ or $g=22$.
Arguing as in the~proof of Lemma~\ref{lemma:space-g-19} and using
Lemma~\ref{lemma:sporadic-genera}, we see that $S\subset F_{4}\cap
F_{6}\cap F_{8}^{\prime}$, which is impossible by
Lemma~\ref{lemma:invariants-properties}.
\end{proof}

\begin{lemma}
\label{lemma:space-d-13} The equality $d=13$ is impossible.
\end{lemma}

\begin{proof}
Suppose that $d=13$. Then $p+q\not\in\{4,6\}$ by
Lemma~\ref{lemma:space-curves-in-PHI-4} and
Remark~\ref{remark:space-p+q}.~But $g=18+q+p\leqslant 2d=26$,
which implies that $g\in\{22,24\}$ by
Lemma~\ref{lemma:space-g-19}. Thus $p+q\in\{4,6\}$, which is
a~contradiction.
\end{proof}

\begin{lemma}
\label{lemma:space-d-14} The equality $d=14$ is impossible.
\end{lemma}

\begin{proof}
Suppose that $d=14$. Then $p+q\ne 0$ by
Lemma~\ref{lemma:space-curves-in-PHI-4} and
Remark~\ref{remark:space-p+q}.~But $g=22+q+p\leqslant 2d=28$,
which implies that $g=22$ by Lemma~\ref{lemma:space-g-19}. Thus
$p=q=0$, which is a~contradiction.
\end{proof}

By Lemmas~\ref{lemma:space-d-6}, \ref{lemma:space-d-7},
\ref{lemma:space-d-8}, \ref{lemma:space-d-9},
\ref{lemma:space-d-10}, \ref{lemma:space-d-11},
\ref{lemma:space-d-12}, \ref{lemma:space-d-13},
\ref{lemma:space-d-14} and Theorem~\ref{theorem:Castelnuovo},~one
has~$d\geqslant 15$, which implies that $d=15$ by
Lemma~\ref{lemma:space-curves}. Then $p+q\ne 3$ by
Lemma~\ref{lemma:space-curves-in-PHI-4} and
Remark~\ref{remark:space-p+q}.~But $g=26+q+p\leqslant 2d=30$,
which implies that $g=29$ by Lemma~\ref{lemma:space-g-19}. Thus,
we have $p+q=3$, which is a~contradiction.

The~assertion of Theorem~\ref{theorem:auxiliary} is proved.

\section{Proof of Theorem~\ref{theorem:main}}
\label{section:v22}

Throughout this section we use assumptions and notation of
Theorem~\ref{theorem:main}, and we identify the~threefold $X$ with
its anticanonical image in $\mathbb{P}^{13}$ (cf.
Theorem~\ref{theorem:v22-U-14}). Suppose that
Theorem~\ref{theorem:main} is false. Let us derive
a~contradiction.

\begin{lemma}
\label{lemma:NF} There is a~$G$-invariant linear system
$\mathcal{M}$ without fixed components on~$X$ such~that
$\mathbb{NCS}(X, \lambda \mathcal{M})\ne\varnothing$, where
$\lambda$ is a~positive rational number such that
$-K_{X}\sim_{\mathbb{Q}}\lambda\mathcal{M}$.
\end{lemma}

\begin{proof}
The~required assertion is well-known (see \cite{Co00},
\cite[Corollary~A.18]{Ch09}).
\end{proof}

\begin{lemma}
\label{lemma:v22-curves} Let $\Lambda$ be a~union of all curves in
$\mathbb{NLCS}(X, 2\lambda \mathcal{M})$. Then
$\mathrm{deg}(\Lambda)\leqslant 21$.
\end{lemma}

\begin{proof}
Let $M_{1}$ and $M_{2}$ be general surfaces in $\mathcal{M}$, and
let $H$ be a~general surface in the~linear system $|-K_{X}|$. By
Theorem~\ref{theorem:Corti}, we have
$$
22\big\slash\lambda^{2}=H\cdot M_{1}\cdot M_{2}\geqslant\sum_{P\in \Lambda\cap H}\mathrm{mult}_{P}\Big(M_{1}\cdot M_{2}\Big)>\Big(H\cdot\Lambda\Big)\big\slash\lambda^{2}=\mathrm{deg}\big(\Lambda\big)\big\slash\lambda^{2},%
$$
which implies that $\mathrm{deg}(\Lambda)\leqslant 21$.
\end{proof}

By Lemma~\ref{lemma:mult-by-two}, the~set $\mathbb{NLCS}(X,
2\lambda \mathcal{M})$ contains every center in $\mathbb{NCS}(X,
\lambda \mathcal{M})$. Put
$$
\hat{\mu}=\mathrm{sup}\left\{\epsilon\in\mathbb{Q}\ \left|%
\aligned
&\text{the~log pair}\ \Big(X, \epsilon\mathcal{M}\Big)\ \text{is log canonical}\\
\endaligned\right.\right\}<2\lambda,
$$
let $\Sigma_{8}$ be the~unique $G$-invariant subset of the~
threefold $X$ consisting of $8$ points (see
Lemma~\ref{lemma:v22-8-points}), and let $F$ be the~unique
$G$-invariant surface in the~linear system $|-K_{X}|$ (see
Corollary~\ref{corollary:v22-K3}).

\begin{lemma}
\label{lemma:v22-8-points} Suppose that
$\Sigma_{8}\subset\mathbb{LCS}(X, \hat{\mu}\mathcal{M})$. Then
$\Sigma_{8}\not\subset\mathbb{NCS}(X, \lambda \mathcal{M})$, and
the~set $\mathbb{NLCS}(X, 2\lambda\mathcal{M})$ does not contain
curves that pass through a~point in $\Sigma_{8}$.
\end{lemma}

\begin{proof}
Take $\hat{\epsilon}\in\mathbb{Q}$ such that
$\hat{\mu}<\hat{\epsilon}\hat{\mu}<2\lambda$. By
Lemma~\ref{lemma:Kawamata-Shokurov-trick}, there is
a~$G$-in\-va\-riant linear system without fixed components
$\hat{\mathcal{B}}$ on  $X$, and there are
$\hat{\epsilon}_{1}\in\mathbb{Q}\ni \hat{\epsilon}_{2}$ such that
$1\geqslant \hat{\epsilon}_{1}\gg\hat{\epsilon}_{2}\geqslant 0$
and $\mathbb{LCS}(X,
\hat{\epsilon}_{1}\hat{\mu}\mathcal{M}+\hat{\epsilon}_{2}\hat{\mathcal{B}})=\Sigma_{8}$,
the~log pair
$(X,\hat{\epsilon}_{1}\hat{\mu}\mathcal{M}+\hat{\epsilon}_{2}\mathcal{B})$
is log canonical, and
$\hat{\epsilon}_{1}\hat{\mu}\mathcal{M}+\hat{\epsilon}_{2}\mathcal{B}\sim_{\mathbb{Q}}
\hat{\epsilon} \hat{\mu}\mathcal{M}$. Put
$\hat{D}=\hat{\epsilon}_{1}
\hat{\mu}\mathcal{M}+\hat{\epsilon}_{2}\mathcal{B}$.

Let $\mathcal{I}(X,\hat{D})$ is the~multiplier ideal sheaf of
the~log pair $(X,\hat{D})$, and let $H$ be a~general surface in
the~linear system $|-K_{X}|$. The~sequence of groups
$$
0\to H^{0}\Big(\mathcal{O}_{X}\big(H\big)\otimes\mathcal{I}\big(X,\hat{D}\big)\Big)\to H^{0}\Big(\mathcal{O}_{X}\big(H\big)\Big)\to H^{0}\big(\mathcal{O}_{\Sigma_{8}}\big)\to 0%
$$
is exact by Theorem~\ref{theorem:Shokurov-vanishing}. Then
$\Sigma_{8}$ imposes independent linear conditions on surfaces in
$|-K_{X}|$.

Let $\mathcal{D}$ be a~linear subsystem of the~linear system
$|-K_{X}|$  that consists of all surfaces passing through
$\Sigma_{8}$. Then it follows from Theorem~\ref{theorem:v22-U-14}
that $\mathcal{D}$ is the~unique five-dimensional $G$-invariant
linear subsystem of the~linear system $|-K_{X}|$ and
$F\not\in\mathcal{D}$. It is clear that the~base locus of
the~linear system $\mathcal{D}$ does not contain surfaces.

Suppose that the~base locus of the~linear system $\mathcal{D}$ does not contain
any curve that passes through a~point in $\Sigma_{8}$. Let us show that this
assumption implies everything we have to prove.

Let $M_{1}$ and $M_{2}$ be general surfaces in the~linear system
$\mathcal{M}$. Put $M_{1}\cdot M_{1}=\Xi+\Lambda$, where $\Xi$ and
$\Lambda$ are effective one-cycles such that
$\Sigma_{8}\subset\mathrm{Supp}(\Xi)$,
$\Sigma_{8}\not\subset\mathrm{Supp}(\Lambda)$ and
$\mathrm{Supp}(\Xi)\cap\mathrm{Supp}(\Lambda)$ consists of at most
finitely many points. Let $D$ be a~general surface in the~linear
system $\mathcal{D}$. Since $|\mathrm{Supp}(\Xi)\cap D|<+\infty$,
one has
$$
22\big\slash\lambda^{2}\geqslant 22\big\slash\lambda^{2}-D\cdot\Lambda=D\cdot \Xi\geqslant\sum_{P\in\Sigma_{8}}\mathrm{mult}_{P}\big(\Xi\big)=8\mathrm{mult}_{P}\big(\Xi\big),%
$$
for every point $P\in\Sigma_{8}$. Then $\mathbb{NCS}(X, \lambda
\mathcal{M})$ contains no points in $\Sigma_{8}$ by
Theorem~\ref{theorem:Iskovskikh}.

Let $Z$ be a~$G$-orbit of an irreducible curve in $X$ such that
$\Sigma_{8}\subset Z$. Then
$$
\mathrm{deg}\big(Z\big)=D\cdot Z\geqslant\sum_{P\in\Sigma_{8}}\mathrm{mult}_{P}\big(D\big)\mathrm{mult}_{P}\big(Z\big)\geqslant\sum_{P\in\Sigma_{8}}\mathrm{mult}_{P}\big(Z\big)\geqslant 24%
$$
by Lemma~\ref{lemma:F21-A4}. Hence $\mathbb{NLCS}(X,
2\lambda\mathcal{M})$ contains no components of the~curve~$Z$
by~Lemma~\ref{lemma:v22-curves}, which implies that
$\mathbb{NLCS}(X, 2\lambda\mathcal{M})$ does not contain curves
that pass through a~point in $\Sigma_{8}$.

To complete the~proof, we must show that the~base locus of
the~linear system $\mathcal{D}$ contains no curves that contain
a~point in $\Sigma_{8}$. Suppose that this is not true. Let us
derive a~contradiction.

The~base locus of the~linear system $\mathcal{D}$ contains a~curve
$C$ that is a~$G$-orbit of an~irreducible curve such that
$\Sigma_{8}\subset C$. Then $\mathrm{mult}_{P}(C)\geqslant 3$ for
every $P\in\Sigma_{8}$~by~Lemma~\ref{lemma:F21-A4}.

Let $Q_{1}$, $Q_{2}$, $Q_{3}$ and $Q_{4}$ be general points in
$X$, and let $\mathcal{H}$ be a~linear subsystem in
$|-K_{X}|$~that consists of all surfaces that contain the~set
$\{Q_{1},Q_{2},Q_{3},Q_{4}\}$. Then $\mathcal{H}\cap\mathcal{D}$
is a pencil, and the~base locus of the~linear system $\mathcal{H}$
contains no curves.

Let $H$ be a~general surface in $\mathcal{H}$, and choose
$D_{1}$ and $D_{2}$ to be
general surfaces in the~pencil~\mbox{$\mathcal{H}\cap\mathcal{D}$}. Then
$$
22=H\cdot D_{2}\cdot D_{2}\geqslant
\sum_{i=1}^{4}\mathrm{mult}_{Q_{i}}\big(H\big)\mathrm{mult}_{Q_{i}}\big(D_{1}\big)\mathrm{mult}_{Q_{i}}\big(D_{2}\big)+\big|H\cap C\big|\geqslant 4+\mathrm{deg}\big(C\big),%
$$
since $\mathrm{deg}(C)=|H\cap C|$. Thus, we see that
$\mathrm{deg}(C)\leqslant 18$.

Note that $C\not\subset F$, since $\Sigma_{8}\not\subset F$.
Therefore, we have
$$
\sum_{P\in F\cap C}\mathrm{mult}_{P}\big(C\big)\leqslant\big|F\cap C\big|\leqslant F\cdot C=\mathrm{deg}\big(C\big)\leqslant 18,%
$$
because $C\not\subset F$. Applying
Theorem~\ref{theorem:v22-orbits}, we see that
$\mathrm{deg}(C)=F\cdot C=|F\cap C|=14$ and $C$ is smooth at every
point of the~set $F\cap C$. Then $C$ is reducible by
Lemma~\ref{lemma:long-orbit}.

Put $C=\sum_{i=1}^{r}C_{i}$, where $C_{i}$ is an irreducible curve
and $r\in\N$. Then one has \mbox{$r=14/\mathrm{deg}(C_{1})$}. By
Corollary~\ref{corollary:PSL-permutation}, either
$\mathrm{deg}(C_{1})=1$ or $\mathrm{deg}(C_{1})=2$. On the~other
hand, we know that
$\mathrm{deg}(C_{1})+1\geqslant|C_{1}\cap\Sigma_{8}|$, since
$\Sigma_{8}$ imposes independent linear conditions on surfaces in
$|-K_{X}|$. Hence, we must have
$$
\frac{14}{\mathrm{deg}\big(C_{1}\big)}=r\geqslant \frac{8\cdot 3}{\big|C_{1}\cap\Sigma_{8}\big|},%
$$
because $\mathrm{mult}_{P}(C)\geqslant 3$ for every point
$P\in\Sigma_{8}$. Therefore
$|C_{1}\cap\Sigma_{8}|=1+\mathrm{deg}(C_{1})$.

Note that $\mathrm{deg}(C_{1})\ne 1$. Indeed, if $\mathrm{deg}(C_{1})=1$, then
$|C_{1}\cap\Sigma_{8}|=2$ and $r=28$, because the~group $G$ acts doubly
transitive on the~set $\Sigma_{8}$ (see~\cite[p.~173]{Edge47}), which is
a~contradiction.

Thus, we see that $C_{1}$ is an irreducible conic and $r=7$, so that
$|C_{i}\cap\Sigma_{8}|=3$.

Let $\Pi_{1}$ be a~plane in $\mathbb{P}^{13}$ such that
$C_{1}\subset\Pi_{1}$. Then $\Pi_{1}$ contains $3$ lines
$L_{1}^{1}$, $L_{1}^{2}$~and~$L_{1}^{3}$~such~that
$|L_{1}^{1}\cap\Sigma_{8}|=|L_{1}^{2}\cap\Sigma_{8}|=|L_{1}^{3}\cap\Sigma_{8}|=2$,
and the~$G$-orbit of the~line $L_{1}^{1}$ consists of $28$
different lines, since $G$ acts doubly transitive~on~$\Sigma_{8}$.

Now one can easily see that the~$G$-orbit of the~plane $\Pi_{1}$
consists of at least $28/3>9$ planes, which is impossible since
the~$G$-orbit of $\Pi_{1}$ consists of at most $r=7$ planes.
\end{proof}

By Lemma~\ref{lemma:v22-8-points}, the~set $\mathbb{NLCS}(X, 2\lambda
\mathcal{M})$ contains a~center that is not contained in~$\Sigma_{8}$. Put
$$
\mu=\mathrm{sup}\left\{\epsilon\in\mathbb{Q}\ \left|%
\aligned
&\text{the~log pair}\ \Big(X, \epsilon\mathcal{M}\Big)\ \text{is log canonical outside of the~subset}\ \Sigma_{8}\\
\endaligned\right.\right\}
$$
in the~case when $\Sigma_{8}\subset\mathbb{LCS}(X,
\hat{\mu}\mathcal{M})$. If $\Sigma_{8}\not\subset\mathbb{LCS}(X,
\hat{\mu}\mathcal{M})$, then put $\mu=\hat{\mu}$.

\begin{lemma}
\label{lemma:v22-curves-through-8-points} There exist a~minimal
center in $\mathbb{LCS}(X, \mu\mathcal{M})$ that is not a~point of
the~set $\Sigma_{8}$.
\end{lemma}

\begin{proof}
This immediately follows from Lemma~\ref{lemma:v22-8-points}.
\end{proof}

Let $S$ be a~minimal center in $\mathbb{LCS}(X, \mu\mathcal{M})$ such that
$S\not\subset \Sigma_{8}$, and let $Z$ be its $G$-orbit.

\begin{remark}
\label{remark:v22-8-points-curve} If $Z=S$ and $S$ is a~curve, then
$Z\cap\Sigma_{8}=\varnothing$ by Theorem~\ref{theorem:Kawamata} and
Lemma~\ref{lemma:long-orbit}.
\end{remark}

Let $\epsilon$ be any rational number such that
$\mu<\epsilon\mu<2\lambda$. Then it follows from
Lemma~\ref{lemma:Kawamata-Shokurov-trick} that there exists
a~$G$-in\-va\-riant linear system $\mathcal{B}$ on $X$ such that
$\mathcal{B}$ does not have fixed~components, and there exist
positive  rational numbers $\epsilon_{1}$ and $\epsilon_{2}$ such
that $1\geqslant \epsilon_{1}\gg\epsilon_{2}\geqslant 0$ and
$$
\mathbb{LCS}\Big(X, \epsilon_{1}\mu\mathcal{M}+\epsilon_{2}\mathcal{B}\Big)=\Bigg(\bigsqcup_{g\in G}\Big\{g\big(S\big)\Big\}\Bigg)\bigsqcup\mathbb{NLCS}\Big(X, \mu\mathcal{M}\Big),%
$$
the~log pair
$(X,\epsilon_{1}\mu\mathcal{M}+\epsilon_{2}\mathcal{B})$ is log
canonical at every point of the~subvariety $Z$, and $\epsilon_{1}
\mu\mathcal{M}+\epsilon_{2}\mathcal{B}\sim_{\mathbb{Q}} \epsilon
\mu\mathcal{M}$. Put $D=\epsilon_{1}
\mu\mathcal{M}+\epsilon_{2}\mathcal{B}$. By
Lemma~\ref{lemma:v22-curves-through-8-points}, we may have
the~following cases:
\begin{itemize}
\item[$(\mathbf{A})$] $\mathrm{LCS}(X, D)=Z$ and the~log pair $(X,D)$ is log canonical,%
\item[$(\mathbf{B})$] $\mathrm{LCS}(X, D)=Z\sqcup\Sigma_{8}$ and $Z$ is a~finite set,%
\item[$(\mathbf{C})$] $\mathrm{LCS}(X, D)=Z\sqcup\Sigma_{8}$ and $Z$ is a~curve.%
\end{itemize}

Let $\mathcal{L}$ be the~union of all connected components of
the~subscheme $\mathcal{L}(X, D)$ whose supports do not contains
any component of the~subvariety $Z$. Then
$$
\mathrm{Supp}\big(\mathcal{L}\big)=\left\{\aligned%
&\varnothing\ \text{in the~case $(\mathbf{A})$},\\
&\Sigma_{8}\ \text{in the~cases $(\mathbf{B})$ and $(\mathbf{C})$}.\\
\endaligned
\right.
$$

Let $\mathcal{I}(X,D)$ be the~multiplier ideal sheaf of the~log
pair $(X,D)$, and let $H$ be a~general surface in the~linear
system $|-K_{X}|$.~Then
\begin{equation}
\label{equation:v22-equality-main}
h^{0}\Big(\mathcal{O}_{Z}\otimes\mathcal{O}_{X}\big(H\big)\Big)=14-h^{0}\Big(\mathcal{O}_{X}\big(H\big)\otimes\mathcal{I}\big(X,D\big)\Big)-h^{0}\big(\mathcal{O}_{\mathcal{L}}\big)\leqslant 14%
\end{equation}
by Theorem~\ref{theorem:Shokurov-vanishing}, since $\mathcal{L}$
is at most a~zero-dimensional subscheme. We have
\mbox{$h^{0}(\mathcal{O}_{\mathcal{L}})\in\{0,8\}$}.

\begin{corollary}
\label{corollary:v22-ILC} If $Z$ is a~finite set, then
$|Z|\leqslant 14$ and the~points of $Z$ impose independent linear
conditions on surfaces in $|-K_{X}|$.
\end{corollary}

\begin{lemma}
\label{lemma:v22-no-points} The~subvariety $S$ is a~curve.
\end{lemma}

\begin{proof}
Suppose that $S$ is not a~curve. Then $|Z|=14$ by
Theorem~\ref{theorem:v22-orbits} and
Corollary~\ref{corollary:v22-ILC}, because we know that
$Z\ne\Sigma_{8}$. Note that $Z$ is not contained in any surface in
$|-K_{X}|$ by Corollary~\ref{corollary:v22-ILC}. Let $\mathcal{R}$
be a linear subsystem of the~linear system $|-2K_{X}|$ that
consists of all surfaces in $|-2K_{X}|$ that pass through the~set
$Z$. Then its~base locus consists of the~set $Z$ by
\cite[Theorem~2]{EiJ87}.

Let $M_{1}$ and $M_{2}$ be general surfaces in the~linear system
$\mathcal{M}$. Put $M_{1}\cdot M_{1}=\Xi+\Lambda$, where $\Xi$ and
$\Lambda$ are effective one-cycles such that
$Z\subset\mathrm{Supp}(\Xi)$ and
$Z\not\subset\mathrm{Supp}(\Lambda)$ and
$\mathrm{Supp}(\Xi)\cap\mathrm{Supp}(\Lambda)$ consists of at most
finitely many points. If $S\in\mathbb{NCS}(X, \lambda
\mathcal{M})$, then $\mathrm{mult}_{S}(\Xi)>4/\lambda^{2}$ by
Theorem~\ref{theorem:Iskovskikh}. Let $R$ be a~general surface in
the~ linear system $\mathcal{R}$. Then
$$
44\big\slash\lambda^{2}\geqslant 44\big\slash\lambda^{2}-R\cdot\Lambda=R\cdot \Xi\geqslant\sum_{P\in Z}\mathrm{mult}_{P}\big(\Xi\big)=14\mathrm{mult}_{S}\big(\Xi\big),%
$$
which implies that $g(S)\not\in\mathbb{NCS}(X, \lambda
\mathcal{M})$ for every $g\in G$.

By Lemma~\ref{lemma:v22-8-points}, the~set $\mathbb{NLCS}(X,
2\lambda \mathcal{M})$ contains a~center not contained in
$\Sigma_{8}\cup Z$. Put
$$
\bar{\mu}=\mathrm{sup}\left\{\epsilon\in\mathbb{Q}\ \left|%
\aligned
&\text{the~log pair}\ \Big(X, \epsilon\mathcal{M}\Big)\ \text{is log canonical outside of the~set}\ \Sigma_{8}\cup Z\\
\endaligned\right.\right\}.
$$

Note that $\bar{\mu}\geqslant \mu$ and $S\in\mathbb{LCS}(X,
\bar{\mu} \mathcal{M})$. Moreover, if $\bar{\mu}>\mu$, then
$S\in\mathbb{NLCS}(X, \bar{\mu} \mathcal{M})$, because
$S\in\mathbb{LCS}(X, \mu\mathcal{M})$. It follows from
Lemma~\ref{lemma:mult-by-two} that $\bar{\mu}<2\lambda$.

Let $\Gamma$ be a~center in $\mathbb{LCS}(X,
\bar{\mu}\mathcal{M})$ such that $\Gamma\not\subset
\Sigma_{8}\sqcup Z$, and let $\Delta$ be its $G$-orbit. Then
$\Delta\cap \Sigma_{8}=\varnothing$ by
Lemma~\ref{lemma:v22-8-points}. Note that $\Delta\cap
Z=\varnothing$ if $\Gamma$ is a~point and
$\Gamma\in\mathbb{NLCS}(X, 2\lambda\mathcal{M})$.

Suppose that the~set $\mathbb{NLCS}(X, 2\lambda\mathcal{M})$ does not contain
curves in $X$ that have a~non-empty intersection with the~set $Z$. Let us use
this assumption to derive a~contradiction.

Let $\bar{\epsilon}$ be a~rational number such that
$\bar{\mu}<\bar{\epsilon}\bar{\mu}<2\lambda$. Arguing as in the~proof of
Lemma~\ref{lemma:Kawamata-Shokurov-trick}, we~obtain a~$G$-in\-va\-riant linear
system $\mathcal{B}^{\prime}$ on $X$ such that $\mathcal{B}^{\prime}$ does not
have fixed~components.
Moreover, we can choose
positive
rational numbers~$\bar{\epsilon}_{1}$ and~$\bar{\epsilon}_{2}$ such that
$1\geqslant \bar{\epsilon}_{1}\gg\bar{\epsilon}_{2}\geqslant 0$ and
$$
\mathbb{LCS}\Big(X,
\bar{\epsilon}_{1}\bar{\mu}\mathcal{M}+\bar{\epsilon}_{2}\mathcal{B}^{\prime}\Big)=
\Bigg(\bigsqcup_{g\in G}\Big\{g\big(\Gamma\big)\Big\}\Bigg)
\bigsqcup \Bigg(\bigsqcup_{g\in G}\Big\{g\big(S\big)\Big\}\Bigg)
\bigsqcup\mathbb{NLCS}\Big(X, \bar{\mu}\mathcal{M}\Big)
$$
if $\bar{\mu}=\mu$, or
$$
\mathbb{LCS}\Big(X,
\bar{\epsilon}_{1}\bar{\mu}\mathcal{M}+\bar{\epsilon}_{2}\mathcal{B}^{\prime}\Big)=
\Bigg(\bigsqcup_{g\in G}\Big\{g\big(\Gamma\big)\Big\}\Bigg)\bigsqcup\mathbb{NLCS}\Big(X, \bar{\mu}\mathcal{M}\Big)$$
if $\bar{\mu}>\mu$.

Put $\bar{D}=\bar{\epsilon}_{1}
\bar{\mu}\mathcal{M}+\bar{\epsilon}_{2}\mathcal{B}^{\prime}$. Then
$\Gamma$ is a~connected component of the~subscheme~\mbox{$\mathcal{L}(X,
\bar{D})$}. Let $\bar{\mathcal{L}}$ be the~union of all connected
components of the~subscheme $\mathcal{L}(X, \bar{D})$ whose
supports do not contains any component of the~subvariety $\Delta$.
Then \mbox{$Z\subseteq\mathrm{Supp}(\bar{\mathcal{L}})$}, which implies
that
$h^{0}(\mathcal{O}_{\bar{\mathcal{L}}}\otimes\mathcal{O}_{X}(H))\geqslant
14$, because $\mathbb{NLCS}(X, 2\lambda\mathcal{M})$ does not
contain curves that have a~non-empty intersection with $Z$.

Let $\mathcal{I}(X,\bar{D})$ be the~multiplier ideal sheaf of
the~log pair $(X,\bar{D})$. Then
$$
0\leqslant h^{0}\Big(\mathcal{O}_{\Delta}\otimes\mathcal{O}_{X}\big(H\big)\Big)=14-h^{0}\Big(\mathcal{O}_{X}\big(H\big)\otimes\mathcal{I}\big(X,\bar{D}\big)\Big)-h^{0}\Big(\mathcal{O}_{\bar{\mathcal{L}}}\otimes\mathcal{O}_{X}\big(H\big)\Big)\leqslant 0,%
$$
by Theorem~\ref{theorem:Shokurov-vanishing}. Thus, we have
$h^{0}(\mathcal{O}_{\Delta}\otimes\mathcal{O}_{X}(H))=0$, which
implies that $\Gamma$ is not a~point. It follows from
Theorem~\ref{theorem:Kawamata} that $\Gamma$ is a~smooth curve of
genus $g$ such that $H\cdot\Gamma\geqslant 2g-1$. By
Remark~\ref{remark:centers}, the~curve $\Delta$ is a~disjoint
union of smooth curves isomorphic to $\Gamma$. Then
$$
0=h^{0}\Big(\mathcal{O}_{\Delta}\otimes\mathcal{O}_{X}\big(H\big)\Big)\geqslant H\cdot\Gamma-g+1>0,%
$$
which is a~contradiction.

Thus, there is a~curve $C_{1}\in\mathbb{NLCS}(X,
2\lambda\mathcal{M})$ such that $C_{1}\cap Z\ne\varnothing$. Let
$C$ be the \mbox{$G$-orbit}  of the~curve $C_{1}$. Then $\mathbb{NLCS}(X,
2\lambda\mathcal{M})$ contains every irreducible~component of
the~curve $C$, which implies that $\mathrm{deg}(C)\leqslant 21$.
Hence, we have $Z\subset C$. But $\mathrm{mult}_{P}(C)\geqslant 3$
for every $P\in\Sigma_{8}$  by Lemma~\ref{lemma:F21-A4}. Thus, we
have
$$
42\geqslant 2\mathrm{deg}\big(C\big)=R\cdot C\geqslant\sum_{P\in Z}\mathrm{mult}_{P}\big(R\big)\mathrm{mult}_{P}\big(C\big)\geqslant 42,%
$$
where $R$ is a~general surface in $\mathcal{R}$. Therefore
$\mathrm{deg}(C)=21$ and $\mathrm{mult}_{P}(C)=3$ for every
$P\in\Sigma_{8}$. Note that $C\not\subset F$, since $Z\not\subset
F$. Thus
$$
\mathrm{mult}_{P}\big(C\big)\big|F\cap C\big|\leqslant\mathrm{mult}_{P}\Big(F\cdot C\Big)\big|F\cap C\big|=F\cdot C=\mathrm{deg}\big(C\big)=21,%
$$
for every point $P\in Z$. So the~curve $C$ is smooth at every
point of the~set $F\cap C$ by Theorem~\ref{theorem:v22-orbits},
which immediately implies that $C$ is reducible by
Lemma~\ref{lemma:long-orbit}.

Put $C=\sum_{i=1}^{r}C_{i}$, where $C_{i}$ is an irreducible curve
and $r\in\N$. Then \mbox{$\mathrm{deg}(C_{1})\in\{1,3\}$}. But
$\mathrm{deg}(C_{1})+1\geqslant|C_{1}\cap Z|$, because the~points
of $Z$ impose independent linear conditions on surfaces in
$|-K_{X}|$. We have
$$
\frac{21}{\mathrm{deg}\big(C_{1}\big)}=r\geqslant \frac{14\cdot 3}{\big|C_{1}\cap Z\big|},%
$$
because $\mathrm{mult}_{P}(C)\geqslant 3$ for every point $P\in
Z$. Thus $\mathrm{deg}(C_{1})=1$, $r=21$ and $|C_{1}\cap Z|=2$.
Since irreducible components of the~curve $C$ are lines and
$\mathrm{mult}_{P}(C)=3$ for every point~$P\in Z$, we can easily
see that each connected component of the~curve $C$ must have at
least $4$ components, so that $C$ has at most $6$ connected
components. Then $C$ is connected by
Corollary~\ref{corollary:PSL-permutation}.

Let $P$ be a~point in $Z$, let $G_{P}$ be its stabilizer subgroup
in $G$. Then
\mbox{$G_{P}\cong\mathrm{A}_{4}$} by~Lemma~\ref{lemma:PSL-maximal-subgroups},
and there are exactly $3$ irreducible components of the~curve $C$
containing $P$, since $\mathrm{mult}_{P}(C)=3$. Without loss of
generality, we may assume that $P=C_{1}\cap C_{2}\cap C_{3}$.
The~group $G_{P}$ naturally acts on the~set
$\{C_{1},C_{2},C_{3}\}$, which implies that there is a~subgroup
$G_{P}^{\prime}\subset G_{P}\cong\mathrm{A}_{4}$ such that
$G_{P}^{\prime}$ acts trivially on $\{C_{1},C_{2},C_{3}\}$ and
$G_{P}^{\prime}\cong\mathbb{Z}_{2}\times \mathbb{Z}_{2}$. Note
that $C_{1}$ is $G_{P}^{\prime}$-invariant. Let $\bar{P}$ be
the~point in $Z\cap C_{1}$ such that $\bar{P}\ne P$, let
$G_{\bar{P}}$ be its stabilizer subgroup in $G$. Then
$$
\mathbb{Z}_{2}\times\mathbb{Z}_{2}\cong G_{P}^{\prime}\subset G_{\bar{P}}\cong\mathrm{A}_{4},%
$$
because $Z\cap C_{1}=\{P,\bar{P}\}$. But the~group $G_{\bar{P}}$
contains a~unique subgroup that is isomorphic~to~$G_{P}^{\prime}$,
which very easily (almost immediately) implies that every point of
the~set $Z$ is $G_{P}^{\prime}$-invariant, which is impossible,
since the~group $G$ acts faithfully on $Z$, because the~group $G$
is simple.
\end{proof}

By Theorem~\ref{theorem:Kawamata}, the~curve $S$ is a~smooth curve
of genus $g$ such that \mbox{$\mathrm{deg}(S)\geqslant 2g-1$}. By
Remark~\ref{remark:centers}, the~curve $Z$ is a~disjoint union of
smooth curves isomorphic to $S$. Let $r$ be the~number of
connected components of the~curve $Z$. Put $d=\mathrm{deg}(S)$.
Then
\begin{equation}
\label{equation:v22-equality-main-special}
r\big(d-g+1\big)=14-h^{0}\Big(\mathcal{O}_{X}\big(H\big)\otimes\mathcal{I}\big(X,D\big)\Big)-h^{0}\big(\mathcal{O}_{\mathcal{L}}\big)%
\end{equation}
by $(\ref{equation:v22-equality-main})$. Note that
$h^{0}(\mathcal{O}_{\mathcal{L}})=0$ if and only if $\mathcal{L}(X,D)=Z$.
Finally, put \mbox{$q=h^{0}(\mathcal{O}_{X}(H)\otimes\mathcal{I}(X,D))$}.

\begin{lemma}
\label{lemma:v22-no-8-points}
One has $\mathcal{L}(X,D)=Z$.
\end{lemma}
\begin{proof}
Suppose that $Z\subsetneq\mathcal{L}(X,D)$. Then
$h^{0}(\mathcal{O}_{\mathcal{L}})=8$. It follows from
$(\ref{equation:v22-equality-main-special})$ that
$r(d-g+1)=6-q\leqslant 6$, since
$h^{0}(\mathcal{O}_{\mathcal{L}})\ne 0$. But $d-g+1\geqslant 1$,
which implies that $r=1$. Therefore $g\leqslant d-g+1=6-q\leqslant
6$, which implies that $g=3$ by
Lemma~\ref{lemma:sporadic-genera}. If $S\not\subset F$, then
$|F\cap S|\leqslant 21$ by Lemma~\ref{lemma:v22-curves}, which
contradicts Lemma~\ref{lemma:long-orbit}. We see that $S\subset
F$. But $q\leqslant 6$, so that $q=1$ by
Lemmas~\ref{theorem:v22-U-14}. Thus~\mbox{$d=7$}. There is
a~natural faithful action of the~group $G$ on the~curve $S$ such
that every $G$-invariant divisor on  $S$ has even degree by
Theorem~\ref{theorem:Dolgachev}. Hence $d$ is even, which is
a~contradiction.
\end{proof}

In particular, we see that the~cases $(\mathbf{B})$ and
$(\mathbf{C})$ are impossible.

\begin{lemma}
\label{lemma:7-lines-or-points} Suppose that $\mathrm{deg}(S)=1$.
Then $r\ne 7$.
\end{lemma}

\begin{proof}
Suppose that $r=7$. Then $q=0$ by
$(\ref{equation:v22-equality-main-special})$. In particular, we
have $S\not\subset F$. Then $|F\cap Z|\leqslant F\cdot Z=7$, which
contradicts Theorem~\ref{theorem:v22-orbits}, because $F\cap Z$ is
\mbox{$G$-invariant}.
\end{proof}

\begin{lemma}
\label{lemma:v22-reducible-curves} The~equality $r=1$ holds.
\end{lemma}
\begin{proof}
Suppose that $r\geqslant 2$. Then $r\geqslant 7$ by
Corollary~\ref{corollary:PSL-permutation}. If $r\geqslant 8$, then
$d-g+1=1$ by $(\ref{equation:v22-equality-main-special})$, which
implies that $g=d\geqslant 2g-1$, which leads to a~contradiction.

We see that $r=7$, so that $1\leqslant d-g+1\leqslant 2$ by
$(\ref{equation:v22-equality-main-special})$. Hence $d-g+1=2$,
since the~equality $d-g+1=1$ leads to a~contradiction. Therefore,
we have $g=d-1\geqslant 2g-2$, which gives $g\leqslant 2$ and
$d\leqslant 3$. Therefore $g=0$ and $d=1$, which is impossible by
Lemma~\ref{lemma:7-lines-or-points}.
\end{proof}

Thus, there is a~natural faithful action of the~group $G$ on
the~curve $S$. We have $d=13+g-q\geqslant 2g-1$, which implies
that $q\leqslant 14-g$. Note that $g\leqslant 14-q\leqslant 14$.
Then $g\in\{3,8,10\}$ by Lemma~\ref{lemma:sporadic-genera}.

\begin{lemma}
\label{lemma:S-H} The~curve $S$ is contained in the~surface $F$.
\end{lemma}

\begin{proof}
If $S\not\subset F$, then $|F\cap S|\leqslant 21$ by
Lemma~\ref{lemma:v22-curves}, which is impossible by
Lemma~\ref{lemma:long-orbit}.
\end{proof}

Applying Lemmas~\ref{theorem:v22-U-14}
 and \ref{lemma:S-H}, we see that $q\in\{1,7,8\}$.

\begin{lemma}
\label{lemma:g-3} The~equality $g=3$ holds.
\end{lemma}

\begin{proof}
Suppose that $g\ne 3$. Then $g\in\{8,10\}$. It follows from
$(\ref{equation:v22-equality-main-special})$ that \mbox{$q\leqslant
14-g\leqslant 6$}, which implies that $q=1$. Then $d\in\{20,22\}$,
which is impossible by Lemmas~\ref{lemma:g-8-d-7} and
\ref{lemma:g-10-d-3}.
\end{proof}

Thus, we have $d=16-q$, where $q\in\{1,7,8\}$.

\begin{lemma}
\label{lemma:g-3-d-8} The~equality $d=8$ holds.
\end{lemma}

\begin{proof}
By Theorem~\ref{theorem:Dolgachev}, there is a~$G$-invariant line
bundle $\theta\in\mathrm{Pic}(S)$ of degree~$2$ such that
$\mathrm{Pic}^{G}(S)$ is generated by $\theta$. In particular, we
see that $d$ is even. But  $d=16-q$ and $q\in\{1,7,8\}$, so that
$d=8$.
\end{proof}

Let $Q$ be a~surface in the~threefold $X$ that is swept out by
lines. Then $Q\sim -2K_{X}$ and the~surface $Q$ is irreducible by
Lemma~\ref{lemma:V22-Hilbert-scheme-of-lines}. Note that $Q$ is
$G$-invariant.

\begin{lemma}
\label{lemma:g-3-d-8-F-Q} The curve $S$ is contained in $Q\cap F$.
\end{lemma}

\begin{proof}
By~Lemma~\ref{lemma:S-H}, we have $S\subset F$. If $S\not\subset
Q$, then $|Q\cap S|\leqslant 2\mathrm{deg}(S)=16$ which is
impossible by Lemma~\ref{lemma:long-orbit}. Thus, we see that
$S\subset Q\cap F$.
\end{proof}

The surface $F$ is a~smooth $K3$ surface by
Lemma~\ref{lemma:v22-K3}. Then $S\cdot S=0$ on the~surface $X$.
Put $Q\big\vert_{F}=mS+\Delta$, where $m\in\N$, and $\Delta$ is
curve such that $S\not\subset\mathrm{Supp}(\Delta)$. Then
\begin{equation}
\label{equation:16-4m}
16=2\mathrm{deg}\big(S\big)=\Big(mS+\Delta\Big)\cdot S=mS\cdot S+\Delta\cdot S=4m+\Delta\cdot S,%
\end{equation}
which implies that $|\Delta\cap S|\leqslant 12$. Hence, it follows
from Lemma~\ref{lemma:long-orbit} that \mbox{$\mathrm{Supp}(\Delta)\cap
S=\varnothing$}, since $\Delta$ is $G$-invariant. Then
$\Delta=\varnothing$, because $\mathrm{Supp}\big(\Delta\big)\cup
S$ is connected (see \cite[Corollary~7.9]{Har77}). Now it follows
from $(\ref{equation:16-4m})$ that $m=4$, which immediately leads
to a~contradiction, since $44=Q\cdot F\cdot H=4 H\cdot S=32$,
where $H$ is a~general surface in $|-K_{X}|$. The~obtained
contradiction completes the~proof of Theorem~\ref{theorem:main}.

\medskip

\begin{proof}[{Proof of Theorem~\ref{theorem:V22-auxiliary}}]
Suppose that the pair $(X, R)$ is not log canonical.
Put
$$
\mu=\mathrm{sup}\left\{\epsilon\in\mathbb{Q}\ \left|%
\aligned
&\text{the~log pair}\ \Big(X, \epsilon R\Big)\
\text{is log canonical}\\
\endaligned\right.\right\}<1.
$$
Let $S$ be a~minimal center of log canonical singularities of the
log~pair $(X, \mu R)$ (see \cite{Kaw98}, \cite{ChSh09a}), and let
$Z$ be the~$G$-orbit of the~subvariety $S$. Then
$\mathrm{dim}(S)\leqslant 1$ since
$-K_X$ generates the group $\mathrm{Pic}(X)$.

Take $\epsilon\in\mathbb{Q}$ such that $1>\epsilon\gg 0$. By
Lemma~\ref{lemma:Kawamata-Shokurov-trick}, there is
a~$G$-in\-va\-riant $\mathbb{Q}$-divisor~$D$ such~that
$D\sim_{\mathbb{Q}} \epsilon R$, the~singularities of the~log pair
$(X,D)$ are log canonical, and every minimal center of log
canonical singularities of the~log pair $(X,D)$ is an irreducible
component of the~subvariety $Z$.

Let $\mathcal{I}(X,D)$ be the~multiplier ideal sheaf of the~log
pair $(X,D)$. Then the~sequence
$$
0\to H^{0}\Big(\mathcal{O}_{X}\otimes\mathcal{I}\big(X,D\big)\Big)\to
H^{0}\big(\mathcal{O}_{X}\big)\to H^{0}\big(\mathcal{O}_{Z}\big)\to 0%
$$
is exact by Theorem~\ref{theorem:Shokurov-vanishing}. In
particular, we see that $Z$ is connected. By
Lemma~\ref{lemma:centers}, the subvariety $Z$ is irreducible.
Hence, we must have $Z=S$.

The variety $X$ does not contain $G$-invariant points
by Theorem~\ref{theorem:v22-orbits}.
Thus, we see that $S$ must be a~curve. Then $S$ is a~smooth and
rational~curve Theorem~\ref{theorem:Kawamata}, which is
impossible, because the~group $G$ cannot act non-trivially
on~$\mathbb{P}^{1}$.
\end{proof}

\appendix

\section{Prime Fano threefolds of degree $22$}
\label{section:mukai}

In this section we describe Mukai's constructions of prime Fano
threefolds of  degree~$22$. Let $F(x,y,z)$ be a~quartic form, let
$C$ be a~curve in $\mathbb{P}^{2}$ defined by
$$
F\big(x,y,z\big)=0\subset\mathbb{P}^{2}\cong\mathrm{Proj}\Big(\mathbb{C}\big[x,y,z\big]\Big),
$$
and let $L_{1},L_{2},L_{3},L_{4},L_{5},L_{6}$ be six different
lines in $\mathbb{P}^{2}$.

\begin{definition}
\label{definition:n-gon} We say that $\sum_{i=1}^{6}L_{i}$ is
a~hexagon in $\mathbb{P}^{2}$.
\end{definition}

Let $l_{i}(x,y,z)$ is a~linear form such that the~line $L_{i}$ is
given by $l_{i}(x,y,z)=0$.

\begin{definition}[{\cite[Definition~4.1]{DoKa93}}]
\label{definition:polar} The~hexagon $\sum_{i=1}^{6}L_{i}$ is
polar to the~curve $C$ if $F(x,y,z)=\sum_{i=1}^{6}l_{i}^4(x,y,z)$.
\end{definition}

Consider $\sum_{i=1}^{6}L_{i}$ as an~element of the~Hilbert
scheme of points in the dual plane~$\check{\mathbb{P}}^{2}$.

\begin{definition}[{\cite{Mu89},
\cite{Mu92},~\cite{Sch01}}] \label{definition:VSP} The~variety of
polar hexagons to the~curve $C$ is
$$
\mathrm{VSP}\big(C,6\big)=\overline{\Bigg\{ \Gamma \in \mathrm{Hilb}_{6}\Big(\check{\mathbb{P}}^2\Big)\ \Big|\ \Gamma\ \text{is polar to the~curve}\ C \Bigg\}}\subset\mathrm{Hilb}_{6}\Big(\check{\mathbb{P}}^2\Big).%
$$
\end{definition}

Put $X=\mathrm{VSP}(C,6)$. 
Note that $X$ is a smooth Fano threefold of anticanonical
degree~$22$ provided that $C$ is general (see~\cite[Theorem~5]{Mu89}, 
\cite[Theorem~11]{Mu92}). Below we will assume that~$X$ is smooth. 

Put \mbox{$W=\mathrm{Spec}(\mathbb{C}[x,y,z])\cong\mathbb{C}^{3}$}. Then
$F(x,y,z)\in\mathrm{Sym}^4(W^{\vee})$ and the~partial derivatives
of the~form $F(x,y,z)$ give an~embedding $\phi\colon
W\to\mathrm{Sym}^3(W^\vee)$.

Suppose that $C$ is not degenerate (see
\cite[Definition~2.8]{DoKa93}). Then every hexagon $\Gamma\in X$
defines a~six-dimensional subspace
$W_{\Gamma}\subset\mathrm{Sym}^3(W^{\vee})$ such that
\mbox{$\phi(W)\subset W_{\Gamma}$}. This gives a~rational map
$X\dasharrow \mathrm{Gr}(3, U_7)$, where
$U_7\cong\mathrm{Sym}^3(W^{\vee})/W$. The constructed rational map
$X\dasharrow \mathrm{Gr}(3, U_7)$ can be extended to an~embedding
\mbox{$X\hookrightarrow\mathrm{Gr}(3, U_7)$}. This gives an~embedding
$X\hookrightarrow\mathbb{P}(\Lambda^3(U_7))$.

There is a~natural sequence of maps
\begin{multline*}
\Lambda^2(W)\otimes
\mathrm{Sym}^4\big(W\big)\overset{\alpha}\longrightarrow
\Big(W\otimes W\Big)\otimes \Big(\mathrm{Sym}^2\big(W\big)\otimes
\mathrm{Sym}^2\big(W\big)\Big)\overset{\beta}\longrightarrow{}\\
{}\overset{\beta}\longrightarrow \mathrm{Sym}^3\big(W\big)\otimes
\mathrm{Sym}^3\big(W\big)\overset{\gamma}
\longrightarrow\Lambda^2\Big(\mathrm{Sym}^3\big(W\big)\Big),%
\end{multline*}
and it follows from \cite[Section~2.3]{DoKa93} that the~quartic
form $F(x,y,z)$ defines a~natural map $\delta_F\colon
\mathrm{Sym}^2(W)\to\mathrm{Sym}^2(W^{\vee})$. Since the~quartic
$C$ is not degenerate, the~map $\delta_F$ is invertible.
Therefore, there is a~natural choice of a~non-zero element
\mbox{$\xi{\delta_F^{-1}}\in\mathrm{Sym}^4(W)$} via the~natural map
$$\mathrm{Hom}\Big(\mathrm{Sym}^2\big(W^\vee\big),\mathrm{Sym}^2\big(W\big)\Big)\cong\mathrm{Sym}^2\big(W\big)\otimes\mathrm{Sym}^2\big(W\big)\overset{\xi}\longrightarrow\mathrm{Sym}^4\big(W\big),%
$$
which implies that the~composition $\gamma\circ\beta\circ\alpha$
gives a~map
$\zeta\colon\Lambda^2(W)\to\Lambda^2(\mathrm{Sym}^3(W))$, where
$\Lambda^2(W)\cong W^{\vee}$.

\begin{lemma}\label{lemma:spusk}
Let $\omega$ be an~element in $\mathrm{im}(\zeta)$ considered as
a~skew form on $\mathrm{Sym}^3(W^{\vee})$, and let
$\Pi\subset\mathrm{Sym}^3(W^\vee)$ be the~kernel of the~form
$\omega$. Then $\mathrm{im}(\phi)\subset\Pi$.
\end{lemma}
\begin{proof}
This is a straightforward computation.
\end{proof}

By Lemma~\ref{lemma:spusk},  the~map $\zeta$ gives us a~map
$W^{\vee}\to\Lambda^2(U_7^{\vee})$. Therefore, one has
$$
W^{\vee}\otimes U_7^{\vee}\overset{\sigma}\to\Lambda^2(U_7^{\vee})\otimes U_7^{\vee}\overset{\upsilon}\to \Lambda^3(U_7^{\vee})%
$$
so that $\upsilon\circ\sigma$ is a~monomorphism. Put
$U_{14}=\Lambda^3(U_7)\slash(W^{\vee}\otimes
U_7^{\vee})\cong\mathbb{C}^{14}$ and consider
$\mathrm{im}(\upsilon\circ\sigma)$ as a~$21$-dimen\-sional linear
system of hyperplanes in $\P(\Lambda^3(U_7))$ vanishing~on
the~image of the~threefold $X$. This gives us a~natural embedding
$X\hookrightarrow\mathbb{P}(U_{14})$.

\begin{theorem}[{cf.~\cite{Mu89}, \cite{Mu92}, \cite{Sch01}}]
\label{theorem:U14} The~embedding
$X\hookrightarrow\mathbb{P}(U_{14})$ is the~anticanonical
embedding.
\end{theorem}

\section{Representation theory}
\label{section:characters}

In this section we collect some facts about the~groups $\PSLF$ and
$\SLF$.

\begin{lemma}[\cite{Atlas}]\label{lemma:PSL-maximal-subgroups}
Let $\Gamma$ be a~maximal subgroup in $\PSLF$. Then
\begin{itemize}
\item either $\Gamma\cong\mathbb{Z}_7\rtimes\mathbb{Z}_{3}$ 
and $\Gamma$ is unique up to conjugation,%
\item or $\Gamma\cong\mathrm{S}_4$ and $\PSLF$ contains 
two subgroups isomorphic to $\Gamma$ up to conjugation.%
\end{itemize}
\end{lemma}

\begin{corollary}\label{corollary:PSL-permutation}
If $\PSLF$ acts transitively on a~finite set $\Sigma$ such that
$|\Sigma|\leqslant 41$, then $|\Sigma|\in\{1,7,8,14,21,24,28\}$.
\end{corollary}

The~group $\PSLF$ has exactly six non-isomorphic irreducible
representations~(see~\cite{Atlas}), which we denote by $I$,
$W_{3}$, $W_3^{\vee}$, $W_6$, $W_7$, $W_8$. The~values of their
characters are listed~in the~table:

\begin{center}\renewcommand\arraystretch{1.1}
\begin{tabular}{|c|c|c|c|c|c|c|}%\label{table:psl27}
\hline
& $\mathrm{id}$ & ($2$) & ($3$) & ($4$) & ($7$) & ($7'$)\\
\hline
$\sharp$ & $1$ & $21$ & $56$ & $42$ & $24$ & $24$\\
\hline
$I$ & $1$ & $1$ & $1$ & $1$ & $1$ & $1$\\
\hline
$W_3$ & $3$ & $-1$ & $0$ & $1$ & $\epsilon$ & $\bar{\epsilon}$\\
\hline
$W_3^{\vee}$ & $3$ & $-1$ & $0$ & $1$ & $\bar{\epsilon}$ & $\epsilon$\\
\hline
$W_6$ & $6$ & $2$ & $0$ & $0$ & $-1$ & $-1$\\
\hline
$W_7$ & $7$ & $-1$ & $1$ & $-1$ & $0$ & $0$\\
\hline
$W_8$ & $8$ & $0$ & $-1$ & $0$ & $1$ & $1$\\
\hline
\end{tabular}
\end{center}
We use the~following notation. The~first row represents
the~conjugacy classes in~$\PSLF$: the~symbol $\mathrm{id}$ denotes
the~identity element,  the~symbol $(n)$ denotes a~class of
elements of order~$n$, the~symbols $(7)$ and $(7')$ denote two
different conjugacy classes of elements of order $7$; note that if
$g\in (7)$, then $g^2\in (7)$ and $g^4\in (7)$, while $g^3\in
(7')$, $g^5\in (7')$ and $g^6\in (7')$. The~second row lists
the~number of elements in each conjugacy class. The~next six rows
list the~values of the~characters of irreducible representations.
By $\epsilon$ we denote the~complex number $-1/2+\sqrt{-7}/2$, and
by $\bar{\epsilon}$ its complex conjugate.

Looking at the~above table, one easily obtains the~following
corollaries.

\begin{corollary}\label{corollary:A4-representation-dim-6}
Let $\Gamma$ be a~subgroup of the~group $\PSLF$ such
that $\Gamma\cong\mathrm{A}_{4}$. Then $W_3$ is an~irreducible $\Gamma$-representation, and
 $W_6$ is a~sum of two irreducible three-dimensional $\Gamma$-representations.
\end{corollary}

\begin{corollary}\label{corollary:characters}
The~following isomorphisms of the~representations of the~group
$\PSLF$ hold:
\begin{gather*}
\mathrm{Sym}^2(W_3)\cong W_6\cong\mathrm{Sym}^2(W_3^{\vee}),\ \Lambda^2(W_3^{\vee})\cong W_3,\ \mathrm{Sym}^3(W_3^{\vee})\cong W_7\oplus W_3,\\
\Lambda^4(W_7)\cong\Lambda^3(W_7)^{\vee}\cong\Lambda^3(W_7),\ W_7\otimes W_3^{\vee}\cong W_6\oplus W_7\oplus W_8,\\
\Lambda^3\big(W_7\big)\cong I\oplus W_6\oplus W_6\oplus W_7\oplus W_7\oplus W_8.%
\end{gather*}
\end{corollary}

Let $\hat{G}$ be a subgroup in $\mathrm{SL}_{3}(\mathbb{C})$ that
is isomorphic to $\PSLF$, and let
$\phi\colon\mathrm{SL}_{3}(\mathbb{C})\to\mathrm{Aut}(\mathbb{P}^{2})$
be a~natural projection.  Put $G=\phi(\hat{G})$. Then
$G\cong\hat{G}$.

\begin{lemma}[{\cite[Section~2.10]{YauYu93}}]
\label{lemma:Klein-small-invariants} There are no $G$-invariant
curves in $\mathbb{P}^{2}$ of degrees $1$, $2$, $3$, and $5$.
There is unique $G$-invariant curve in $\P^2$ of degree $4$, which
is isomorphic to the~quartic curve described in
Example~\ref{example:V22}. There is unique $G$-invariant curve in
$\P^2$ of degree $6$, which is isomorphic to the~Hessian curve of
the~quartic curve described in Example~\ref{example:V22}.
\end{lemma}

\begin{lemma}
\label{lemma:Klein-small-orbits} There are no $G$-invariant
subsets in $\P^2$ consisting of at most $20$ points.
\end{lemma}

\begin{proof}
Restricting $W_3$ to subgroups of the~group $\PSLF$, we obtain
the~required assertion.
\end{proof}

Let $\Gamma$ be a~subgroup in~$\SLF$,  let $\pi\colon\SLF\to\PSLF$
be a~natural epimorphism.~Then $\Gamma\cong 2.\pi(\Gamma)$ if
$\pi(\Gamma)$ is isomorphic to $\PSLF$, $\mathrm{S}_{4}$,
$\mathrm{A}_{4}$, $\mathbb{Z}_{7}\rtimes\mathbb{Z}_{3}$,
$\mathrm{D}_{4}$, or~$\mathrm{S}_{3}$.

The~following table contains the~character values of one
four-dimensional and one eight-di\-men\-si\-onal irreducible
re\-pre\-sen\-ta\-tions of the~group $\mathrm{SL}_{2}(7)$ and some
information about its subgroups:
\begin{center}\renewcommand\arraystretch{1.1}
\begin{tabular}{|c|c|c|c|c|c|c|c|c|c|c|c|}
\hline & $\mathrm{id}$ & $-\mathrm{id}$ & $(3)_3$ & $(3)_6$ &
$(7)_7$ & $(7)_7$ & $(7)_{14}$ & $(7)_{14}$ & $(2)_4$ & $(4)_8$ & $(4)_8$\\
\hline ${G}$ & $1$ & $1$ & $56$ & $56$ & $24$ & $24$ & $24$ & $24$ & $42$ & $42$ & $42$\\
\hline ${2.\mathrm{S}_4}$ & $1$ & $1$ & $8$ & $8$ & $0$ & $0$ & $0$ & $0$ & $18$ & $6$ & $6$\\
\hline ${2.\mathrm{A}_4}$ & $1$ & $1$ & $8$ & $8$ & $0$ & $0$ & $0$ & $0$ & $6$ & $0$ & $0$\\
\hline ${2.(\mathbb{Z}_{7}\rtimes\mathbb{Z}_{3})}$ & $1$ & $1$ & $14$ & $14$ & $3$ & $3$ & $3$ & $3$ & $0$ & $0$ & $0$\\
\hline ${2.D_{4}}$ & $1$ & $1$ & $0$ & $0$ & $0$ & $0$ & $0$ & $0$ & $10$ & $2$ & $2$\\
\hline ${2.\mathrm{S}_3}$ & $1$ & $1$ & $2$ & $2$ & $0$ & $0$ & $0$ & $0$ & $6$ & $0$ & $0$\\
\hline \hline $U_4$ & $4$ & $-4$ & $1$ & $-1$ & $\bar{\alpha}$ & $\alpha$ & $-\alpha$ & $-\bar{\alpha}$ & $0$ & $0$ & $0$\\
\hline
$U_8$ & $8$ & $-8$ & $-1$ & $1$ & $1$ & $1$ & $-1$ & $-1$ & $0$ & $0$ & $0$\\
\hline
\end{tabular}
\end{center}
We use the following notation. The~first row represents
the~conjugacy classes in $\SLF$: the~symbol $\mathrm{id}$  denotes
the~identity element, the~symbol $-\mathrm{id}$  denotes
the~element different from $\mathrm{id}$ such that
$-\mathrm{id}\in\mathrm{ker}(\pi)$, the~symbol $(n)_k$ denotes a
conjugacy class that consists of elements of order $k$ such that
their images in $\PSLF$ have order~$n$. The~next six rows list
the~number of elements in the corresponding conjugacy classes in
some subgroups of the~group $\SLF$. The~last two rows list
the~values of the~characters of two irreducible representations.
The~symbol $\alpha$ denotes the~complex number
$-(\zeta^3+\zeta^5+\zeta^6)$, where $\zeta$ is a primitive seventh
root of unity, and $\bar{\alpha}$ denotes the~complex conjugate of
$\alpha$.

\begin{lemma}
\label{lemma:F21} Suppose that $\Gamma\cong
2.(\mathbb{Z}_{7}\rtimes\mathbb{Z}_{3})$. Then $U_4\cong T\oplus
J$ and
$$
U_8\cong T\oplus T_{1}\oplus J_{1}\oplus J_{2}
$$
as representations of the~group $\Gamma$, where $J$, $J_{1}$ and
$J_{1}$ are pairwise non-isomorphic one-dimen\-si\-o\-nal
representations, while $T$ and $T_{1}$ are irreducible three-dimensional
representations.
\end{lemma}

\begin{proof}
Let $\chi_4$ and $\chi_8$ be the~characters of the~representations
$U_4$ and $U_8$, respectively. Then $\langle \chi_4,
\chi_4\rangle=2$, which immediately implies that $U_4\cong J\oplus
T$ for some one-dimensional representations $J$ and some
irreducible three-dimensional representation~$T$ of the~group
$\Gamma$, because irreducible representations of the~group
$\Gamma$ are either one-dimensional or three-dimensional.
Similarly,~we~have $\langle \chi_8, \chi_8\rangle=4$, which
implies that $U_8\cong J_{1}\oplus J_{2}\oplus T_{1}\oplus T_{2}$
for some one-dimensional representations $J_{1}\not\cong J_{2}$
and some irreducible three-dimensional representations
$T_{1}\not\cong T_{2}$ of the~group $\Gamma$.

We may assume that $T_{2}\cong T$, because there exist exactly two
three-dimensional representations of the~group $\Gamma$ with
a~non-trivial action of its center. But \mbox{$\langle \chi_4,
\chi_8\rangle=1$}, which implies that that neither $J_{1}$ nor
$J_{2}$ is isomorphic to $J$.
\end{proof}

Note that we can consider $\PSLF$-representations as
$\SLF$-representations.

\begin{lemma}
\label{lemma:some-SL-representations} One has
$\mathrm{Sym}^4(U_4)\cong I\oplus W_6\oplus W_6\oplus W_7\oplus
W_7\oplus W_8$ as representations of the~group $\PSLF$ or
the~group $\SLF$.
\end{lemma}

\begin{proof}
This follows from elementary and explicit computations (see
also~\cite[Appendix~1]{Do99}).
\end{proof}

\begin{lemma}
\label{lemma:SL-2-7-subgroups} As a~representation of the~group
$\Gamma$, the~representation $U_4$ splits as a~sum of two
irreducible two-dimensional  subrepresentations  if $\Gamma\cong
2.\mathrm{S}_4$, a~sum of two irreducible two-dimensional
subrepresentations  if $\Gamma\cong 2.\mathrm{A}_4$, a~sum of two
irreducible two-dimensional subrepresentations if $\Gamma\cong
2.\mathrm{D}_4$, a~sum of an irreducible two-dimensional and two
non-isomorphic one-dimensional subrepresentations if~\mbox{$\Gamma\cong
2.\mathrm{S}_3$}.
\end{lemma}

\begin{proof}
Let $\chi_4$ be the~character of the~representations $U_4$. If
$\Gamma\cong 2.\mathrm{D}_4$, then $\langle \chi_4,
\chi_4\rangle=2$, which easily implies that $U_4$ splits as a sum
of two irreducible two-dimensional subrepresentations, because
$2.\mathrm{D}_4$ has no odd-dimensional non-trivial irreducible
representations.

Since all irreducible representations of the~group
$2.\mathrm{A}_4$ with a non-trivial action of its~center are
two-dimensional, the~representation $U_4$ splits as a sum of two
irreducible two-dimensional subrepresentations of the~group
$\Gamma$ if $\Gamma\cong 2.\mathrm{A}_4$ or $\Gamma\cong
2.\mathrm{S}_4$.

To complete the~proof, we may assume that $\Gamma\cong
2.\mathrm{S}_3$. Then there is an epimorphism
$\Gamma\to\mathbb{Z}_4$. Let $U$ be the~standard unitary
two-dimensional irreducible representation of the~group $\Gamma$,
let $J$ and $J_{1}$ be one-dimensional representations of
the~group $\Gamma$ that arise from the~faithful non-isomorphic
one-dimensional representations of the~group $\mathbb{Z}_4$. Then
$U_4\cong U\oplus J\oplus J_{1}$ as can be seen from the~following
table that lists the~character values of these~representations:
\begin{center}\renewcommand\arraystretch{1.1}
\begin{tabular}{|c|c|c|c|c|c|}
\hline
& $\mathrm{id}$ & $-\mathrm{id}$ & $(3)_3$ & $(3)_6$ & $(2)_4$\\
\hline
${2.\mathrm{S}_3}$ & $1$ & $1$ & $2$ & $2$ & $6$\\
\hline
\hline
$U_4$ & $4$ & $-4$ & $1$ & $-1$ & $0$\\
\hline
$U$ & $2$ & $-2$ & $-1$ & $1$ & $0$\\
\hline
$J$ & $4$ & $-4$ & $1$ & $-1$ & $\sqrt{-1}$\\
\hline
$J_1$ & $4$ & $-4$ & $1$ & $-1$ & $-\sqrt{-1}$\\
\hline
\end{tabular}
\end{center}
where we used notation similar to the~ones used in the~table
above.
\end{proof}

\end{document}